\documentclass[11pt]{amsart}
    \usepackage{amsmath,amssymb,amsthm,amsfonts,amsrefs}
    \usepackage{geometry}
    \usepackage{graphicx}
    \usepackage[all]{xy}
    \usepackage[colorlinks, linkcolor=blue, anchorcolor=blue,
            citecolor=blue]{hyperref}

    \usepackage{enumerate}
    \usepackage{tikz}
    \usepackage{tikz-cd}
    \usepackage{xcolor,float}
    \usepackage{mathrsfs}
    \usepackage{caption}

    \usepackage{fancyhdr}
    \usepackage{mathrsfs}
    \newtheorem{definition}{Definition}[section]
    \newtheorem{theorem}{Theorem}[section]
    \newtheorem{proposition}[theorem]{Proposition}
    \newtheorem{lemma}[theorem]{Lemma}
    \newtheorem{corollary}[theorem]{Corollary}
    
    \newtheorem{example}{Example}[section]
    \newtheorem{remark}[example]{Remark}

\newcommand{\SF}{\mathscr{F}}
\newcommand{\SG}{\mathscr{G}}

\newcommand{\SL}{\mathscr{L}}

\newcommand{\bC}{\mathbb{C}}

\newcommand{\bR}{\mathbb{R}}
\newcommand{\bZ}{\mathbb{Z}}

\begin{document}

\title{Lagrangian Cobordism Functor in Microlocal Sheaf Theory II}
\date{}
\author{Wenyuan Li}
\address{Department of Mathematics, Northwestern University.}
\email{wenyuanli2023@u.northwestern.edu}
\maketitle

\begin{abstract}
    For an exact Lagrangian cobordism $L$ between Legendrians in $J^1(M)$ from $\Lambda_-$ to $\Lambda_+$ whose Legendrian lift is $\widetilde{L}$, we prove that sheaves in $Sh_{\widetilde{L}}(M \times \bR \times \bR_{>0})$ are equivalent to sheaves at the negative end $Sh_{\Lambda_-}(M \times \bR)$ together with the data of local systems $Loc({L})$ by studying sheaf quantizations for general noncompact Lagrangians. Thus we interpret the Lagrangian cobordism functor between $Sh_{\Lambda_\pm}(M \times \bR)$ as a correspondence parametrized by $Loc({L})$. This enables one to consider generalizations to immersed Lagrangian cobordisms. We also prove that the Lagrangian cobordism functor is action decreasing and recover results on the lengths of embedded Lagrangian cobordisms. Finally, using the construction of Courte-Ekholm, we obtain a family of Legendrians with sheaf categories Morita equivalent to chains on based loop spaces on the Lagrangian fillings.
\end{abstract}

\section{Introduction}

\subsection{Context of Lagrangian cobordism}
    Our goal in this paper is to study exact Lagrangian cobordisms between Legendrian submanifolds in 1-jet bundles from the perspective of microlocal sheaf theory, following the previous work \cite{LiCobordism}.

    In the framework of symplectic field theory \cite{SFT}, given Legendrian submanifolds $\Lambda_\pm \subset (J^1(M), \alpha_\text{std})$, an exact Lagrangian cobordism in the symplectization $L \subset (J^1(M) \times \mathbb{R}_{>0}, $ $d(s\alpha_\text{std}))$ should induce a morphism between the Chekanov-Eliashberg dg algebras \cite{Ekcobordism,EHK}. By enhancing the Chekanov-Eliashberg dg algebra with coefficients in the chains of based loop spaces \cite{EkholmLekili}, we expect a dg algebra morphism
    $$\Phi_L^*: \mathcal{A}_{C_{-*}(\Omega_*\Lambda_+)}(\Lambda_+) \rightarrow \mathcal{A}_{C_{-*}(\Omega_*\Lambda_-)}(\Lambda_-) \otimes_{C_{-*}(\Omega_*\Lambda_-)} C_{-*}(\Omega_*L),$$
    and a functor over the $\mathcal{A}_\infty$-categories of augmentations \cite{AugSheaf}, i.e.~an $\mathcal{A}_\infty$-category associated to 1-dimensional representations of the dg algebra (see \cite{PanTorus} for the case of Legendrian knots)\footnote{We use the convention that $\mathcal{A}ug_-(\Lambda)$ is the augmentation category for the dg algebra $\mathcal{A}(\Lambda)$ without loop space coefficients, while $\mathcal{A}ug_+(\Lambda)$ is the augmentation category for $\mathcal{A}_{C_{-*}(\Omega_*\Lambda)}(\Lambda)$ following \cite{EkholmLekili}. For Legendrian knots, objects in $\mathcal{A}ug_-(\Lambda)$ lift uniquely to $\mathcal{A}ug_+(\Lambda)$ \cite{Leverson}, so people don't distinguish the objects.}
    $$\Phi_L: \mathcal{A}ug_+(\Lambda_-) \times_{Loc^1(\Lambda_-)} Loc^1(L) \rightarrow \mathcal{A}ug_+(\Lambda_+).$$

    Combining the work of Ganatra-Pardon-Shende \cite{GPS3} and the Legendrian surgery formula \cite{EkholmLekili,EkSurgery,AsplundEkholm}, we know the Chekanov-Eliashberg dg algebra with loop space coefficients is equivalent to the compact objects of sheaves with singular support on the Legendrian
    $$Sh^c_\Lambda(M \times \mathbb{R}) \simeq \mathrm{Perf}\,\mathcal{W}(T^*(M \times \mathbb{R}), \Lambda)^\text{op} \simeq \mathrm{Perf}\,\mathcal{A}_{C_{-*}(\Omega_*\Lambda)}(\Lambda)^\text{op}.$$
    Respectively, we also know that augmentation category of the dg algebra are equivalent to microlocal rank~1 sheaves for Legendrian links $\Lambda \subset \mathbb{R}^3$ \cite{AugSheaf} (with generalizations to certain cases in higher dimensional Legendrians \cite{AugSheafknot,AugSheafsurface,AugSheafLink} and higher dimensional representations \cite{RepSheaf,SackelGraph})
    $$Sh^1_\Lambda(\mathbb{R}^2) \simeq \mathcal{A}ug_+(\Lambda).$$
    Hence we expect a corresponding construction in microlocal sheaf theory.

    Such a functor is constructed in our previous work \cite{LiCobordism} following Nadler-Shende \cite{NadShen}.

\begin{theorem}[\cite{LiCobordism} following \cite{NadShen}]\label{thm:cob-functor}
    Let $\Lambda_-, \Lambda_+ \subset J^1(M)$ be Legendrian submanifolds, and $L \subset J^1(M) \times \mathbb{R}$ an exact Lagrangian cobordism from $\Lambda_-$ to $\Lambda_+$. Then there is a cobordism functor between the sheaf categories of compact objects
    $$\Phi_L^*: \, Sh^c_{\Lambda_+}(M \times \mathbb{R}) \longrightarrow Sh^c_{\Lambda_-}(M \times \mathbb{R}) \otimes_{Loc^c(\Lambda_-)} Loc^c(L),$$
    and a fully faithful adjoint functor between the sheaf categories of proper objects
    $$\Phi_L: \, Sh^b_{\Lambda_-}(M \times \mathbb{R}) \times_{Loc^b(\Lambda_-)} Loc^b(L) \hookrightarrow Sh^b_{\Lambda_+}(M \times \mathbb{R})$$
    such that concatenations of Lagrangian cobordisms give rise to compositions of functors.
\end{theorem}

    However, there is another perspective of understanding the algebraic structures arising from exact Lagrangian cobordisms. An exact Lagrangian cobordism $L \subset J^1(M) \times \mathbb{R}_{>0}$ can be lifted to a Legendrian cobordism $\widetilde{L} \subset J^1(M \times \mathbb{R}_{>0})$ with conical ends. In $J^1(\mathbb{R})$ and $J^1(S^1)$, Pan-Rutherford showed that the dg algebra map can be viewed as a bimodule \cite{PanRuther}. By enhancing with loop space coefficients, we conjectured \cite{LiCobordism}*{Section 1.3.3} that one can define a dg algebra with loop space coefficients on the cobordism $\widetilde{L}$ coming from an embedded Lagrangian so that
    $$\mathcal{A}_{C_{-*}(\Omega_*\widetilde{L})}(\widetilde{L}) \xrightarrow{\sim} \mathcal{A}_{C_{-*}(\Omega_*\Lambda_-)}(\Lambda_-) \otimes_{C_{-*}(\Omega_*\Lambda_-)} C_{-*}(\Omega_*L),$$
    and then the enhanced dg algebra map can be viewed as correspondences parametrized by chains on the based loop space of $L$
    $$\mathcal{A}_{C_{-*}(\Omega_*\Lambda_-)}(\Lambda_-) \otimes_{C_{-*}(\Omega_*\Lambda_-)} C_{-*}(\Omega_*L) \xrightarrow{\sim} \mathcal{A}_{C_{-*}(\Omega_*\widetilde{L})}(\widetilde{L}) \leftarrow \mathcal{A}_{C_{-*}(\Omega_*\Lambda_+)}(\Lambda_+).$$
    Forgetting the data of $C_{-*}(\Omega_*L)$, we in particular recover the diagram of dg algebras with no loop space coefficients that Pan-Rutherford proved for Legendrian knots
    $$\mathcal{A}(\Lambda_-) \xrightarrow{\sim} \mathcal{A}(\widetilde{L}) \leftarrow \mathcal{A}(\Lambda_+).$$
    Similarly, in microlocal sheaf theory, one may also consider the restriction functors to both ends which define correspondences between the sheaf categories with singular supports on $\Lambda_\pm$ parametrized by local systems on $L$ \cite{LiCobordism}*{Section 1.3.3}. Our goal in this paper is to realize the Lagrangian cobordism functor as correspondences, which in particular will provide a framework for studying immersed Lagrangian cobordisms. 

    We remark that, though less general than the setting first paper \cite{LiCobordism}, the setting in this paper is more intuitive and straightforward from the geometric perspective, and is the one considered in recent works on Legendrian weaves \cite{CasalsZas,AlgWeave}.

\subsection{Sheaf quantization of Lagrangian cobordisms}
    Let $\Lambda \subset T^{*,\infty}_{\tau>0}(M \times \bR)$ be a Legendrian submanifold. We consider the category $Sh_\Lambda(M \times \bR)_0$ of sheaves with singular support in $\Lambda$ and acyclic stalks at $-\infty$, and $Sh^c_\Lambda(M \times \bR)_0$ and $Sh^b_\Lambda(M \times \bR)_0$ the subcategories of compact and proper objects.

\begin{remark}
    Throughout the paper, we assume that the Lagrangians $L$ and Legendrians $\Lambda_-, \Lambda_+$ are equipped with Maslov data that are compatible with respect to inclusions \cite[Section 10]{NadShen}. When $\Bbbk$ is a ring, the existence of Maslov data requires that the Maslov class of the Lagrangian $\mu(L) = 0$ and that of the Legendrians $\mu(\Lambda_\pm) = 0$. When $\mathrm{char}\,\Bbbk \neq 2$, we require in addition that $L$ and $\Lambda_\pm$ are relatively spin.
\end{remark}

    For any sheaf in $Sh_{\widetilde{L}}(M \times \mathbb{R} \times \mathbb{R}_{>0})_0$, by restrictions to both ends we can get sheaves in $Sh_{\Lambda_\pm}(M \times \mathbb{R})_0$ and the microlocalization in $Loc({L})$. However, in order to get a Lagrangian cobordism functor from the diagram of restrictions, it is necessary to define a functor that construct sheaves in $Sh_{\widetilde{L}}(M \times \mathbb{R} \times \mathbb{R}_{>0})_0$ from sheaves at the negative end $Sh_{\Lambda_-}(M \times \bR)_0$ and local systems on the cobordism $Loc(L)$. Our first theorem solves this problem.

\begin{theorem}\label{thm:cond-quan-cob-intro}
    Let $\Lambda_\pm \subset J^1(M)$ be closed Legendrian submanifolds, and $L \subset J^1(M) \times \mathbb{R}_{>0}$ an embedded exact Lagrangian cobordism from $\Lambda_-$ to $\Lambda_+$, which lifts to a Legendrian cobordism $\widetilde{L} \subset (J^1(M) \times \mathbb{R}_{>0}) \times \mathbb{R} \cong J^1(M \times \mathbb{R}_{>0})$. Write $i_-: M \times \mathbb{R} \times \{s_-\} \hookrightarrow M \times \mathbb{R} \times \mathbb{R}_{>0}$ for some small $s_->0$. Then there is an equivalence functor, which we call the conditional sheaf quantization functor,
    $$\Psi_L: \,Sh^b_{\Lambda_-}(M \times \mathbb{R})_0 \times_{Loc^b(\Lambda_-)} Loc^b(\widetilde{L}) \xrightarrow{\sim} Sh^b_{\widetilde{L}}(M \times \mathbb{R} \times \mathbb{R}_{>0})_0$$
    that is the inverse of the fiber product between the restriction functor $i_-^{-1}$ and microlocalization functor $m_{\widetilde{L}}$. Correspondingly, there is the left adjoint
    $$\Psi_L^*: \, Sh^c_{\widetilde{L}}(M \times \mathbb{R} \times \mathbb{R}_{>0})_0 \xrightarrow{\sim} Sh^c_{\Lambda_-}(M \times \mathbb{R})_0 \times_{Loc^c(\Lambda_-)} Loc^c(\widetilde{L}).$$
\end{theorem}

    Constructing sheaves from known Lagrangian subamnifolds (often called sheaf quantization) has been the core problem in the field, studied in a number of celebrated works \cite{Tamarkin1,GKS,Gui,JinTreu,AsanoIkeimmersion}. The work of Guillermou \cite{Gui} and Jin-Treumann \cite{JinTreu} have considered exact Lagrangians with Legendrian lifts $\widetilde{L} \subset J^1(M)$ that are closed or with Legendrian boundaries at positive infinity $T^{*,\infty}M$, and constructed fully faithful functors
    $$\Psi_L: Loc(L) \xrightarrow{\sim} Sh_{\widetilde{L}}(M \times \mathbb{R})_0$$
    that are inverses of taking microlocalization (also known as microlocal monodromy).

    However, for Lagrangian cobordisms between Legendrian submanifolds, we know for sure that there does not always exist a sheaf quantization which produces sheaves in $Sh_{\widetilde{L}}(M \times \mathbb{R} \times \mathbb{R}_{>0})$ (for example the trivial endocobordism of a stabilized or loose Legendrian). Thus what we prove should be viewed as a conditional sheaf quantization result in the spirit of Guillermou-Jin-Treumann, which explains that given a local system in $Loc(L)$, the necessary condition of existence of a sheaf quantization at the negative end $Sh_{\Lambda_-}(M \times \mathbb{R})$ is in fact also the sufficient condition.

\subsection{Lagrangian cobordism functor as correspondence}
    We are able to show that the Lagrangian cobordism functor in Theorem \ref{thm:cob-functor} is compatible with the restriction functors, which implies that the Lagrangian cobordism functor is indeed determined by the correspondence coming from the restriction functors to both conical ends.

\begin{theorem}\label{thm:compatible-cob-intro}
    Let $\Lambda_\pm \subset J^1(M)$ be closed Legendrian submanifolds, and $L \subset J^1(M) \times \mathbb{R}_{>0}$ an embedded exact Lagrangian cobordism from $\Lambda_-$ to $\Lambda_+$, which lifts to a Legendrian cobordism $\widetilde{L} \subset (J^1(M) \times \mathbb{R}_{>0}) \times \mathbb{R} \cong J^1(M \times \mathbb{R}_{>0})$.
    Write $i_\pm: M \times \mathbb{R} \times \{s_\pm\} \hookrightarrow M \times \mathbb{R} \times \mathbb{R}_{>0}$ for some small $s_- > 0$ and large $s_+ > 0$. Then there is a commutative diagram between sheaf categories of compact objects
    \[\xymatrix@R=6mm{
    Sh^c_{\Lambda_+}(M \times \mathbb{R}) \ar[rr]^{\Phi_L^*\hspace{40pt}} \ar[dr]_{\iota_{\widetilde{L}}^* \circ\, i_{+!}} & & Sh^c_{\Lambda_-}(M \times \mathbb{R}) \times_{Loc^c(\Lambda_-)} Loc^c(\widetilde{L}) \ar[dl]^{(\iota_{\widetilde{L}}^* \circ \, i_{-!}, m_{\widetilde{L}}^*)}, \\
    & Sh^c_{\widetilde{L}}(M \times \mathbb{R} \times \mathbb{R}_{>0}) &
    }\]
    where $m_{\widetilde{L}}^*$ is the left adjoint of the microlocalization, and $\iota_{\widetilde{L}}^* \circ i_{+!}$ is the left adjoint of $i_+^{-1}$. Correspondingly there is a commutative diagram between sheaf categories of proper objects
    \[\xymatrix@R=6mm{
    & Sh^b_{\widetilde{L}}(M \times \mathbb{R} \times \mathbb{R}_{>0}) \ar[dl]_{(i_-^{-1}, m_{\widetilde{L}})} \ar[dr]^{i_+^{-1}} & \\
    Sh^b_{\Lambda_-}(M \times \mathbb{R}) \times_{Loc^b(\Lambda_-)} Loc^b(\widetilde{L}) \ar[rr]_{\hspace{40pt}\Phi_L} & & Sh^b_{\Lambda_+}(M \times \mathbb{R}),
    }\]
    where $m_{\widetilde{L}}$ is the microlocalization. In particular, $i_+^{-1} \circ \Psi_L \simeq \Phi_L$ and $\Psi_L^* \circ \iota_{\widetilde{L}}^* \circ i_{+!} \simeq \Phi_L^*$.
\end{theorem}
\begin{remark}
    Though the functor $\Psi_L$ from the negative end to the cobordism is always an equivalence, the restriction $i_+^{-1}$ to the positive end is not. Hence we will see that the functor $\Phi_L$ from the negative end to the positive end is in general not an equivalence.
\end{remark}

    The above result indicates that we can in fact generalize the Lagrangian cobordism functor to certain immersed exact Lagrangian cobordisms whenever $(i_-^{-1}, m_{\widetilde{L}})$ is an essentially surjective functor so that we can choose a right inverse.

\begin{definition}
    Let $\Lambda_\pm \subset J^1(M)$ be closed Legendrian submanifolds, and $L \subset J^1(M) \times \mathbb{R}_{>0}$ an immersed exact Lagrangian cobordism from $\Lambda_-$ to $\Lambda_+$, which lifts to a Legendrian cobordism $\widetilde{L} \subset (J^1(M) \times \mathbb{R}_{>0}) \times \mathbb{R} \cong J^1(M \times \mathbb{R}_{>0})$. Suppose
    $$(i_-^{-1}, m_{\widetilde{L}}): Sh^b_{\widetilde{L}}(M \times \mathbb{R} \times \mathbb{R}_{>0}) \rightarrow Sh^b_{\Lambda_-}(M \times \mathbb{R}) \times_{Loc^b(\Lambda_-)} Loc(L)$$
    is essentially surjective with a chosen right inverse $\Psi_L$. Then the Lagrangian cobordism functor for the immersed cobordism is
    $$\Phi_L = i_+^{-1} \circ \Psi_L: \, Sh^b_{\Lambda_-}(M \times \mathbb{R}) \times_{Loc^b(\Lambda_-)} Loc^b(L) \rightarrow Sh^b_{\Lambda_+}(M \times \mathbb{R}).$$
\end{definition}
\begin{remark}
    Consider $Sh^b_{\Lambda_\pm}(M \times \mathbb{R})_\text{tri}$ and $Sh^b_{\widetilde{L}}(M \times \mathbb{R} \times \mathbb{R}_{>0})_\text{tri}$ be the subcategories of sheaves whose microlocalizations are constant local systems. When
    $$i_-^{-1}: Sh^b_{\widetilde{L}}(M \times \mathbb{R} \times \mathbb{R}_{>0})_\text{tri} \rightarrow Sh^b_{\Lambda_-}(M \times \mathbb{R})_\text{tri}$$
    is essentially surjective with a chosen right inverse $\Psi_L$, then the cobordism functor is
    $$\Phi_L = i_+^{-1} \circ \Psi_L: \, Sh^b_{\Lambda_-}(M \times \mathbb{R})_\text{tri} \rightarrow Sh^b_{\Lambda_+}(M \times \mathbb{R})_\text{tri}.$$
    This means that for any sheaf $\mathscr{F}_- \in Sh^b_{\Lambda_-}(M \times \mathbb{R})_\text{tri}$, we choose an extension $\mathscr{F} \in Sh^b_{\widetilde{L}}(M \times \mathbb{R} \times \mathbb{R}_{>0})_\text{tri}$, and the assignment $\mathscr{F}_- \mapsto i_+^{-1}\mathscr{F}$ determines a cobordism functor.
\end{remark}

    One may compare the definition with the work of Pan-Rutherford \cite{PanRuther}, where they considered augmentations $\epsilon_- \in \mathcal{A}ug_-(\Lambda_-)$ (without enhancing with loop space coefficients). When $\epsilon_-$ lifts to an augmentation $\widetilde{\epsilon} \in \mathcal{A}ug_-(\widetilde{L})$, then one can pull it back to get $\epsilon_+ \in \mathcal{A}ug_-(\Lambda_+)$.

    Under the above framework, the correspondences of sheaf or flag moduli spaces induced by algebraic weaves \cite[Theorem 1.4]{AlgWeave} is indeed induced by the immersed Lagrangian cobordism functors (where immersed cobordisms are given by weaves with caps and cups).

\subsection{Action Filtration and Cobordism Length}
    Legendrian contact homologies are dg algebras generated by Reeb chords on the Legendrians, and the lengths of Reeb chords naturally induces a filtration on the dg algebra. It is known that Lagrangian cobordism maps between Legendrian contact homologies preserve the action filtration by the lengths of Reeb chords \cite{Ekcobordism,SabTraynorLength}, from which one can define the capacities of Legendrian submanifolds with augmentations \cite{SabTraynorLength}. 

    Under the framework of Lagrangian cobordisms as correspondences between the positive and negative ends, we are able to recover the action filtration of the Lagrangian cobordism maps from the sheaf theory perspective.

    Recall that in our previous paper \cite{LiEstimate}, following Asano-Ike \cite{AsanoIke}, for sheaves $\SF, \SG \in Sh_\Lambda(M \times \bR)_0$, we have defined a persistence module $\mathscr{H}om_{(0,\infty)}(\SF, \SG)$, in particular, a sequence of modules with maps, where $c_i$ are lengths of Reeb chords on $\Lambda$,
    $$Hom(\SF, \SG) \rightarrow Hom(\SF, T_{c_1}(\SG)) \rightarrow \dots \rightarrow Hom(\SF, T_{c_k}(\SG)) \rightarrow \dots$$
    so that when $\mathrm{supp}(\SF)$ and $\mathrm{supp}(\SG)$ are compact, the sequence will end with zero (in finitely many steps).

    Consider on $J^1(M)$ the local coordinates $(x, \xi, t)$ where $x \in M, \xi \in T^*_xM$ and $t \in \bR$. For $\Lambda \subset J^1(M)$, we will use the convention
    $$\Lambda^s = \{(x, s\xi, st) | (x, \xi, t) \in \Lambda\} \subset J^1(M),$$
    which is the image of $\Lambda$ under the contactomorphism defined by scaling.

\begin{theorem}\label{thm:persist-intro}
    Let ${L} \subset J^1(M) \times \mathbb{R}_{>0}$ be an embedded Lagrangian cobordism from $\Lambda_- \subset J^1(M)$ to $\Lambda_+ \subset J^1(M)$ that is conical outside $J^1(M) \times (s_-, s_+)$. Let $\SL, \SL' \in Loc^b(L)$, $\SF_-, \SG_- \in Sh_{\Lambda_-}^b(M \times \bR)_0$ and $\SF_+, \SG_+ \in Sh_{\Lambda_+}^b(M \times \bR)_0$, such that under the Lagrangian cobordism functor
    $$\Phi_L(\SF_-, \SL) = \SF_+, \; \Phi_L(\SG_-, \SL') = \SG_+.$$
    Let $\SF_-^{s_-}, \SG_-^{s_-} \in Sh^b_{\Lambda_-^{s_-}}(M \times \bR)_0$ and $\SF_+^{s_+}, \SG_+^{s_+} \in Sh^b_{\Lambda_+^{s_+}}(M \times \bR)_0$ be the images of $\SF_-, \SG_-$ and $\SF_+, \SG_+$ under the scaling contactomorphism. Then there is an (action decreasing) map between persistence modules
    $$\mathscr{H}om_{(0,+\infty)}(\SF_+^{s_+}, \SG_+^{s_+}) \longrightarrow \mathscr{H}om_{(0,+\infty)}(\SF_-^{s_-}, \SG_-^{s_-}).$$
\end{theorem}
\begin{remark}
    In particular, using the description of the persistence module in terms of Reeb chords, the Lagrangian cobordism map preserves the action filtration of Reeb chords rescaled by $s > 0$.
\end{remark}

    The action filtration of Lagrangian cobordism maps has been used to obtain a number of numerical results, including the minimal lengths \cite{SabTraynorLength} of the embedded Lagrangian cobordisms by Sabloff-Traynor. The above result on action filtration allows us to prove numerical relations between the Legendrian capacities for a Lagrangian cobordisms. We in particular recover the result on minimal lengths of cobordisms in \cite[Theorem 1.1 \& 1.2]{SabTraynorLength}.

\begin{remark}
    All the results on the action filtration essentially rely on the assumption that the Lagrangian cobordism is embedded. For immersed exact Lagrangian cobordisms, one can easily build counterexamples.
\end{remark}

\subsection{Constructing Legendrians with simple sheaf categories}
    The conditional sheaf quantization functor essentially classifies sheaves with singular support on the conical Legendrian filling (with no Reeb chords). Using the construction of Ekholm and Courte-Ekholm \cite{CourteEkholm} by gluing two copies of the conical Legendrians together along their boundary, we can construct Legendrian submanifolds with sheaf categories equivalent to local systems on the Lagrangian fillings without appealing to the Legendrian surgery formula.

\begin{theorem}\label{thm:courteekholm-intro}
    For any $n$-dimensional smooth manifold $L$ with boundary such that $TL \otimes_\bR \bC$ is trivial, there is an $n$-dimensional Legendrian $\Lambda(L, L)$ diffeomorphic to $L \cup_{\partial L} L$ with Lagrangian filling $K(L, L)$ diffeomorphic to $L \times (0, 1)$ and
    $$Sh^b_{\Lambda(L, L)}(\bR^{n+1})_0 \simeq Loc^b(L), \; Sh^c_{\Lambda(L, L)}(\bR^{n+1})_0 \simeq Loc^c(L).$$
\end{theorem}
\begin{remark}
    It is proved \cite{CourteEkholm} that when $L \cong D^{n}$ then $\Lambda(D, D)$ is Legendrian isotopic to the standard unknot. Note that this construction is generalized by Roy by considering open book decompositions whose page contains a Lagrangian $L$ with boundary $\partial L$ \cite{Roy}.
\end{remark}

    When $n \geq 3$, for any $n$-dimensional smooth manifold $K$ with boundary such that $TK \otimes_\bR \bC$ is trivial, Eliashberg-Ganatra-Lazarev \cite[Theorem 4.5]{FlexLag} constructed Legendrian submanifolds $\Lambda \subset S^{2n-1}$ with flexible Lagrangian fillings $K \subset D^{2n}$ such that\footnote{The result \cite[Theorem 4.7]{FlexLag} is stated in terms of the wrapped Fukaya categories after handle attachments along $\Lambda$ the Legendrian spheres, but the proof applies to any $\Lambda$ and their partially wrapped Fukaya categories line by line by the Legendrian surgery formula with loop space coefficients \cite{AsplundEkholm}.}
    $$\mathcal{W}(D^{2n}, \Lambda) \simeq \mathrm{Perf}\,C_{-*}(\Omega_*K)$$
    by the Legendrian surgery formula. When $L$ is a regular Lagrangian filling \cite{FlexLag}, Courte-Ekholm proved that $K(L, L)$ is not only flexible but subcritical inside $D^{2n}$ \cite{CourteEkholmv1}\footnote{We would like to thank Oleg Lazarev for explaining to us this result about subcritical Lagrangians in the first version of Courte and Ekholm's preprint.}. In general, it is not known whether $K(L, L)$ is always flexible, but $\Lambda(L, L)$ is determined by the maps over tangent bundles $TL \otimes_\bR \bC \rightarrow TD^{2n}$ \cite{CourteEkholm}.

\subsection*{Acknowledgements}
    We would like to thank Tomohiro Asano and Yuichi Ike for very important discussions on sheaf quantizations of noncompact Lagrangians. Actually, the idea on the construction of sheaf quantization in Section \ref{sec:noncpt} first took shape during the collaboration with Tomohiro and Yuichi on a different project. Then the author would like to express gratitude to his advisors Emmy Murphy and Eric Zaslow for helpful discussions and warm encouragements throughout, and in particular, to Emmy Murphy for asking the sheaf theoretic proof of the result on minimal lengths of cobordisms. We would like to thank Xin Jin and David Treumann for fruitful discussions on the sheaf quantization results, and in particular David Treumann for numerous enlightening questions on this project. We would like to thank Roger Casals, C\^ome Dattin, St\'ephane Guillermou, Hingyuan Hu, James Hughes, Christopher Kuo and Jun Zhang for their interest in this project. Finally, we would like to thank Tomohiro Asano and Yuichi Ike again for helpful comments on an earlier version of the draft, and Oleg Lazarev for important comments on the result of Courte and Ekholm.

\section{Geometric Preliminary of Lagrangian Cobordisms}

    We discuss the necessary geometric preliminaries, namely the translation between exact Lagrangian cobordisms and conical Legendrian cobordisms, and Weinstein tubular neighbourhoods of exact Lagrangian cobordisms. Everything in this section is essentially trivial, but they turn out to be the key ingredients we need for our main results.

    We remark that, as mentioned in Pan-Rutherford \cite{PanRutherImmersed}, conical Legendrian cobordisms are equivalent to Legendrian cobordisms in the sense of Arnol'd \cite{ArnoldCobordism}, which is a homotopical theoretic object \cite{AudinCobordism,EliashCobordism,Vasilev}. However, conical Legendrian cobordisms with no Reeb chords, coming from embedded Lagrangian cobordisms, are much more rigid, and the only examples we know are Legendrian isotopies \cite{GroEliashGF,Chantraine}, Lagrangian handle attachments \cite{EHK,Rizconnectsum}, and Lagrangian caps \cite{LagCap}.

\subsection{Lagrangian cobordisms and Legendrian cobordisms}\label{sec:conical-leg}

    In this section, we explain the relation between Lagrangian cobordisms in the symplectization of $J^1(M) \cong T^{*,\infty}_{\tau > 0}(M \times \bR_t)$ and conical Legendrian cobordisms in $J^1(M \times \mathbb{R}_{>0,s}) \cong T^{*,\infty}_{\tau > 0}(M \times \bR_{>0,s} \times \bR_z)$. For the local coordinates $(x, \xi, t) \in J^1(M)$, the contact form is $\alpha_\text{std} = dt + \xi dx$; for the local coordinates $(x, s; y, \sigma; z) \in J^1(M \times \bR_{>0,s})$, the contact form is $\alpha_\text{std} = dz + \sigma ds + y dx$. 

    Let $(Y, \alpha)$ be a cooriented contact manifold. The symplectization is defined as $(Y \times \mathbb{R}_{>0,s}, d(s\alpha))$. Following \cite[Section 2.8]{SFT}, Chantraine \cite{Chantraine} and Ekholm \cite{Ekcobordism}, for instance, considered the category of Lagrangian cobordisms between Legendrians in the symplectization.

\begin{definition}
    The category of Lagrangian cobordisms $\mathrm{Cob}(Y)$, has objects being Legendrian submanifolds $\Lambda \subset Y$ and morphisms $Hom(\Lambda_-, \Lambda_+)$ being exact Lagrangian submanifolds $L \subset (Y \times \mathbb{R}_{>0,s}, d(s\alpha))$ with $s\alpha|_L = df_L$ such that
    $$L \cap (Y \times (0, s_-)) = \Lambda_- \times (0, s_-), \,\, L \cap (Y \times (s_+, +\infty)) = \Lambda_+ \times (s_+, +\infty)$$
    for some $0 < s_- < s_+ < +\infty$, and the primitive $f_L$ is a constant on $\Lambda_- \times (0, s_-)$ and $\Lambda_+ \times (s_+, +\infty)$. We call such an $L$ a Lagrangian cobordism from $\Lambda_-$ to $\Lambda_+$.

    Compositions in $\mathrm{Cob}(Y)$ are defined by concatenating Lagrangian cobordisms along their common conical ends. We will denote the concatenation of $L_0 \in Hom(\Lambda_0, \Lambda_1)$ and $L_1 \in Hom(\Lambda_1, \Lambda_2)$ by $L_0 \cup L_1$.
\end{definition}
\begin{remark}
    The assumption that the primitive $f_L$ is a constant on $\Lambda_- \times (0, s_-)$ and $\Lambda_+ \times (s_+, \infty)$ is made to ensure that concatenations of exact Lagrangians are exact \cite{ChantraineEx}.
\end{remark}

    For exact Lagrangians in the symplectization $(J^1(M) \times \bR_{>0,s}, d(s\alpha_\text{std}))$, one can consider the Legendrian lift in $((J^1(M) \times \bR_{>0,s}) \times \bR_w, dw + s\alpha_\text{std})$. It is known \cite{ChekanovGF,PanRuther} that there is a (strict) contactomorphism
    \[\begin{array}{cccc}
    \varphi: & (J^1(M) \times \bR_{>0,s}) \times \bR_w & \rightarrow & J^1(M \times \bR_{>0,s}) \\
    & (x, \xi, t; s; w) & \mapsto & (x, s; s\xi, -t; st + w).
    \end{array}\]
    Therefore, an exact Lagrangian cobordism gives a conical Legendrian cobordism with no Reeb chords \cite{PanRuther} (see also \cite{CasalsZas}).

\begin{definition}\label{def:conical-cob}
    Let $\Lambda_\pm \subset J^1(M)$ be Legendrian submanifolds. Then a Legendrian $\widetilde{L} \subset J^1(M \times \bR_{>0})$ is a conical Legendrian cobordism from $\Lambda_-$ to $\Lambda_+$ if for some $s_-, s_+ \in \bR_{>0}$, there are $w_{0,-}, w_{0,+} \in \bR$ such that
    \[\begin{split}
    \widetilde{L} \cap J^1(M \times (0, s_-)) &= \{(x, s, s\xi, -t, st + w_{0,-}) | (x, \xi, t) \in \Lambda_-, s \in (0, s_-)\}, \\
    \widetilde{L} \cap J^1(M \times (s_+, +\infty)) &= \{(x, s, s\xi, -t, st + w_{0,+}) | (x, \xi, t) \in \Lambda_+, s \in (s_+, +\infty)\}.
    \end{split}\]
\end{definition}

\begin{figure}
  \centering
  \includegraphics[width=0.8\textwidth]{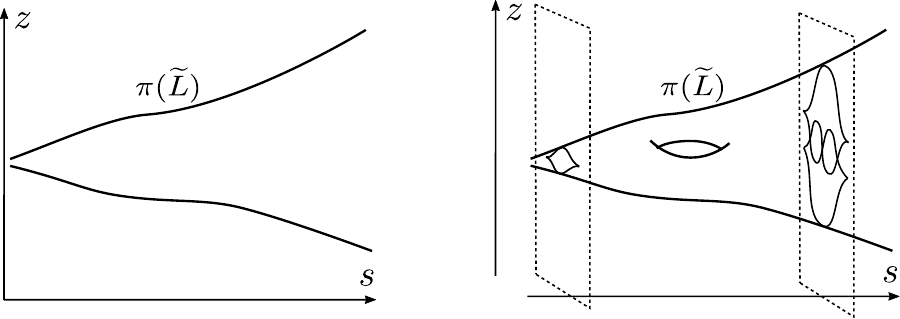}\\
  \caption{The front projection of a conical Legendrian cobordism $\widetilde{L} \subset J^1(\mathrm{pt} \times \bR_{>0})$ (on the left) and $J^1(\mathbb{R} \times \bR_{>0})$ (on the right).}\label{fig:conic-leg}
\end{figure}

\subsection{Weinstein neighbourhood for Legendrian cobordisms}
    For a closed Lagrangian submanifold $L \subset X$, Weinstein neighbourhood theorem asserts that there is a Weinstein tubular neighbourhood of $L \subset X$ which is symplectomorphic to a neighbourhood of the zero section $L \subset T^*L$. Similarly, for any closed Legendrian submanifold $\Lambda \subset Y$, there is a Weinstein tubular neighbourhood of $\Lambda \subset Y$ which is contactomorphic to a neighbourhood of the zero section $\Lambda \subset J^1(\Lambda)$.

    However, the neighbourhood theorem for noncompact Lagrangian/Legendrian submanifolds could be nontrivial, as the radius of the tubular neighbourhood may not have a positive lower bound with respect to the given Riemannian metric. This will be essential when we discuss the sheaf quantization problem for noncompact Lagrangian/Legendrians in Sections \ref{sec:noncpt} and \ref{sec:quan-cob}. To deal with this issue, we first introduce the notion of an adapted metric following \cite[Section 2.2.2]{GroEliashGF}.

\begin{definition}[Eliashberg-Gromov \cite{GroEliashGF}]
    A Riemannian metric $g$ on a symplectic manifold $X$ is adapted to the symplectic form $\omega$ on $X$ if for any $H \in C^\infty(X)$
    $$\|dH\|_g = \|X_H\|_g.$$
    or equivalently $\omega = \sum_{i=1}^n dx_i \wedge dy_i$ for some $g$-orthonormal coframing
    $$\left< dx_1, \dots, dx_n, dy_1, \dots, dy_n \right>.$$

    A Riemannian metric $g$ on a contact manifold $Y$ is adapted to the contact form $\alpha$ on $Y$ if for any $H \in C^\infty(Y)$
    $$\|dH\|_g^2 + |H|^2 = \|X_H\|^2_g$$
    or equivalently $\alpha = dz - \sum_{i=1}^n dx_i \wedge dy_i$ for some $g$-orthonormal coframing
    $$\left< dx_1, \dots, dx_n, dy_1, \dots, dy_n , dz \right>.$$
\end{definition}
\begin{example}\label{ex:adapt-metric}
    Consider $X = T^*M$ and $\omega = d\lambda_\text{std}$. Then a Riemannian metric $g_M$ on $M$ determines an adapted Riemannian metric on $T^*M$ by
    $$g_{T^*M} = g_M + g_M^\vee:  \in T_xM \oplus T_x^*M \times T_xM \oplus T_x^*M \rightarrow \bR,$$
    where $g^\vee: T^*M \otimes T^*M \rightarrow \bR$ is the dual bilinear form to $g: TM \times TM \rightarrow \bR$. It also determines an adapted Riemannian metric on $J^1(M)$ by
    $$g_{J^1(M)} = g_{T^*M} + dz^2:  T_xM \oplus T_x^*M \oplus T_z\bR \times T_xM \oplus T_x^*M \oplus T_z\bR \rightarrow \bR.$$
    In particular, when $g_M$ is complete, $g_{T^*M}$ and $g_{J^1(M)}$ are complete as well. We call them the standard adapted metric on $T^*M$ and $J^1(M)$.
\end{example}

    Later we will see in Sections \ref{sec:noncpt} and \ref{sec:quan-cob} that the reason we discuss metrics on symplectic/contact manifolds is to understand when a noncompact Hamiltonian vector field can be integrated. Adapted metrics allow us to estimate the norm of the Hamiltonian vector fields in terms of their $C^1$-norm. On the other hand, complete metrics allow us to deduce existence of the integration flow from the estimation of the norm of vector fields.

    It is proved that any symplectic manifold admits a complete adapted metric \cite{GroEliashGF}. It seems unclear whether a contact manifold always has a complete adapted metric, but we will only need the case of cotangent bundles and 1-jet bundles.

\begin{definition}\label{def:nbhd-positive-rad}
    Let $L \subset X$ be a submanifold. A (tubular) neighbourhood $U$ of $L$ of positive radius $r > 0$ with respect to a metric $g$ on $X$ is a (tubular) neighbourhood $U$ such that for any $x \in X$ with $d_g(x, L) \leq r$, we have $x \in U$.
\end{definition}
\begin{lemma}\label{lem:nbhd-symp-cont}
    Let $L \subset (X, d\lambda_X)$ be an exact Lagrangian submanifold. Suppose $L$ has a tubular neighbourhood of positive radius $r > 0$ with respect to a complete adapted metric $g_X$, then the Legendrian lift $\widetilde{L} \subset (X \times \bR, dt - \lambda_X)$ also has a tubular neighbourhood of positive radius $r > 0$ with respect to the complete adapted metric $g_X \oplus g_{\bR, \text{std}}$, where $g_{\bR, \text{std}}$ is the Euclidean metric.
\end{lemma}

    However, for Lagrangian cobordisms between closed Legendrians, such a tubular neighbourhood does not exist, for the simple reason that the symplectic area near the concave end of the symplectization has an upper bound, while a tubular neighbourhood of positive radius for the conical/cylindrical submanifold cannot have a bounded symplectic area.

\begin{figure}
  \centering
  \includegraphics[width=1.0\textwidth]{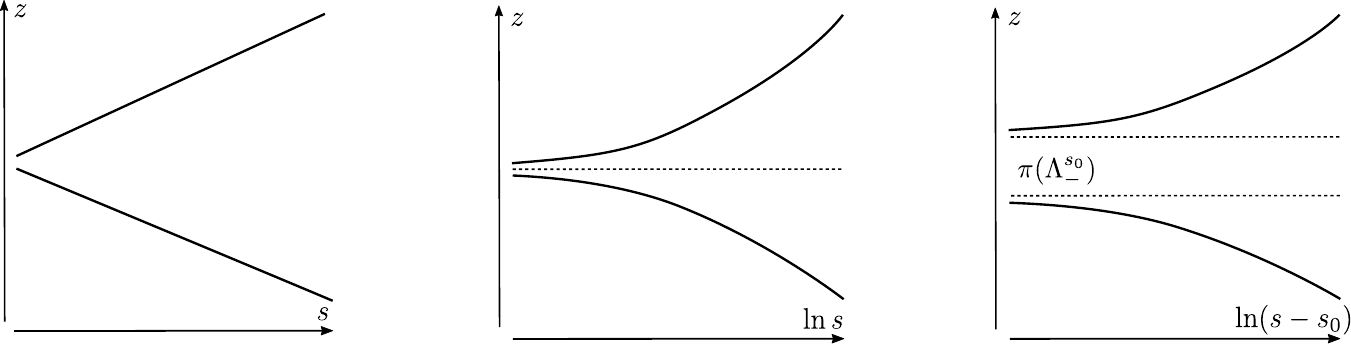}\\
  \caption{The figure on the left is the front projection of a conical Legendrian cobordism $\widetilde{L} \subset J^1(\mathrm{pt} \times \bR_{>0})$. The figure in the middle is the front projection after applying the diffeomorphism $s \mapsto \ln s$, where the complete adapted metric induced by $g_M + s^{-2} ds^2$ becomes the Euclidean metric. The figure on the right is the front projection after applying the diffeomorphism $s \mapsto \ln (s - s_0)$, where the complete adapted metric induced by $g_M + (s - s_0)^{-2} ds^2$ becomes the Euclidean metric.}\label{fig:leg-nbhd}
\end{figure}

    For simplicity, we restrict to the case $J^1(M) \subset T^{*,\infty}(M \times \bR)$. In this case the symplectization is symplectomophic to a cotangent bundle
    $$J^1(M) \times \bR_{>0} \xrightarrow{\sim} T^*(M \times \bR_{>0}), \; (x, \xi, t; s) \mapsto (x, s, s\xi, -t).$$
    We will consider a different standard complete adapted metric on $T^*(M \times \bR_{>0})$ induced by the complete metric $g_M + s^{-2} ds^2$ on $M \times \bR_{>0}$, which is $g_M + g_M^\vee + s^{-2} ds^2 + s^2 dt^2$; see Figure \ref{fig:leg-nbhd} middle (note that the metric $s^{-2} ds^2$ is identical to the Euclidean metric under the diffeomorphism $s \mapsto \ln s$).

\begin{lemma}\label{lem:nbhd-sft-cob-1}
    Let $L \subset T^{*}(M \times \bR_{>0}) \cong J^1(M) \times \bR_{>0}$ be a Lagrangian cobordism between closed Legendrians from $\Lambda_-$ to $\Lambda_+ \subset J^1(M)$. Then for any sufficiently small $s_0 > 0$, $L \cap T^{*}(M \times (s_0, +\infty))$ has a tubular neighbourhood of positive radius $r > 0$ with respect to the adapted metric on $T^*(M \times \bR_{>0})$.
\end{lemma}
\begin{proof}
    First consider $L \cap T^{*,\infty}M \times (s_0, s_1')$ where $s_0$ is small, $s_1$ is sufficiently large and $s_1' > s_1$. Since the intersection is a bounded subset, there exists a tubular neighbourhood of positive radius $r > 0$. Then consider $L \cap T^{*,\infty}M \times (s_1, +\infty)$. Since $s_1$ is sufficiently large, we may assume that
    $$L \cap T^{*,\infty}M \times (s_1, +\infty) = \{(x, s, s\xi) | (x, \xi) \in \Lambda_+, s \in (s_1, +\infty)\}.$$
    Then since $\Lambda_+$ is a closed Legendrian, it has a tubular neighbourhood of positive radius $r > 0$ with respect to the any complete metric. The adapted metric on $T^*(M \times (s_1, +\infty))$ is given by
    $$g_{T^*(M \times \bR_{>0})} = g_M + g_M^\vee + s^{-2} ds^2 + s^2 d\sigma^2 = dx^2 + dy^2 + s^{-2} ds^2 + s^2 d\sigma^2.$$
    Under the identification $J^1(M) \times \bR_{>0} \xrightarrow{\sim} T^*(M \times \bR_{>0}), \; (x, s, y, \sigma) = (x, s, s\xi, -t)$, the cone on the right hand side is identified with the product cylinder on the left hand side $\Lambda_+ \times (s_1, +\infty)$. The pullback metric is identified with
    $$g_M + s^2 g_M^\vee + s^2 dt^2 + s^{-2} ds^2 = dx^2 + s^2 d\xi^2 + s^2 dt^2 + s^{-2} ds^2,$$
    which is bounded from below by the product metric $g_{J^1(M)} + s^{-2} ds^2$ on $J^1(M) \times (s_1, +\infty)$. Therefore, by considering the product neighbourhood of $\Lambda_+ \times (s_1, +\infty)$, we get a tubular neighbourhood of positive radius.
\end{proof}

    We can restrict to the open submanifold $J^1(M) \times (s_0, +\infty)$ which is symplectomorphic to $T^*(M \times (s_0, +\infty))$. Consider the complete adapted metric on the submanifold induced by $g_M + (s - s_0)^{-2} ds^2$ on $M \times (s_0, +\infty)$; see Figure \ref{fig:leg-nbhd} right (note that the metric $(s - s_0)^{-2} ds^2$ is identical to the Euclidean metric under the diffeomorphism $s \mapsto \ln (s - s_0)$). The advantage of this new metric is that it is complete on $T^*(M \times (s_0, +\infty))$, so that we can deal with the subset independently
    $$L \cap T^*(M \times (s_0, +\infty)) \subset T^*(M \times (s_0, +\infty))$$
    when studying Hamiltonian vector fields and their integration flows in later sections.

    However, under this metric, $L \cap T^*(M \times (s_0, +\infty))$ is no longer a bounded subset in an ambient manifold, so it is no longer true that $L \cap T^*(M \times (s_0, +\infty))$ has a tubular neighbourhood of positive radius $r > 0$ with respect to this new complete adapted metric. We only have a weaker result (weaker in the sense of Lemma \ref{lem:nbhd-symp-cont}).

\begin{lemma}\label{lem:nbhd-sft-cob-2}
    Let $L \subset J^1(M) \times \bR_{>0} \cong T^*(M \times \bR_{>0})$ be a Lagrangian cobordism between closed Legendrians from $\Lambda_-$ to $\Lambda_+ \subset J^1(M)$. Then for any sufficiently small $s_0 > 0$, the Legendrian lift $\widetilde{L} \cap J^1(M \times (s_0, +\infty))$ has a tubular neighbourhood of positive radius $r > 0$ with respect to the complete adapted metric on $J^1(M \times (s_0, +\infty))$.
\end{lemma}
\begin{proof}
    We notice that the same argument in Lemma \ref{lem:nbhd-sft-cob-1} shows that for any $s_0' > s_0$, $L \cap J^1(M \times (s_0', +\infty))$ admits a tubular neighbourhood of positive radius $r > 0$. Therefore, by Lemma \ref{lem:nbhd-symp-cont}, it suffices to show that the Legendrian lift $\widetilde{L} \cap J^1(M \times (s_0, s_0'))$ admits a tubular neighbourhood of positive radius.

    On $J^1(M \times (s_0, s_0'))$, we know that the Legendrian lift of the cobordism is
    \[\begin{split}
    \widetilde{L} \cap J^1(M \times (s_0, s_0')) & = \{(x, s, s\xi, -t, st) | (x, \xi, t) \in \Lambda_-, s \in (s_0, s_0')\} \\
    & = \{ (x, s, y, -s^{-1}z, z) | (x, y, z) \in \Lambda_-^s, s \in (s_0, s_0')\}.
    \end{split}\]
    where $\Lambda_-^s = \{(x, s\xi, st) | (x, \xi, t) \in \Lambda_-\}$ is a closed Legendrian. $\Lambda_-^s$ has a tubular neighbourhood of positive radius $r > 0$ with respect to the complete adapted metric on $J^1(M)$
    $$g_{J^1(M)} = g_M + g_M^\vee + dz^2.$$
    On the other hand, the complete adapted metric on $J^1(M \times (s_0, +\infty))$ is given by
    $$g_{J^1(M \times (s_0, +\infty))} = g_M + g_M^\vee + (s - s_0)^{-2} ds^2 + (s - s_0)^2 dt^2 + dz^2. $$
    When $s > s_0 > 0$, we know that $g_{J^1(M \times (s_0, +\infty))}$ is bounded from below by the product metric $ g_{J^1(M)} + (s - s_0)^{-2} ds^2 + (s - s_0)^2 dt^2$. By considering the product neighbourhood of $\Lambda_-^{s_0} \times (s_0, s_0')$, we get a tubular neighbourhood of positive radius for the cylinder $\Lambda_-^{s_0} \times (s_0, s_0')$. Finally, we estimate the distance between the cylinder and the cone
    $$\{(x, s, y, 0, z) | (x, y, z) \in \Lambda_-^{s_0}, s \in (s_0, s_0')\}, \;\, \{ (x, s, y, -s^{-1}z, z) | (x, y, z) \in \Lambda_-^s, s \in (s_0, s_0')\}.$$
    Consider pairs of the form $(x, s, y, 0, z)$ and $(x, s, ss_0^{-1}y, -s^{-1}z, ss_0^{-1}z)$, and set
    $$r = \max_{(x, y, z) \in \Lambda_-^{s_0}}(|y|^2 + z^2)^{1/2}.$$
    Then we know that the distance is bounded by
    \[\begin{split}
    \sup_{(x, y, z) \in \Lambda_-^{s_0}, s \in (s_0, s_0')} & d_{J^1(M \times (s_0, +\infty))}\big((x, s, y, 0, z), (x, s, ss_0^{-1}y, -s^{-1}z, ss_0^{-1}z)\big) \\
    & \leq \sup_{(x, y, z) \in \Lambda_-^{s_0}}d_{(s_0, +\infty)}((s_0')^{-1}z, 0) + d_{J^1(M)}\big((x, y, z), (x, s_0's_0^{-1}y, s_0's_0^{-1}z)\big) \\
    & \leq \sup_{(x, y, z) \in \Lambda_-^{s_0}}(s - s_0)((s_0')^{-1}z - 0) + \big((y - s_0's_0^{-1}y)^2 + (z - s_0's_0^{-1}z)^2\big)^{1/2} \\
    & \leq (s_0' - s_0)(s_0')^{-1}r + (s_0's_0^{-1} - 1)r = (s_0' s_0^{-1} - s_0(s_0')^{-1})r.
    \end{split}\]
    We know that the distance can be arbitrarily small when $s_0'$ is sufficiently close to $s_0$. Therefore, a tubular neighbourhood of positive radius for $\Lambda_-^{s_0} \times (s_0, s_0')$ gives a (possibly smaller) tubular neighbourhood of positive radius for the cone.
\end{proof}

\section{Sheaf Quantization of Lagrangian Cobordisms}

    Given an exact Lagrangian cobordism $L$ from $\Lambda_-$ to $\Lambda_+ \subset J^1(M)$, following Section \ref{sec:conical-leg}, we can identify it with a conical Legendrian cobordism with no Reeb chords $\widetilde{L} \subset J^1(M \times \mathbb{R}_{>0})$. The main theorem in this section is a conditional sheaf quantization theorem:

\begin{theorem}\label{thm:cond-quan-cob}
    Let $\widetilde{L} \subset J^1(M \times \mathbb{R}_{>0})$ be a conical Legendrian cobordism with no Reeb chords from $\Lambda_- \subset J^1(M)$ to $\Lambda_+ \subset J^1(M)$ with vanishing Maslov class and relative 2nd Stiefel-Whitney class. Then there exists a fully faithful (conditional) sheaf quantization functor
    $$\Psi_{\widetilde{L}}: \, Sh_{\Lambda_-}(M \times \mathbb{R})_0 \times_{Loc(\Lambda_-)} Loc(\widetilde{L}) \xrightarrow{\sim} Sh_{\widetilde{L}}(M \times \mathbb{R} \times \mathbb{R}_{>0})_0$$
    on the subcategories of sheaves with acyclic stalks at $-\infty$.
\end{theorem}

    Following the convention of \cite{GKS,Gui}, sheaf quantization is a functor from Lagrangian submanifolds $L \subset T^*M$ with local systems to sheaves with singular support on the Legendrian lift, i.e.
    $$\Psi_L: \, Loc(L) \rightarrow Sh_{\widetilde{L}}(M \times \mathbb{R})_0.$$
    For conical Lagrangian cobordisms, we do not have a sheaf quantization functor in general. In fact, we need the prescribed data near the negative end $\Lambda_-$.

    Let $T_u: J^1(M \times \mathbb{R}_{>0}) \rightarrow J^1(M \times \mathbb{R}_{>0})$ be the Reeb flow.Guillermou and Jin-Treumann \cite{Gui,JinTreu} considered the sheaf quantization of closed exact Lagrangians or exact Lagrangians with positive conical ends. First they construct a doubling functor using a small Reeb push-off in a Weinstein tubular neighbourhood of the Legendrian lift
    $$Loc(L) \hookrightarrow Sh_{T_{-\epsilon}(\widetilde{L}) \cup T_\epsilon(\widetilde{L})}(M \times \mathbb{R})_0$$
    and then push one of the copies to infinity through a Legendrian isotopy and get
    $$Loc(L) \rightarrow Sh_{\widetilde{L}}(M \times \mathbb{R})_0.$$

    However, for general noncompact Lagrangians, without control on the Weinstein tubular neighbourhood, one may not get a uniform Reeb push-off $T_{-\epsilon}(\widetilde{L}) \cup T_\epsilon(\widetilde{L})$ for some fixed time $\epsilon > 0$, and then fail to connect the small Reeb push-off with some large Reeb push-off. Even though we could still find a Legendrian isotopy between small push-offs and large push-offs, the norm of the candidate Hamiltonian vector field may be unbounded and we may no longer get a Hamiltonian vector field. For Lagrangian cobordisms, one can easily see that the negative end is exactly where the radius of the Weinstein tubular neighbourhood loses control.

    Therefore, our strategy is to construct the doubling separately near the negative end and away from the negative end. Near the negative end, using the sheaf singularly supported on a single copy of the Legendrian, one can immediately define a sheaf supported on a double copy of the Legendrian by hand, while away from the negative end, we have good control on the radius of the Weinstein tubular neighbourhood theorem and the original doubling construction works. Then we show that one can push off one of the copies to infinity.

    We set up the general theory of the doubling construction and sheaf quantization for noncompact Legendrians in the course of the proof, which we believe should apply to other interesting noncompact Legendrians as well.

\subsection{Preliminary in microlocal sheaf theory}
    Sheaves on manifolds are systematically studied by Kashiwara-Schapira \cite{KS}. They introduced the notion of singular supports and investigated the microlocal behaviour of sheaves on manifolds.

\begin{definition}
    Let $\mathscr{F} \in Sh(M)$ be a sheaf on $M$ over the dg-derived category of the base ring $\Bbbk$. Then the singular support $SS(\mathscr{F}) \subseteq T^*M$ (resp.~$SS^\infty(\mathscr{F}) \subseteq T^{*,\infty}M$) is the closure of points $(x, \xi) \in T^*M$ (resp.~$(x, \xi) \in T^{*,\infty}M$) such that there exists $\varphi \in C^1(M)$, $\varphi(x) = 0, d\varphi(x) = \xi$ and
    $$\Gamma_{\varphi^{-1}((-\infty, 0])}(\mathscr{F})_x \neq 0.$$
    For any conical subset $\widehat \Lambda \subseteq T^*M$ (resp.~any subset $\Lambda \subseteq T^{*,\infty}M$), $Sh_{\hat{\Lambda}}(M)$ (resp.~$Sh_\Lambda(M)$) is the full subcategory of sheaves such that $SS(\mathscr{F}) \subseteq \widehat\Lambda$ (resp.~$SS^\infty(\mathscr{F}) \subseteq \Lambda$).
\end{definition}

    When $\widehat \Lambda \subset T^*M$ is subanalytic Lagrangian (resp.~$\Lambda \subset T^{*,\infty}M$ is subanalytic Legendrian), $Sh_{\hat{\Lambda}}(M)$ (resp.~$Sh_\Lambda(M)$) are compactly generated dg categories \cite[Corollary 4.21]{GPS3}. Using the fact that restrictions $Sh_\Lambda(M) \rightarrow Sh_{\Lambda \cap T^*U}(U)$ and inclusions $Sh_\Lambda(M) \hookrightarrow Sh_{\Lambda'}(M)$ preserve limits and colimits, we can construct left and right adjoints and restrict the left adjoints to compact objects.

\begin{remark}\label{rem:cpt-prop}
    For $\Lambda \subseteq \Lambda'$, the tautological inclusion $\iota_{\Lambda \Lambda'}: Sh_\Lambda(M) \hookrightarrow Sh_{\Lambda'}(M)$ preserves limits and colimits and hence admits left and right adjoints $\iota_{\Lambda \Lambda'}^*, \iota_{\Lambda \Lambda'}^!$, where the left adjoint $\iota_{\Lambda \Lambda'}^*$ preserves compact objects \cite[Lemma 4.12]{GPS3} (see \cite{Kuo} for the geometric construction of the adjoint functors). When $\Lambda' = T^{*,\infty}M$, we often just omit them in the subscripts.

    For $U \subseteq M$, the restriction functor $i_U^{-1}: Sh_\Lambda(M) \rightarrow Sh_{\Lambda \cap T^*U}(U)$ preserves limits and colimits and hence admits left and right adjoints, where the left adjoint preserves compact objects. In fact, since $i_U^{-1} = i_U^{-1} \circ \iota_\Lambda$ where $\iota_\Lambda: Sh_\Lambda(M) \hookrightarrow Sh(M)$ is the inclusion, the left adjoint is $\iota_\Lambda^* \circ i_{U !}$, the composition of the left adjoint (between categories of all sheaves) $i_{U !}: Sh(U) \rightarrow Sh(M)$ and the left adjoint of the tautological inclusion $\iota_\Lambda^*: Sh(M) \rightarrow Sh_\Lambda(M)$.
\end{remark}

    The celebrated theorem of Guillermou-Kashiwara-Schapira showed that the category of sheaves is invariant under contact Hamiltonian isotopies and the action of contact Hamiltonian isotopies preserve singular supports.

\begin{definition}\label{lagmovie}
    Let $\widehat H_s: T^*M \times I \rightarrow \mathbb{R}$ be a homogeneous Hamiltonian on $T^*M$, and $H_s = \widehat H_s|_{T^{*,\infty}M}$ the corresponding contact Hamiltonian on $T^{*,\infty}M$. For a conical Lagrangian $\widehat\Lambda$, the Lagrangian movie of $\widehat\Lambda$ under the Hamiltonian isotopy $\varphi_{\widehat H}^s\,(s \in I)$ is
    $$\widehat\Lambda_{\widehat H} = \{(x, \xi, s, \sigma) | (x, \xi) = \varphi_{\widehat H}^s(x_0, \xi_0), \sigma = -\widehat H_s \circ \varphi_{\widehat H}^s(x_0, \xi_0), (x_0, \xi_0) \in \widehat \Lambda\} \subset T^*(M \times I).$$
    For a Legendrian $\Lambda$, the Legendrian movie of $\Lambda$ under the corresponding contact Hamiltonian isotopy is $\Lambda_H = \widehat\Lambda_{\widehat H} \cap T^{*,\infty}M$.
\end{definition}
\begin{theorem}[Guillermou-Kashiwara-Schapira \cite{GKS}*{Proposition 3.12}]\label{GKS}
    Let $H_s: T^{*,\infty}M \times I \rightarrow \mathbb{R}$ be a contact Hamiltonian on $T^{*,\infty}M$ and $\Lambda$ a Legendrian in $T^{*,\infty}M$. Then there are equivalences
    $$Sh_{\Lambda}(M) \xleftarrow{\sim} Sh_{\Lambda_{H}}(M \times I) \xrightarrow{\sim} Sh_{\varphi_{H}^1(\Lambda)}(M),$$
    given by restriction functors $i_0^{-1}$ and $i_1^{-1}$ where $i_s: M \times \{s\} \hookrightarrow M \times I$ is the inclusion. We denote their inverses by $\Psi_H^0$ and $\Psi_H^1$, and $\Psi_H = i_1^{-1} \circ \Psi_H^0$.
\end{theorem}

\subsection{Structure of Kashiwara-Schapira stack}
    We review the definition and properties of microlocalization and the sheaf of categories $\mu Sh$, which has been introduced and studied in \cite{KS}*{Section 6}, \cite{Gui}*{Section 6} or \cite{NadWrapped}*{Section 3.4}. This is a category that we will frequently use. Here we follow the definition in \cite{NadShen}*{Section 5}.

\begin{definition}\label{musheaf}
    Let $\widehat\Lambda \subseteq T^*M$ be a conical subset. Then define a presheaf of dg categories on $T^*M$ supported on $\widehat \Lambda$ to be
    $$\mu Sh^{\text{pre}}_{\widehat\Lambda}: \, \widehat \Omega \mapsto Sh_{\widehat \Lambda\, \cup \,T^*M \backslash \widehat \Omega}(M)/Sh_{T^*M \backslash \widehat \Omega}(M).$$
    The sheafification of the presheaf $\mu Sh^{\text{pre}}_{\widehat\Lambda}$ is $\mu Sh_{\widehat\Lambda}$. In particular, we write $\mu Sh = \mu Sh_{T^*M}$ for the sheaf of dg categories on $T^*M$.

    Let $Sh_{(\widehat\Lambda)}(M)$ be the subcategory of sheaves $\mathscr{F}$ such that there exists some neighbourhood $\widehat \Omega$ of $\widehat\Lambda$ satisfying $SS(\mathscr{F}) \cap \widehat \Omega \subset \widehat\Lambda$. For $\mathscr{F, G} \in Sh_{(\widehat\Lambda)}(M)$, let the sheaf of homomorphisms in the sheaf of categories $\mu Sh_{\widehat\Lambda}$ be
    $$\mu hom(\mathscr{F}, \mathscr{G})|_{\widehat\Lambda}: \, \widehat\Omega \mapsto Hom_{\mu Sh_{\widehat\Lambda}(\widehat \Omega)}(\mathscr{F, G}).$$
    Write $\mu hom(\mathscr{F}, \mathscr{G})$ to be the sheaf of homomorphisms in $\mu Sh$.

    Let $\Lambda \subset T^{*,\infty}M$ be a subset where $T^{*,\infty}M$ is identified with the unit cotangent bundle. Then $\mu Sh_\Lambda$ is defined by $\mu Sh_\Lambda = \mu Sh_{\Lambda \times \mathbb{R}_{>0}}|_\Lambda$.
\end{definition}
\begin{remark}
    We define the sheafification in the (large) category of dg categories whose morphisms are exact functors. When $\widehat\Lambda$ is a conical subanalytic Lagrangian, the sheafification takes value in the (large) category of presentable dg categories whose morphisms are colimit preserving functors \cite{NadShen}*{Remark 6.1}.
\end{remark}

    Denote by $m_\Lambda$ the natural quotient functor on the sheaf of categories, which, on the level of global sections, induces
    $$m_\Lambda: \, Sh_\Lambda(M) \rightarrow \mu Sh_\Lambda(\Lambda).$$
    We call $m_\Lambda$ the microlocalization functor. Like the restriction functors on sheaf categories, it preserves limits and colimits and hence admits a left adjoint preserving compact objects under certain assumptions.

\begin{remark}\label{rem:micro-cpt-prop}
    When $\Lambda \subset T^{*,\infty}M$ is subanalytic Legendrian, $Sh_\Lambda(M)$ and $\mu Sh_\Lambda(\Lambda)$ are compactly generated dg categories \cite[Corollary 7.19]{GPS3}. The microlocalization $m_\Lambda$ preserves limits and colimits, and hence admits left and right adjoints $m_\Lambda^*, m_\Lambda^!$, where $m_\Lambda^*$ preserves compact objects.
\end{remark}

    The following lemma immediately follows from the isomorphism $\Gamma(T^*M, \mu hom(\mathscr{F, G})) \simeq Hom(\mathscr{F}, \mathscr{G})$ \cite[Equation (4.3.1)]{KS} and the estimation $\mathrm{supp}(\mu hom(\mathscr{F, G})) \subseteq SS(\mathscr{F}) \cap SS(\mathscr{G})$ \cite{KS}*{Corollary 5.4.10}.

\begin{lemma}[\cite{NadWrapped}*{Remark 3.18}]\label{sheafmusheaf}
    Let $\widehat\Lambda \subset T^*M$ be a conical subanalytic Lagrangian. Then
    $$Sh_{M \cup \widehat\Lambda}(M) \xrightarrow{\sim} \mu Sh_{M \cup \widehat\Lambda}(M \cup \widehat\Lambda).$$
\end{lemma}

    Guillermou \cite{Gui} proved that the structure of the Kashiwara-Schapira stack for a smooth Legendrian can be trivialized once certain characteristic classes vanish.

\begin{theorem}[\cite{Gui}*{Proposition 6.6 \& Lemma 6.7}, \cite{NadShen}*{Corollary 5.4}]\label{microstalk}
    Let $\Lambda \subset T^{*,\infty}M$ be a subanalytic Legendrian. For any smooth point $p = (x, \xi) \in \Lambda \subset T^{*,\infty}M$, the stalk $\mu Sh_{\Lambda,p} \simeq \mathrm{Mod}(\Bbbk)$.
\end{theorem}

\begin{theorem}[Guillermou \cite{Gui}*{Theorem 11.5}]\label{Gui}
    Let $\Lambda \subset T^{*,\infty}M$ be a smooth Legendrian submanifold. Suppose the Maslov class $\mu(\Lambda) = 0$ and $\Lambda$ is relative spin, then as sheaves of categories
    $$\mu Sh_\Lambda \xrightarrow{\sim} Loc_\Lambda.$$
\end{theorem}

    Similar to the invariance of the sheaf category under Hamiltonian isotopies, the category of microlocal sheaves (Kashiwara-Schapira stack) is also invariant under contact Hamiltonian isotopies.

\begin{theorem}[Kashiwara-Schapira \cite{KS}*{Theorem 7.2.1}, Nadler-Shende \cite{NadShen}*{Lemma 5.6}, \cite{LiCobordism}*{Theorem 2.10}]\label{cont-trans}
    Let $H_s: T^{*,\infty}M \times I \rightarrow \mathbb{R}$ be a contact Hamiltonian on $T^{*,\infty}M$ and $\Lambda$ a Legendrian in $T^{*,\infty}M$. Then there are equivalences
    $$\mu Sh_\Lambda(\Lambda) \xleftarrow{\sim} \mu Sh_{\Lambda_H}(\Lambda_H) \xrightarrow{\sim} \mu Sh_{\varphi_H^1(\Lambda)}(\varphi_H^1(\Lambda))$$
    given by restriction functors $i_0^{-1}$ and $i_1^{-1}$ where $i_s: M \times \{s\} \hookrightarrow M \times I$ is the inclusion. We denote their inverses by $\Psi_H^0$ and $\Psi_H^1$, and $\Psi_H = i_1^{-1} \circ \Psi_H^0$.
\end{theorem}
\begin{remark}\label{GKSviaCont}
    From the proof in \cite[Theorem 2.10]{LiCobordism}, one may notice that there is a commutative diagram
    \[\xymatrix{
    Sh_{\Lambda_{H}}(M \times I) \ar[r]^{\hspace{7pt}i_s^{-1}} \ar[d] & Sh_{\varphi_{H}^s(\Lambda)}(M) \ar[d] \\
    \mu Sh_{\Lambda_H}(\Lambda_H) \ar[r]^{i_s^{-1}\hspace{12pt}} & \mu Sh_{\varphi_H^s(\Lambda)}(\varphi_H^s(\Lambda)).
    }\]
\end{remark}

\subsection{Doubling Functor for Noncompact Legendrians}\label{sec:noncpt}
    The doubling construction goes back to Guillermou \cite{Gui} (and Jin-Treumann \cite{JinTreu}), who defined an unconditional doubling functor for a compact exact Lagrangian (or an exact Lagrangian with positive conical ends)
    $$w_\Lambda: \, \mu Sh_\Lambda(\Lambda) \hookrightarrow Sh_{T_{-\epsilon}(\Lambda) \cup T_\epsilon(\Lambda)}(M \times \mathbb{R}),$$
    and is also formulated in a different way in Nadler-Shende \cite{NadShen}. Similar construction has also appeared in partially wrapped Fukaya categories \cite{SylvanOrlov,GPS2} and twisted generating families \cite{TwistGF}. This construction plays a crucial role in the proof of the sheaf quantization theorems.

    The goal of this section is to generalize the doubling construction to certain noncompact Legendrian submanifolds. Our main theorem is the following.

\begin{theorem}\label{thm:double-noncpt}
    Let $\Lambda \subset J^1(M)$ be any smooth Legendrian submanifold. If there exists a complete adapted metric on $J^1(M)$ such that $\Lambda$ admits a tubular neighbourhood of positive radius $\epsilon_0 > 0$, then for any $0 < \epsilon < \epsilon_0$, there exists a fully faithful functor
    $$w_\Lambda: \, \mu Sh_\Lambda(\Lambda) \hookrightarrow Sh_{T_{-\epsilon}(\Lambda) \cup T_\epsilon(\Lambda)}(M \times \mathbb{R})$$
    such that there is an exact triangle $T_{-\epsilon} \rightarrow T_\epsilon \rightarrow w_\Lambda \circ m_\Lambda$.
\end{theorem}

    Moreover, we get an adjunction property for the doubling functor.

\begin{theorem}\label{thm:double-noncpt-ad}
    Let $\Lambda \subset J^1(M)$ be any smooth Legendrian submanifold. If there exists a complete adapted metric on $J^1(M)$ such that $\Lambda$ admits a tubular neighbourhood of positive radius $\epsilon_0 > 0$, then the doubling functor can be identified with the left adjoint of microlocalization
    $$m_\Lambda^* = \iota_\Lambda^* \circ w_\Lambda[-1]: \, \mu Sh_\Lambda(\Lambda) \rightarrow Sh_{\Lambda}(M \times \mathbb{R}),$$
    where $\iota_\Lambda^*$ is the left adjoint to the inclusion $\iota_\Lambda: Sh_\Lambda(M \times \bR) \hookrightarrow Sh(M \times \bR)$.
\end{theorem}

    We remark that the assumption that $\Lambda$ admits a tubular neighbourhood of positive radius is crucial, and indeed this is why there is not in general a doubling functor for conical Legendrian cobordisms with a uniform $\epsilon > 0$ (and why we need extra data near the negative end of the cobordisms).

\subsubsection{Local doubling functor for Legendrians}
    We briefly recall the construction. First of all, given any $\mathscr{F} \in \mu Sh_\Lambda(\Lambda)$, for a sufficiently small open subsets $V \subset U \subset M \times \mathbb{R}$, the refined microlocal cut-off lemma \cite{Gui}*{Lemma 6.7} or \cite{KuoLi}*{Corollary 4.16} ensures that there exists a sheaf $\mathscr{F}_{U} \in Sh_\Lambda(U)$. Let
    $$\Lambda_{\pm T} = \{(x, t, u; \xi, \tau, \nu) \mid u > 0, (x, t; \xi, \tau) \in T_{\pm u}(\Lambda), \nu = -H_u \circ T_u(x, t; \xi, \tau)\}$$
    be the Legendrian movies of $\Lambda$ under the positive/negative Reeb flows. Let
    $$T_\pm: Sh_\Lambda(M \times \mathbb{R}) \rightarrow Sh_{\Lambda_{\pm T}}(M \times \mathbb{R} \times (0, +\infty))$$
    be the equivalences of sheaf categories induced by the corresponding Hamiltonian isotopies $T_\pm$ by Guillermou-Kashiwara-Schapira Theorem \ref{GKS} \cite{GKS}.

\begin{lemma}[\cite{Tamarkin1}*{Section 2.2.2}, \cite{Gui}*{Corollary 16.6}, \cite{GuiSchapira}*{Equation (76)--(77)}]\label{lem:small-pushoff}
    Let $\widetilde{L} \subset J^1(M \times \mathbb{R}_{>0})$ be a conical Legendrian cobordism from $\Lambda_- \subset J^1(M)$ to $\Lambda_+ \subset J^1(M)$. Then for any $\SF, \SG \in Sh_{\widetilde{L}}(M \times \bR \times \bR_{>0})$ and $\epsilon > 0$ sufficiently small,
    $$Hom(\SF, \SG) = Hom(\SF, T_\epsilon(\SG)).$$
\end{lemma}

    Then for $\epsilon > 0$ a sufficiently small number depending on $U, V$ and $\Lambda$, we define
    $$\widetilde{w}_\Lambda(\mathscr{F})_{V \times (0, \epsilon)} = \mathrm{Cone}(T_-(j_{U*}\mathscr{F}_{U}) \rightarrow T_+(j_{U*}\mathscr{F}_{U}))|_{V \times (0, \epsilon)}.$$
    Locally, it follows from the Sato-Sabloff exact triangle of Ike \cite[Section 4.3]{Ike} or \cite[Section 3.2]{LiEstimate}, \cite[Section 4.1]{KuoLi} that the assignment is fully faithful. Then using the local-to-global argument, we have the well-definedness along with full faithfulness \cite{KuoLi}*{Section 4.4}. In summary, we have the following proposition.

\begin{proposition}\label{prop:double-para}
    Let $\Lambda \subset T^{*,\infty}_{\tau > 0}(M \times \mathbb{R})$ be any subanalytic Legendrian subset. Then there exists a locally finite open covering $\{U_\alpha\}_{\alpha \in I}$ and an open refinement $\{V_\alpha\}_{\alpha \in I}$, with a collection of $\epsilon_\alpha > 0$ depending on $U_\alpha$ and $V_\alpha$, such that there is a fully faithful functor
    $$\widetilde{w}_\Lambda: \, \mu Sh_\Lambda(\Lambda) \hookrightarrow Sh_{\Lambda_{-T} \cup \Lambda_T}\Big(\bigcup_{\alpha \in I} V_\alpha \times (0, \epsilon_\alpha)\Big),$$
    where $T_- \rightarrow T_+ \rightarrow \widetilde{w}_\Lambda \circ m_\Lambda$ is an exact triangle, viewed as functors from sheaves on $M \times \bR$ to sheaves on $\bigcup_{\alpha \in I} V_\alpha \times (0, \epsilon_\alpha) \subset M \times \bR \times (0, +\infty)$.
\end{proposition}

    Moreover, using the local-to-global argument, we also get an adjunction property following from the exact triangle \cite{KuoLi}*{Section 4.5}. In summary, we have the following proposition.

\begin{proposition}\label{prop:double-ad}
    Let $\Lambda \subset T^{*,\infty}_{\tau > 0}(M \times \mathbb{R})$ be any subanalytic Legendrian subset. Let $\{U_\alpha\}_{\alpha \in I}$, $\{V_\alpha\}_{\alpha \in I}$ and $\{\epsilon_\alpha\}_{\alpha \in I}$ be a collection as above. Then the doubling functor is the left adjoint of microlocalization
    $$m_{\Lambda \times \bR_{>0}}^* = \iota^*_{\Lambda \times \bR_{>0}} \circ \widetilde{w}_\Lambda[-1]: \, \mu Sh_\Lambda(\Lambda) \rightarrow Sh_{\Lambda \times \bR_{>0}}\Big(\bigcup_{\alpha \in I} V_\alpha \times (0, \epsilon_\alpha)\Big),$$
    where $\iota^*_{\Lambda \times \bR_{>0}}$ is the left adjoint to $\iota_{\Lambda \times \bR_{>0}}: Sh_{\Lambda \times \bR_{>0}}(M \times \bR \times \bR_{>0}) \hookrightarrow Sh(M \times \bR \times \bR_{>0})$.
\end{proposition}

    When $U \cap \pi(\Lambda) = \varnothing$, note that we can choose $\mathscr{F}_U = 0_U$ and $\epsilon = +\infty$. When $\Lambda \subset T^{*,\infty}_{\tau > 0}(M \times \mathbb{R})$ is compact, we can choose an open covering such that only finitely many open subsets intersect $\pi(\Lambda) \subset M \times \mathbb{R}$. Then there exists a uniform positive number
    $$0 < \epsilon < \min_{\alpha \in I}\epsilon_\alpha = \min_{\alpha \in I, \, U_\alpha \cap \Lambda \neq \varnothing}\epsilon_\alpha.$$
    In other words, we have
    $$M \times \mathbb{R} \times \{\epsilon\} \subseteq \bigcup_{\alpha \in I} V_\alpha \times (0, \epsilon_\alpha)$$
    Hence by restricting to $M \times \mathbb{R} \times \{\epsilon\}$, there is the following corollary:

\begin{corollary}\label{cor:double-cpt}
    Let $\Lambda \subset T^{*,\infty}_{\tau > 0}(M \times \mathbb{R})$ be a compact subanalytic Legendrian subset. Then there exists a fully faithful functor
    $$w_\Lambda: \, \mu Sh_\Lambda(\Lambda) \hookrightarrow Sh_{T_{-\epsilon}(\Lambda) \cup T_\epsilon(\Lambda)}(M \times \mathbb{R}).$$
\end{corollary}
\begin{proof}
    Since $T_-(\Lambda) \cup T_+(\Lambda)$ is non-characteristic with respect to $M \times \mathbb{R} \times \{\epsilon\}$, we know that by restricting $\widetilde{w}_\Lambda(\mathscr{F})$ to $M \times \mathbb{R} \times \{\epsilon\}$, the singular support is contained in $T_{-\epsilon}(\Lambda) \cup T_\epsilon(\Lambda)$. Full faithfulness follows from the fact that $T_-(\Lambda) \cup T_+(\Lambda)$ is the Legendrian movie of a Hamiltonian flow.
\end{proof}

    In general, there may not exist any $\epsilon > 0$ such that
    $$M \times \mathbb{R} \times \{\epsilon\} \subseteq \bigcup_{\alpha \in I} V_\alpha \times (0, \epsilon_\alpha).$$
    Hence it is difficult to construct the doubling functor for a uniform Reeb pushoff $\epsilon > 0$.

\subsubsection{Doubling for noncompact Legendrians}
    While there may not be a uniform $\epsilon > 0$ so that $M \times \mathbb{R} \times \{\epsilon\} \subseteq \bigcup_{\alpha \in I} V_\alpha \times (0, \epsilon_\alpha)$, there always exists some smooth function $\rho: M \times \mathbb{R} \rightarrow (0, +\infty)$ such that
    $$\mathrm{Graph}(\rho) \subseteq \bigcup_{\alpha \in I} V_\alpha \times (0, \epsilon_\alpha).$$
    Instead of restricting to the slice $M \times \mathbb{R} \times \{\epsilon\}$, we will restrict to the slice $\mathrm{Graph}(\rho)$.

    Let the contact Hamiltonian be $\rho(x, t)$. The induced Hamiltonian vector field is $X_\rho = \rho(x, t)\partial_t -(\partial_x\rho(x, t) + \xi\,\partial_t\rho(x, t))\partial_\xi$, and the Hamiltonian diffeomorphism is
    $$T_\rho(x, t; \xi) = (x, t+\rho(x, t); \xi - \partial_x\rho(x, t) - \xi\,\partial_t\rho(x, t)).$$
    We will show that by restricting the doubled sheaf to the slice $\mathrm{Graph}(\rho)$, we get a sheaf with singular support on $T_{-\rho}(\Lambda) \cup T_\rho(\Lambda)$.

\begin{proposition}\label{prop:double-nonuni}
    Let $\Lambda \subset T^{*,\infty}_{\tau > 0}(M \times \mathbb{R})$ be a subanalytic Legendrian subset. Then for any $C^1$-small smooth function $\rho: M \times \mathbb{R} \rightarrow (0, +\infty)$, there exists a fully faithful functor
    $$w_{\Lambda,\rho}: \, \mu Sh_\Lambda(\Lambda) \hookrightarrow Sh_{T_{-\rho}(\Lambda) \cup T_\rho(\Lambda)}(M \times \mathbb{R}).$$
\end{proposition}
\begin{proof}
    Since $\rho: M \times \mathbb{R} \rightarrow (0, +\infty)$ is $C^1$-small, we know that $T_-(\Lambda) \cup T_+(\Lambda)$ is non-characteristic with respect to $\mathrm{Graph}(\rho) \subseteq M \times \mathbb{R} \times (0, +\infty)$. Hence by restricting  $\widetilde{w}_\Lambda(\mathscr{F})$ to $\mathrm{Graph}(\rho)$, the singular support is contained in $T_{-\rho}(\Lambda) \cup T_\rho(\Lambda)$ by \cite[Proposition 5.4.13]{KS}. Full faithfulness follows from Guillermou-Kashiwara-Schapira Theorem \ref{GKS} \cite{GKS}, using the fact that up to the reparametrization
    $$\phi_\rho: M \times \mathbb{R} \times (0, +\infty) \rightarrow M \times \mathbb{R} \times (0, +\infty), \; (x, t, \epsilon \rho(x, t)) \mapsto (x, t, \epsilon),$$
    $T_-(\Lambda) \cup T_+(\Lambda)$ is the Legendrian movie of a Hamiltonian flow.
\end{proof}

\begin{proposition}\label{prop:double-para-ad}
    Let $\Lambda \subset T^{*,\infty}_{\tau > 0}(M \times \mathbb{R})$ be a subsanalytic Legendrian subset. Then for any $C^1$-small smooth function $\rho: M \times \mathbb{R} \rightarrow (0, +\infty)$, the doubling functor is the left adjoint to the microlocalization
    $$m_\Lambda^* = \iota_\Lambda^* \circ w_{\Lambda,\rho}[-1]: \mu Sh_\Lambda(\Lambda) \rightarrow Sh_\Lambda(M).$$
\end{proposition}
\begin{proof}
    This follows immediately from Proposition \ref{prop:double-ad} and Theorem \ref{GKS} \cite{GKS}.
\end{proof}

    As explained at the beginning, for a general noncompact Legendrian, this is in fact the best we can do. In the following section, we will explain how to strengthen the result with the presence of a tubular neighbourhood of positive radius.

\begin{proof}[Proof of Theorem \ref{thm:double-noncpt}]
    Consider the Weinstein tubular neighbourhood $U_{\epsilon_0}(\Lambda) \cong J^1_{<\epsilon_0}(\Lambda)$ of radius $\epsilon_0 > 0$ of the noncompact Legendrian $\Lambda \subset J^1(M)$ and $\rho: M \times \bR \rightarrow \bR$ sufficiently small so that $T_{\pm \rho}(\Lambda) \subset J^1_{<\epsilon_0}(\Lambda)$. We claim that there is a contact Hamiltonian isotopy $\varphi_H^u, \, 0 \leq u \leq 1$, such that
    $$\varphi_H^1(T_{\pm \rho}(\Lambda)) = T_{\pm \epsilon}(\Lambda).$$
    Consider the standard coordinates in $J^1_{<\epsilon_0}(\Lambda) \subseteq J^1(\Lambda)$. There exists a Legendrian isotopy from $T_{\pm \rho}(\Lambda)$ to $T_{\pm\epsilon}(\Lambda)$
    \[T_{\pm u\epsilon \pm (1-u)\rho}(\Lambda) = \{(x, (1-u)d\rho(x, t), ut + (1-u)\rho(x, t)) | (x, t) \in \Lambda \times \mathbb{R}\}.\]
    We then extend the Legendrian isotopy to a contact Hamiltonian isotopy. Fix $\epsilon < \epsilon' < \epsilon_0$. Define $H: \Lambda \times \mathbb{R} \rightarrow \mathbb{R}$ such that
    \[H|_{\bigcup_{u \in [0,1]}\pi(\Lambda_{\pm u})} = \pm (\epsilon - \rho(x, t)).\]
    Then the corresponding contact Hamiltonian vector field
    \[X_H = H\frac{\partial}{\partial t} + \sum_{i=1}^n\frac{\partial H}{\partial x_i}\frac{\partial}{\partial \xi_i}\]
    is tangent to the fibers/leaves $J_x^1(\Lambda) = T_x^*\Lambda \times \bR$. Since the vector field is bounded on all the fibers, the integral flow is well defined for all time $u \geq 0$ and $\varphi_H^u(T_{\pm \rho}(\Lambda)) = T_{\pm u\epsilon \pm (1-u)\rho}(\Lambda)$. Then we can cut off the Hamiltonian such that
    $$H'|_{J^1_{\leq \epsilon}(\Lambda)} = H|_{J^1_{\leq \epsilon}(\Lambda)}, \; H|_{J^1(\Lambda) \backslash J^1_{<\epsilon'}(\Lambda)} = 0.$$

    The condition that the Hamiltonian $H'$ is supported in $J^1_{<\epsilon'}(\Lambda)$ then allows us to extend the Hamiltonian flow trivially to the ambient manifold $J^1(M)$. Using Guillermou-Kashiwara-Schapira Theorem \ref{GKS} \cite{GKS}, we can then conclude that there is an equivalence
    $$Sh_{T_{-\rho}(\Lambda) \cup T_\rho(\Lambda)}(M \times \mathbb{R}) \xrightarrow{\sim} Sh_{T_{-\epsilon}(\Lambda) \cup T_\epsilon(\Lambda)}(M \times \mathbb{R}).$$
    Then the result follows immediately from Proposition \ref{prop:double-nonuni}.
\end{proof}
\begin{remark}
    Assume that there exists an open covering $\{U_\alpha\}_{\alpha \in I}$ and a refinement of $\{V_\alpha\}_{\alpha \in I}$ of $M \times \mathbb{R}$ with a collection of time intervals $\epsilon_\alpha > 0$ with a uniform lower bound. Then by restricting to small open subsets, one can easily show that the doubling functor constructed using Corollary \ref{cor:double-cpt} agrees with Theorem \ref{thm:double-noncpt}.
\end{remark}

\begin{proof}[Proof of Theorem \ref{thm:double-noncpt-ad}]
    Using Guillermou-Kashiwara-Schapira Theorem \ref{GKS} \cite{GKS}, we can then conclude that there is an equivalence of sheaf categories
    $$Sh_{T_{-\rho}(\Lambda) \cup T_\rho(\Lambda)}(M \times \mathbb{R}) \xrightarrow{\sim} Sh_{T_{-\epsilon}(\Lambda) \cup T_\epsilon(\Lambda)}(M \times \mathbb{R}).$$
    Then the result follows immediately from Proposition \ref{prop:double-para-ad}.
\end{proof}

    There are definitely examples of noncompact Legendrians that do not admit a tubular neighbourhood of a positive radius with respect to any complete adapted metric. Conical Legendrian cobordisms are one class of such examples. This is why we need some extra data. We will deal with them in the next section.

\subsubsection{Separating double copies of the Legendrian}
    Given the doubling on $\Lambda \cup T_\epsilon(\Lambda)$ for some $\epsilon > 0$, to separate the double copies of the Legendrian, we need to apply some Hamiltonian isotopy from $\Lambda \cup T_\epsilon(\Lambda)$ to $\Lambda \cup T_u(\Lambda)$ for any $u > 0$. We show that this can again be done once there exists some tubular neighbourhood of positive radius.

\begin{proposition}\label{prop:separate}
    Let $\Lambda \subset J^1(M)$ be any smooth Legendrian submanifold. When there exists a complete adapted metric on $J^1(M)$ such that $\Lambda$ admits a neighbourhood of positive radius $\epsilon/2 > 0$ disjoint from $\bigcup_{u \geq \epsilon}T_u(\Lambda)$, then for any $u \geq \epsilon$, there exists an equivalence functor
    $$Sh_{\Lambda \cup T_{\epsilon}(\Lambda)}(M \times \mathbb{R}) \xrightarrow{\sim} Sh_{\Lambda \cup T_{u}(\Lambda)}(M \times \mathbb{R}).$$
\end{proposition}
\begin{proof}
    Denote the Weinstein tubular neighbourhood of $\Lambda$ by $U_\epsilon(\Lambda)$. By the assumption, we know that $U_{\epsilon/2}(\Lambda)$ and $\bigcup_{u \geq \epsilon}T_{u}(\Lambda)$ have a positive distance. Choose a cut-off function $H: J^1(M) \rightarrow \mathbb{R}$ such that $$H|_{U_{\epsilon/2}(\Lambda)} \equiv 0, \; H|_{J^1(M) \backslash U_{\epsilon}(\Lambda)} \equiv 1, \; |dH| \leq 3/\epsilon < +\infty.$$
    Since the metric on $J^1(M)$ is adapted, we know that the Hamiltonian vector field $|X_H| \leq (1 + 9/\epsilon^2)^{1/2} < +\infty$. Since in addition that the metric on $J^1(M)$ is complete, we know that the Hamiltonian flow $\varphi_H^u$ exists for any $u \in \mathbb{R}$. Moreover, when $u \geq \epsilon$,
    $$\varphi_H^{u-\epsilon}\big(\Lambda \cup T_{2\epsilon/3}(\Lambda)\big) = \widetilde{L} \cup T_{u}(\widetilde{L}).$$
    Therefore, by Guillermou-Kashiwara-Schapira Theorem \ref{GKS} \cite{GKS}, we can conclude that there is a canonical equivalence as in the statement of the proposition.
\end{proof}
\begin{corollary}
    Let $\Lambda \subset J^1(M)$ be any smooth Legendrian submanifold. When there exists a complete adapted metric on $J^1(M)$ such that $\Lambda$ admits a tubular neighbourhood of positive radius $\epsilon > 0$ disjoint from $\bigcup_{u \geq \epsilon}T_{u}(\Lambda)$, then for any $s > \epsilon$, there exists a fully faithful functor
    $$w_\Lambda: \, \mu Sh_\Lambda(\Lambda) \hookrightarrow Sh_{\Lambda \cup T_{u}(\Lambda)}(M \times \mathbb{R}).$$
    In particular, when the Lagrangian projection $\pi_{Lag}(\Lambda) \subset T^*M$ admits a tubular neighbourhood of positive radius, we always get the above functor.
\end{corollary}

    Using the above results, we can in fact prove the sheaf quantization theorem for certain noncompact embedded Lagrangians that admit a tubular neighbourhood of a positive radius for some adapted metric. In particular, we believe that we can recover the result of Jin-Treumann \cite{JinTreu}. These applications will appear elsewhere.

\subsection{Doubling for conical Legendrian cobordisms}\label{sec:quan-cob}
    We will modify the doubling construction in this section and construct the conditional doubling functor for conical Legendrian cobordisms. The reason we construct a conditional doubling functor with prescribed data at the negative end is exactly because we do not have a Weinstein tubular neighbourhood of positive radius at the negative end.

    Recall that $Sh_\Lambda(M \times \bR)_0$ is the subcategory of sheaves with acyclic stalks at $-\infty$. This will be used frequently in the statements.

\begin{theorem}\label{thm:double-cob}
    Let $\widetilde{L} \subset J^1(M \times \mathbb{R}_{>0})$ be a conical Legendrian cobordism from $\Lambda_- \subset J^1(M)$ to $\Lambda_+ \subset J^1(M)$. Then there exists a fully faithful conditional doubling functor on the subcategory of sheaves
    $$w_{\widetilde{L}}: \, Sh_{\Lambda_-}(M \times \mathbb{R})_0 \times_{\mu Sh_{\Lambda_-}(\Lambda_-)} \mu Sh_{\widetilde{L}}(\widetilde{L}) \xrightarrow{\sim} Sh_{T_{-\epsilon}(\widetilde{L}) \cup T_\epsilon(\widetilde{L})}(M \times \mathbb{R} \times \mathbb{R}_{>0})_0.$$
\end{theorem}

\begin{figure}
  \centering
  \includegraphics[width=0.8\textwidth]{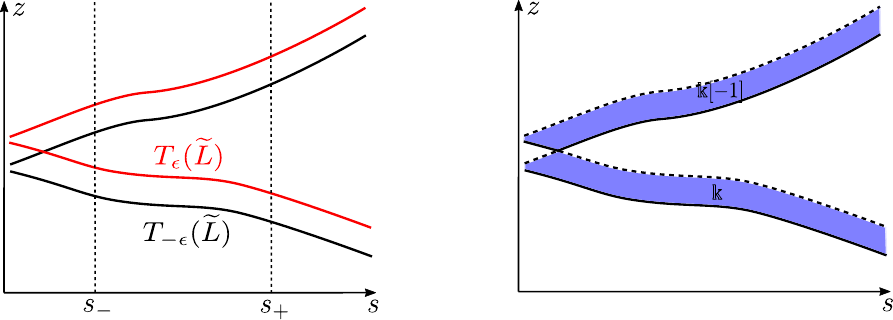}\\
  \caption{The conditional doubling construction which define sheaves in $T_{-\epsilon}(\widetilde{L}) \cup T_\epsilon(\widetilde{L})$ for a conical Legendrian cobordism.}\label{fig:condition-double}
\end{figure}

\subsubsection{Doubling near the negative end}
    First, we construct the doubling functor near the negative end of the Legendrian cobordism, using the information $Sh_{\Lambda_-}(M \times \bR)_0$.

    More precisely, assume that $\widetilde{L} \cap J^1(M \times (0, s_-))$ is conical. Consider $c(\Lambda_-)$ the length of the shortest Reeb chord on $\Lambda_-$. Then for $\epsilon > 0$, pick $s_0 > 0$ such that $\epsilon < s_0c(\Lambda_-) < s_-c(\Lambda_-)$. We show the existence of the doubling functor on the conical end $M \times \mathbb{R} \times (0, s_0)$.

\begin{lemma}\label{lem:cob-compatible}
    Let $\widetilde{L} \subset J^1(M \times \mathbb{R}_{>0})$ be a conical Legendrian cobordism from $\Lambda_- \subset J^1(M)$ to $\Lambda_+ \subset J^1(M)$ that is conical on $J^1(M \times (0, s_-))$ where $\epsilon < s_0c(\Lambda_-) < s_-c(\Lambda_-)$. Then there exists a doubling functor
    $$w_{\widetilde{L}}^{(0,s_-)}: Sh_{\Lambda_-}(M \times \mathbb{R}) \rightarrow Sh_{T_{-\epsilon}(\widetilde{L}) \cup T_\epsilon(\widetilde{L})}(M \times \mathbb{R} \times (0, s_-))$$
    such that $i_{s_0}^{-1}w_{\widetilde{L}}^{(0,s_-)}(\mathscr{F}) = w_{\Lambda_-} \circ m_{\Lambda_-}(\mathscr{F})$.
\end{lemma}
\begin{proof}
    Define the projection $\pi: M \times \mathbb{R} \times (0, t_0) \rightarrow M \times \mathbb{R}$ by $\pi(x, z, t) = (x, z/t)$, and let
    $$w_{\widetilde{L}}^{(0,s_-)} (\mathscr{F}) = \mathrm{Cone}(T_{-\epsilon}(\pi^{-1}\mathscr{F}) \rightarrow T_{\epsilon}(\pi^{-1}\mathscr{F})).$$
    Then it is clear that when $SS^\infty(\mathscr{F}) \subset \Lambda_-$, $SS^\infty\big(w_{\widetilde{L}}^{(0,s_-)} (\mathscr{F})\big) \subset (T_{-\epsilon}(\widetilde{L}) \cup T_\epsilon(\widetilde{L})) \cap T^{*,\infty}(M \times \mathbb{R} \times (0, s_-))$. Finally, the identity
    $$i_{s_0}^{-1}w_{\widetilde{L}}^{(0,s_-)}(\mathscr{F}) = w_{\Lambda_-} \circ m_{\Lambda_-}(\mathscr{F})$$
    follows from the exact triangle of functors $T_{-\epsilon} \rightarrow T_\epsilon \rightarrow w_\Lambda \circ m_\Lambda$.
\end{proof}
\begin{remark}
    We explain why it is necessary to assume the existence of a sheaf $\SF \in Sh_{\Lambda_-}(M \times \bR)_0$ in the construction. In fact, this is the obstruction to go from Proposition \ref{prop:double-nonuni} to Theorem \ref{thm:double-noncpt}. Lack of a Weinstein tubular neighbourhood of $\widetilde{L}$ with positive radius at the negative end makes it difficult to connect $T_{-\rho}(\widetilde{L}) \cup T_\rho(\widetilde{L})$ and $T_{-\epsilon}(\widetilde{L}) \cup T_\epsilon(\widetilde{L})$, following the notation in Proposition \ref{prop:double-nonuni}.

    For example, consider the trivial 1-dimensional conical Legendrian cobordism, by the reparametrization identifying $J^1(\mathrm{pt} \times (0, +\infty))$ with $J^1(\mathrm{pt} \times \mathbb{R})$ as in Figure \ref{fig:leg-nbhd} (middle), one may assume that
    $$\widetilde{L} = \{(s, \pm e^s, \pm e^s) | s \in \mathbb{R}\}.$$
    One can easily check that $\inf_{x, x' \in \widetilde{L}}d(x, x') = 0$ under the standard complete adapted metric. Then lack of a Weinstein tubular neighbourhood of positive radius with respect to the standard metric will prevent the Legendrian isotopy between $T_{-\rho}(\widetilde{L}) \cup T_\rho(\widetilde{L})$ and $T_{-\epsilon}(\widetilde{L}) \cup T_\epsilon(\widetilde{L})$ from being extended to a Hamiltonian isotopy.
\end{remark}

    Next, we prove full faithfulness of the functor near negative end.

\begin{lemma}\label{lem:negative-0}
    Let $\widetilde{L} \subset J^1(M \times \mathbb{R}_{>0})$ be a conical Legendrian cobordism from $\Lambda_- \subset J^1(M)$ to $\Lambda_+ \subset J^1(M)$ that is conical on $J^1(M \times (0, s_-))$. Then for $\SF, \SG \in Sh_{\Lambda_-}(M \times \bR)_0$,
    $$Hom(T_{\epsilon}(\pi^{-1}\mathscr{F}), T_{-\epsilon}(\pi^{-1}\mathscr{G})) \simeq 0.$$
\end{lemma}
\begin{proof}
    Since $SS^\infty(T_{-\epsilon}(\pi^{-1}\mathscr{F})) \cap SS^\infty(T_{\epsilon}(\pi^{-1}\mathscr{G})) = \varnothing$, by \cite{KS}*{Proposition 5.4.14}, we know that
    $$SS^\infty\big(\mathscr{H}om(T_{\epsilon}(\pi^{-1}\mathscr{F}), T_{-\epsilon}(\pi^{-1}\mathscr{G})) \big) \subset -T_{\epsilon}(\Lambda \times (0, s_-)) + T_{-\epsilon}(\Lambda_- \times (0, s_-)).$$
    Therefore, $(x, z, s; y, 0, \sigma) \in SS^\infty(\mathscr{H}om(T_{\epsilon}(\pi^{-1}\mathscr{F}), T_{-\epsilon}(\pi^{-1}\mathscr{F})) )$ means there are points $(x, t; \xi, 1)$ and $(x, t'; \xi, 1) \in \Lambda_-$ such that $(x, z, s) = (x, st + \epsilon, s) = (x, st' - \epsilon, s)$ and
    $$(x, z, s; y, 0, \sigma) = -(x, st + \epsilon, s; s\xi, 1, t) + (x, st' - \epsilon, s; s\xi, 1, t').$$
    This implies $s(t' - t) = 2\epsilon$ and thus $\sigma = t' - t > 0$. By microlocal Morse lemma \cite{KS}*{Corollary 5.4.19}, we know that
    $$Hom(T_{\epsilon}(\pi^{-1}\mathscr{F}), T_{-\epsilon}(\pi^{-1}\mathscr{G})) = Hom(T_{\epsilon}(\pi^{-1}\mathscr{F})|_{M \times \bR \times (0, s_-')}, T_{-\epsilon}(\pi^{-1}\mathscr{G})|_{M \times \bR \times (0, s_-')}).$$
    Write $\pi^{-1}\mathscr{F}|_{M \times \bR (0, s_-')} = \pi^{-1}\mathscr{F}|_{(0, s_-')}$. For $s_-' < s_-$ sufficiently small, since $\pi^{-1}\mathscr{F}|_{(0, s_-')}$ has acyclic stalk at $-\infty$, $T_{-\epsilon}(\pi^{-1}\mathscr{G})|_{(0, s_-')}$ is a local system on $\mathrm{supp}(T_{\epsilon}(\pi^{-1}\mathscr{G})|_{(0, s_-')})$. Then we know that
    $$SS^\infty(\mathscr{H}om(T_{\epsilon}(\pi^{-1}\mathscr{F})|_{(0, s_-')}, T_{-\epsilon}(\pi^{-1}\mathscr{G})|_{(0, s_-')})) \subset - SS^\infty(T_{\epsilon}(\pi^{-1}\mathscr{F})|_{(0, s_-')})$$
    consisting of points $(x, t, s; y, \tau, \sigma)$ such that $\tau < 0$. Therefore, by microlocal Morse lemma again,
    \begin{equation*}
    Hom(T_{\epsilon}(\pi^{-1}\mathscr{F})|_{(0, s_-')}, T_{-\epsilon}(\pi^{-1}\mathscr{G})|_{(0, s_-')}) = 0.
    \end{equation*}
    This proves our claim.
\end{proof}

\begin{lemma}\label{lem:negative-hom}
    Let $\widetilde{L} \subset J^1(M \times \mathbb{R}_{>0})$ be a conical Legendrian cobordism from $\Lambda_- \subset J^1(M)$ to $\Lambda_+ \subset J^1(M)$ that is conical on $J^1(M \times (0, s_-))$ where $\epsilon < s_-c(\Lambda_-)$. Then
    $$Hom(T_{-\epsilon}(\pi^{-1}\mathscr{F}), T_{\epsilon}(\pi^{-1}\mathscr{G})) \simeq Hom(\mathscr{F, G}).$$
\end{lemma}
\begin{proof}
    Similarly, we know $(x, z, s; y, 0, \sigma) \in SS^\infty(\mathscr{H}om(T_{-\epsilon}(\pi^{-1}\mathscr{F}), T_{\epsilon}(\pi^{-1}\mathscr{F})) )$ means that there are points $(x, t; \xi, 1)$ and $(x, t'; \xi, 1) \in \Lambda_-$ such that $(x, z, s) = (x, st - \epsilon, s) = (x, st' + \epsilon, s)$ and
    $$(x, z, s; y, 0, \sigma) = (x, st - \epsilon, s; s\xi, 1, t) + (x, st' + \epsilon, s; s\xi, 1, t').$$
    In other words, $s(t' - t) = -2\epsilon$ and thus $\sigma = t' - t < 0$. By microlocal Morse lemma \cite{KS}*{Corollary 5.4.19}, we know that
    $$Hom(T_{-\epsilon}(\pi^{-1}\mathscr{F}), T_{\epsilon}(\pi^{-1}\mathscr{G})) = Hom(T_{-\epsilon}(\pi^{-1}\mathscr{F})|_{M \times \bR \times (s_-', s_-)}, T_{\epsilon}(\pi^{-1}\mathscr{G})|_{M \times \bR \times (s_-', s_-)}).$$
    Then for $s_-' < s_-$ sufficiently close, the Legendrian cones $T_{-\epsilon}(\Lambda_- \times (s_-', s_-))$ and $T_\epsilon(\Lambda_- \times (s_-', s_-))$ are Legendrian movies of a Legendrian isotopy. Hence by Guillermou-Kashiwara-Schapira Theorem \ref{GKS} and Lemma \ref{lem:small-pushoff}
    \[\begin{split}
    Hom(& T_{-\epsilon}( \pi^{-1}\mathscr{F})|_{(s_-', s_-)}, T_{\epsilon}(\pi^{-1}\mathscr{G})|_{(s_-', s_-)}) \\
    &= Hom(T_{-\epsilon/ s_-}\SF, T_{\epsilon/ s_-}\SG) = Hom(\SF, \SG).
    \end{split}\]
    This proves our claim.
\end{proof}

\begin{proposition}\label{prop:double-negative-ff}
    Let $\widetilde{L} \subset J^1(M \times \mathbb{R}_{>0})$ be a conical Legendrian cobordism from $\Lambda_- \subset J^1(M)$ to $\Lambda_+ \subset J^1(M)$ that is conical on $J^1(M \times (0, s_-))$ where $\epsilon < s_-c(\Lambda_-)$. Then the doubling functor is fully faithful
    $$w_{\widetilde{L}}^{(0,s_-)}: Sh_{\Lambda_-}(M \times \mathbb{R})_0 \hookrightarrow Sh_{T_{-\epsilon}(\widetilde{L}) \cup T_\epsilon(\widetilde{L})}(M \times \mathbb{R} \times (0, s_-))_0.$$
\end{proposition}
\begin{proof}
    It suffices to show that for any $\SF, \SG \in Sh_{\Lambda_-}(M \times \bR)_0$ with acyclic stalks at $-\infty$,
    $$Hom(\mathrm{Cone}(T_{-\epsilon}(\pi^{-1}\mathscr{F}) \rightarrow T_{\epsilon}(\pi^{-1}\mathscr{F})), \mathrm{Cone}(T_{-\epsilon}(\pi^{-1}\mathscr{G}) \rightarrow T_{\epsilon}(\pi^{-1}\mathscr{G}))) \simeq Hom(\SF, \SG).$$

    First, we prove that
    $$Hom(T_{\epsilon}(\pi^{-1}\mathscr{F}), \mathrm{Cone}(T_{-\epsilon}(\pi^{-1}\mathscr{G}) \rightarrow T_{\epsilon}(\pi^{-1}\mathscr{G}))) = Hom(\SF, \SG).$$
    By Lemma \ref{lem:negative-0} we know that for $s_-' < s_-$ sufficiently small,
    $$Hom(T_{\epsilon}(\pi^{-1}\mathscr{F}), T_{-\epsilon}(\pi^{-1}\mathscr{G})) = Hom(T_{\epsilon}(\pi^{-1}\mathscr{F})|_{(0, s_-')}, T_{-\epsilon}(\pi^{-1}\mathscr{G})|_{(0, s_-')}) = 0.$$
    On the other hand, we know by Guillermou-Kashiwara-Schapira Theorem \ref{GKS} \cite{GKS} that
    $$Hom(T_{\epsilon}(\pi^{-1}\mathscr{F}), T_{\epsilon}(\pi^{-1}\mathscr{G})) = Hom(T_{\epsilon}(\pi^{-1}\mathscr{F})|_{(0, s_-')}, T_{\epsilon}(\pi^{-1}\mathscr{G})|_{(0, s_-')}) = Hom(\SF, \SG).$$
    Since the natural continuation map given by the Reeb flow $T_t$ factors through the restrictions to $M \times \bR \times (0, s_-')$. This implies that
    $$Hom(T_{\epsilon}(\pi^{-1}\mathscr{F}), \mathrm{Cone}(T_{-\epsilon}(\pi^{-1}\mathscr{G}) \rightarrow T_{\epsilon}(\pi^{-1}\mathscr{G}))) = Hom(\SF, \SG).$$

    Next, we prove that
    $$Hom(T_{-\epsilon}(\pi^{-1}\mathscr{F}), \mathrm{Cone}(T_{-\epsilon}(\pi^{-1}\mathscr{G}) \rightarrow T_{\epsilon}(\pi^{-1}\mathscr{G}))) = 0.$$
    By Lemma \ref{lem:negative-hom}, we know that for $s_-' < s_-$ sufficiently close to each other,
    $$Hom(T_{-\epsilon}(\pi^{-1}\mathscr{F}), T_{\epsilon}(\pi^{-1}\mathscr{G})) = Hom(T_{-\epsilon}( \pi^{-1}\mathscr{F})|_{(s_-', s_-)}, T_{\epsilon}(\pi^{-1}\mathscr{G})|_{(s_-', s_-)}) = Hom(\SF, \SG).$$
    On the other hand, we know by Theorem \ref{GKS} that
    $$Hom(T_{-\epsilon}(\pi^{-1}\mathscr{F}), T_{-\epsilon}(\pi^{-1}\mathscr{G})) = Hom(T_{-\epsilon}( \pi^{-1}\mathscr{F})|_{(s_-', s_-)}, T_{-\epsilon}(\pi^{-1}\mathscr{G})|_{(s_-', s_-)}) = Hom(\SF, \SG)$$
    Since the natural continuation map given by the Reeb flow $T_t$ factors through the restrictions to $M \times \bR \times (s_-', s_-)$. This implies that
    $$Hom(T_{-\epsilon}(\pi^{-1}\mathscr{F}), \mathrm{Cone}(T_{-\epsilon}(\pi^{-1}\mathscr{G}) \rightarrow T_{\epsilon}(\pi^{-1}\mathscr{G}))) = 0.$$
    Combining the two equalities, we can therefore conclude that the doubling functor is fully faithful near the negative end.
\end{proof}

\begin{proposition}\label{prop:double-negative-ad}
    Let $\widetilde{L} \subset J^1(M \times \mathbb{R}_{>0})$ be a conical Legendrian cobordism from $\Lambda_- \subset J^1(M)$ to $\Lambda_+ \subset J^1(M)$ that is conical on $J^1(M \times (0, s_-))$. Then $\iota^*_{\widetilde{L}} \circ w_{\widetilde{L}}^{(0,s_-)}[-1]$ is the left adjoint of the restriction
    $$i_-^{-1}: Sh_{\widetilde{L}}(M \times \bR \times (0, s_-))_0 \hookrightarrow Sh_{\Lambda_-}(M \times \bR).$$
\end{proposition}
\begin{proof}
    It suffices to show that for any $\SF, \SG \in Sh_{\Lambda_-}(M \times \bR)_0$ with acyclic stalks at $-\infty$
    $$Hom(\mathrm{Cone}(T_{-\epsilon}(\pi^{-1}\mathscr{F}) \rightarrow T_{\epsilon}(\pi^{-1}\mathscr{F})), \pi^{-1}\mathscr{G}) \simeq Hom(\SF, \SG).$$
    This follows from the Lemma \ref{lem:negative-0} and \ref{lem:negative-hom}.
\end{proof}

\subsubsection{Doubling away from the negative end}
    We construct the doubling functor away from the negative end using the doubling construction with some uniform Reeb pushoff.

    In fact, by Lemma \ref{lem:nbhd-sft-cob-2}, there exists a complete adapted metric on $J^1(M \times (s_0, +\infty))$ such that $\widetilde{L} \cap J^1(M \times (s_0, +\infty))$ admits a tubular neighbourhood of positive radius. Therefore, we get a fully faithful doubling functor
    $$w_{\widetilde{L}}^{(s_0,+\infty)}: \mu Sh_{\widetilde{L}}(\widetilde{L}) \rightarrow Sh_{T_{-\epsilon}(\widetilde{L}) \cup T_\epsilon(\widetilde{L})}(M \times \mathbb{R} \times (s_0, +\infty)).$$

\begin{lemma}
    Let $\widetilde{L} \subset J^1(M \times \mathbb{R}_{>0})$ be a conical Legendrian cobordism from $\Lambda_- \subset J^1(M)$ to $\Lambda_+ \subset J^1(M)$ that is conical on $J^1(M \times (0, s_-))$ where $\epsilon < s_0c(\Lambda_-) < s_-c(\Lambda_-)$. Given $\mathscr{F} \in Sh_{\Lambda_-}(M \times \mathbb{R})$ and $\mathscr{L} \in \mu Sh_{\widetilde{L}}(\widetilde{L})$, suppose
    $$m_{\Lambda_-}(\mathscr{F}) = i_{\Lambda_-}^{-1}\mathscr{L}.$$
    Then for the inclusions $i_{(s_0,s_-)}^-: M \times \mathbb{R} \times (s_0, s_-) \hookrightarrow M \times \mathbb{R} \times (0, s_-)$ and $i_{(s_0,s_-)}^+: M \times \mathbb{R} \times (s_0, s_-) \hookrightarrow M \times \mathbb{R} \times (s_0, +\infty)$, we have a canonical isomorphism
    $$(i_{(s_0,s_-)}^-)^{-1}w_{\widetilde{L}}^{(0,s_-)}(\mathscr{F}) = (i_{(s_0,s_-)}^+)^{-1}w_{\widetilde{L}}^{(s_0,+\infty)}(\mathscr{L}).$$
\end{lemma}
\begin{proof}
    Since $\widetilde{L}$ is conical on $J^1(M \times (s_0, s_-))$, there is a canonical equivalence by Guillermou-Kashiwara-Schapira Theorem \ref{GKS} \cite{GKS} that
    $$Sh_{\Lambda_-}(M \times \mathbb{R}) \xrightarrow{\sim} Sh_{\widetilde{L}}(M \times \mathbb{R} \times (s_0, s_-)),$$
    it suffices to show that for any $s_1 \in (s_0, s_-)$ and the corresponding inclusions $i_{s_1}^-: M \times \mathbb{R} \times \{s_1\} \hookrightarrow M \times \mathbb{R} \times (0, s_-)$ and $i_{s_1}^+: M \times \mathbb{R} \times \{s_1\} \hookrightarrow M \times \mathbb{R} \times (s_0, +\infty)$, there is an isomorphism
    $$(i_{s_1}^-)^{-1}w_{\widetilde{L}}^{(0,s_-)}(\mathscr{F}) = (i_{s_1}^+)^{-1}w_{\widetilde{L}}^{(s_0,+\infty)}(\mathscr{L}).$$
    On the other hand, as our construction of $w_{\widetilde{L}}^{(s_0,+\infty)}$ is local with respect to small open subsets $U \times I \times J \subset M \times \mathbb{R} \times (s_0, s_-)$, it is compatible with the construction on a single slice $w_{\Lambda_-}$ and there is an obvious isomorphism that
    $$(i_{s_1}^+)^{-1}w_{\widetilde{L}}^{(s_0,+\infty)}(\mathscr{L}) = w_{\Lambda_-}(i_{\Lambda_-}^{-1}\mathscr{L}) = w_{\Lambda_-} \circ m_{\Lambda_-}(\mathscr{F}).$$
    Then the isomorphism follows from Lemma \ref{lem:cob-compatible} that $i_{s_1}^{-1}w_{\widetilde{L}}^{(0,s_-)}(\mathscr{F}) = w_{\Lambda_-} \circ m_{\Lambda_-}(\mathscr{F})$.
\end{proof}
\begin{remark}
    Instead of defining the doubling functor for the Legendrian cobordism in $J^1(M \times \mathbb{R}_{>0})$, one may consider defining the doubling functor on the truncated cobordism in $J^1(M \times (s_0, +\infty))$, where a Weinstein neighbourhood of positive radius always exist. However, then there will be obstruction to apply Reeb pushoff to send one copy of the Legendrian to infinity, see Remark \ref{rem:obstruct-cob}. As we will explain, actually both the construction of doubling and the construction of pushing one copy to infinity come down to the question about the tubular neighbourhoods.
\end{remark}

    Using the above lemma, we can prove the well-definedness of the doubling functor in Theorem \ref{thm:double-cob} which is
    $$w_{\widetilde{L}}: \, Sh_{\Lambda_-}(M \times \mathbb{R}) \times_{\mu Sh_{\Lambda_-}(\Lambda_-)} \mu Sh_{\widetilde{L}}(\widetilde{L}) \hookrightarrow Sh_{T_{-\epsilon}(\widetilde{L}) \cup T_\epsilon(\widetilde{L})}(M \times \mathbb{R} \times \mathbb{R}_{>0}).$$
    Next, we will need to address the full faithfulness of the doubling functor. This follows immediately from full faithfulness on the negative end Proposition \ref{prop:double-negative-ff} and full faithfulness away from the negative end Theorem \ref{thm:double-noncpt}.

    Moreover, we can show the adjunction property in the following proposition.

\begin{theorem}\label{thm:double-cob-ad}
    Let $\widetilde{L} \subset J^1(M \times \mathbb{R}_{>0})$ be a conical Legendrian cobordism from $\Lambda_- \subset J^1(M)$ to $\Lambda_+ \subset J^1(M)$. Then $\iota^*_{\widetilde{L}} \circ w_{\widetilde{L}}[-1]$ is the left adjoint of the functor
    $$(i_-, m_{\widetilde{L}}): Sh_{\widetilde{L}}(M \times \bR \times (0, +\infty))_0 \rightarrow Sh_{\Lambda_-}(M \times \bR)_0 \times_{\mu Sh_{\Lambda_-}(\Lambda_-)} \mu Sh_{\widetilde{L}}(\widetilde{L}).$$
\end{theorem}
\begin{proof}
    On the negative end, for any $\SF \in Sh_{\Lambda_-}(M \times \bR)_0$ and $\SG \in Sh_{\widetilde{L}}(M \times \bR \times (0, s_-))_0$, by Proposition \ref{prop:double-negative-ad}, we know that
    $$Hom(w_{\widetilde{L}}^{(0, s_-)}(\SF), \SG) = Hom(\SF, i_-^{-1}\SG).$$
    Away from the negative end, for any $\SL \in \mu Sh_{\widetilde{L}}(\widetilde{L})$ and $\SG \in Sh_{\widetilde{L}}(M \times \bR \times (s_0, +\infty))_0$, by Theorem \ref{thm:double-noncpt-ad}, we know that
    $$Hom(w_{\widetilde{L}}^{(s_0, +\infty)}(\SL), \SG) = \Gamma(\widetilde{L}, \mu hom(\SL, m_{\widetilde{L}}(\SG))).$$
    On the overlap region, suppose $m_{\Lambda_-}(\SF) = i_-^{-1}\SL$. We also have
    $$Hom(w_{\Lambda_-}(\SL), i_-^{-1}\SG) = \Gamma(\Lambda_-, \mu hom(i_-^{-1}\SL, m_{\Lambda_-}(i_-^{-1}\SG))).$$
    Therefore, we can conclude that globally the adjunction holds.
\end{proof}

\subsection{Sheaf quantization of Lagrangian cobordisms}
    Given the doubling construction in the previous section, now we will pull the double copies $\widetilde{L}$ and $T_{\epsilon}(\widetilde{L})$ apart and get the sheaf quantization functor for the conical Legendrian cobordism $\widetilde{L}$. In addition, at the end of this section, we will explain why there is no unconditional sheaf quantization functor
    $$\Psi_{{L}}: \, \mu Sh_{\widetilde{L}}(\widetilde{L}) \hookrightarrow Sh_{\widetilde{L}}(M \times \mathbb{R} \times \mathbb{R}_{>0})$$
    even though $\widetilde{L}$ is a Legendrian with no Reeb chords.

\subsubsection{Separation of double copied Legendrians}
    First, we apply a contact Hamiltonian flow to separate the front projections of the double copies of the Legendrian.

\begin{proposition}\label{prop:sep-double-cob}
    Let $\widetilde{L} \subset J^1(M \times \mathbb{R}_{>0})$ be a conical Legendrian cobordism with no Reeb chords from $\Lambda_-$ to $\Lambda_+ \subset J^1(M)$. Then for any $u \geq \epsilon > 0$, there is a canonical equivalence
    $$ Sh_{\widetilde{L} \cup T_{\epsilon}(\widetilde{L})}(M \times \mathbb{R} \times \mathbb{R}_{>0}) \xrightarrow{\sim} Sh_{\widetilde{L} \cup T_{u}(\widetilde{L})}(M \times \mathbb{R} \times \mathbb{R}_{>0}).$$
\end{proposition}
\begin{proof}
    Following Proposition \ref{prop:separate}, we need to show that there exists some neighbourhood $U_{\epsilon'}(\widetilde{L})$ of $\widetilde{L}$ that is disjoint from $\bigcup_{u \geq \epsilon}T_{u}(\widetilde{L})$. Without loss of generality, we assume that $\max_{(x, \xi, t)\in \Lambda_-}t > 0$ while $\min_{(x, \xi, t) \in \Lambda_+}t < 0$.

    First, we find a neighbourhood of the negative end of $\widetilde{L}$. Recall $h(\Lambda_-) = \max_{(x, \xi, t)\in \Lambda_-}t - \min_{(x, \xi, t) \in \Lambda_+}t$. Assume that $\widetilde{L}$ is conical on $J^1(M \times (0, s_0))$, and moreover fix $e^{\epsilon/2}s_0h(\Lambda_-) < \epsilon/2$. Then there exists a neighbourhood of $\widetilde{L} \cap J^1(M \times (0, s_0))$ of radius $\epsilon/2$ that is disjoint from $\bigcup_{u \geq \epsilon}T_u(\widetilde{L})$. This is because for points in $\bigcup_{u \geq \epsilon}T_u(\widetilde{L}) \cap J^1(M \times (0, e^{\epsilon/2}s_0))$, the distance between the $z$ coordinates is at least $\epsilon/2$, while for the other points, the distance between the $s$ coordinates (under the Riemannian metric $s^{-2}ds^2$) is at least $\epsilon/2$.

    Then, we find a tubular neighbourhood away from the negative end of $\widetilde{L}$. In fact, by Lemma \ref{lem:nbhd-sft-cob-1}, the Lagrangian projection $L \cap T^*(M \times (s_0, +\infty))$ admits a tubular neighbourhood of positive radius $\epsilon_1$. Therefore, following Lemma \ref{lem:nbhd-symp-cont}, consider a tubular neighbourhood of $\widetilde{L} \cap J^1(M \times (s_0, +\infty))$ of radius $\epsilon' = \min(\epsilon/2, \epsilon_1)$. Then $\widetilde{L} \cap J^1(M \times (s_0, +\infty))$ is disjoint from $\bigcup_{u \geq \epsilon}T_u(\widetilde{L})$. Indeed, for points in $\bigcup_{u \geq \epsilon}T_u(\widetilde{L})$ contained in either $J^1(M \times (0, e^{\epsilon/2}s_0))$ or $J^1(M \times (s_0, +\infty))$, the distance between the $z$ coordinates is at least $\epsilon'$.
    By considering the union of the two neighbourhoods, we complete the proof.
\end{proof}

\begin{remark}\label{rem:obstruct-cob}
    Suppose one starts by working on $J^1(M \times (s_0, +\infty))$ where the doubling construction exists for some uniform $\epsilon > 0$. Then there will be serious difficulty when one tries to pushoff one of the copies by the Reeb flow. This is because by choosing a complete adapted metric on $J^1(M \times (s_0, +\infty))$, different from the restriction of the one on $J^1(M \times \mathbb{R}_{>0})$, the negative end becomes asymptotically horizontal and there will no longer be a tubular neighbourhood of $\Lambda$ with positive radius that is disjoint from $\bigcup_{u \geq \epsilon}T_u(\Lambda)$.

    For example, consider the trivial 1-dimensional conical Legendrian cobordism again as in Figure \ref{fig:leg-nbhd} (right), by the reparametrization identifying $J^1(\mathrm{pt} \times (1, +\infty))$ with $J^1(\mathrm{pt} \times \mathbb{R})$, one may assume that
    $$\widetilde{L} = \{(s, \pm e^s, \pm t_0/2 \pm e^s) | s \in \mathbb{R}\}.$$
    One can easily check that $\inf_{x, x' \in \widetilde{L}}d(x, T_{t_0}(x')) = 0$ under the standard complete adapted metric. Lack of control on the tubular neighbourhood will forbid us to connect the obvious Legendrian isotopy from $\widetilde{L} \cup T_{\epsilon}(\widetilde{L})$ to $\widetilde{L} \cup T_{t_0 + \epsilon}(\widetilde{L})$ by a Hamiltonian isotopy.
\end{remark}

\subsubsection{Fully faithful sheaf quantization functor}
    Based on the proposition, we can finish the proof of Theorem \ref{thm:cond-quan-cob}. Let $t_{\max}(\Lambda_+) = \max_{(x,\xi,t)\in \Lambda_+}t$, $t_{\min}(\Lambda_-) = \min_{(x,\xi,t)\in \Lambda_+}t$ and the height $h(\Lambda_+) = t_{\max}(\Lambda_+) - t_{\min}(\Lambda_+)$. Suppose that $\widetilde{L}$ is conical on $J^1(M \times (s_+, +\infty))$. We choose $\epsilon' > 0$ and $s_+' > 0$ such that $s_+h(\Lambda_+) < s_+'h(\Lambda_+) < \epsilon'$.

\begin{proof}[Proof of Theorem \ref{thm:cond-quan-cob}]
    Using Theorem \ref{thm:double-cob} and Proposition \ref{prop:sep-double-cob}, we know that there exists a doubling functor (fully faithful on the subcategory of sheaves with acyclic stalk at $-\infty$)
    $$w_{\widetilde{L}}': \, Sh_{\Lambda_-}(M \times \mathbb{R})_0 \times_{\mu Sh_{\Lambda_-}(\Lambda_-)} \mu Sh_{\widetilde{L}}(\widetilde{L}) \hookrightarrow Sh_{\widetilde{L} \cup T_{\epsilon'}(\widetilde{L})}(M \times \mathbb{R} \times \mathbb{R}_{>0})_0.$$
    Then by restricting to $M \times (-\infty, s_+' + s_+'t_{\max}(\Lambda_+)) \times (0, s_+')$, we get a functor
    $$Sh_{\Lambda_-}(M \times \mathbb{R})_0 \times_{\mu Sh_{\Lambda_-}(\Lambda_-)} \mu Sh_{\widetilde{L}}(\widetilde{L}) \hookrightarrow Sh_{\widetilde{L}}(M \times (-\infty, s_+' + s_+'t_{\max}(\Lambda_+)) \times (0, s_+'))_0.$$
    Choose a diffeomorphism $\varphi: M \times (-\infty, s_+' + s_+'t_{\max}(\Lambda_+)) \times (0, s_+') \xrightarrow{\sim} M \times \mathbb{R} \times (0, s_+')$ that is the identity on $M \times (-\infty, s_+'/2 + s_+'t_{\max}(\Lambda_+)) \times (0, s_+')$. We will get the first equivalence
    $$Sh_{\widetilde{L}}(M \times (-\infty, s_+' + s_+'t_{\max}(\Lambda_+)) \times (0, s_+')) \xrightarrow{\sim} Sh_{\widetilde{L}}(M \times \mathbb{R} \times (0, s_+')).$$
    Then since $\widetilde{L}$ is conical on $J^1(M \times (s_+, +\infty))$, consider the equivalence by Guillermou-Kashiwara-Schapira Theorem \ref{GKS} \cite{GKS} that
    $$Sh_{\widetilde{L}}(M \times \mathbb{R} \times (s_+, s_+')) \xrightarrow{\sim} Sh_{\widetilde{L}}(M \times \mathbb{R} \times (s_+, +\infty)).$$
    We will therefore get the second equivalence
    $$Sh_{\widetilde{L}}(M \times \mathbb{R} \times (0, s_+')) \xrightarrow{\sim} Sh_{\widetilde{L}}(M \times \mathbb{R} \times (0, +\infty)).$$
    Therefore, combing the first equivalence and the second equivalence, we can conclude that there exists a conditional sheaf quantization functor.

    Now, we show that the sheaf quantization functor is fully faithful when restricted to the subcategory $Sh_{\Lambda_-}(M \times \bR)_0$. Let $j: M \times (-\infty, s_+' + s_+'t_{\max}(\Lambda_+)) \times (0, s_+') \hookrightarrow M \times \bR \times \bR_{>0}$ and $\varphi: M \times (-\infty, s_+' + s_+'t_{\max}(\Lambda_+)) \times (0, s_+') \xrightarrow{\sim} M \times \mathbb{R} \times (0, s_+')$ be the diffeomorphism. By Theorem \ref{thm:double-cob-ad}, we have
    \[\begin{split}
    Hom(\Psi_L(\SF, \SL), \Psi_L(\SF, \SL)) &= Hom(\varphi^{-1}j^{-1}w_{\widetilde{L}}'(\SF, \SL), \varphi^{-1}j^{-1}w_{\widetilde{L}}'(\SF, \SL)) \\
    &= Hom(w_{\widetilde{L}}'(\SF, \SL), j_*\varphi_*\varphi^{-1}j^{-1}w_{\widetilde{L}}'(\SF, \SL)) \\
    &\simeq Hom((\SF, \SL), (i_{\Lambda_-}^{-1}, m_{\widetilde{L}})(w_{\widetilde{L}}'(\SF, \SL))) \\
    &= Hom((\SF, \SL), (\SF, \SL)).
    \end{split}\]
    This concludes the proof of the full faithfulness property.

    Finally, it suffices to prove essential surjectivity in order to conclude that this is an equivalence. In fact, for any $\SF \in Sh_{\widetilde{L}}(M \times \bR)_0$, we can easily show that $\Psi_L(i_-^{-1}\SF, m_{\widetilde{L}}(\SF)) = \SF$. Actually, when constructing the doubling $w_{\widetilde{L}}(i_-^{-1}\SF, m_{\widetilde{L}}(\SF))$, we have
    $$w_{\widetilde{L}}(i_-^{-1}\SF, m_{\widetilde{L}}(\SF)) = \mathrm{Cone}(T_{-\epsilon}(\SF) \rightarrow T_{\epsilon}(\SF)).$$
    At the negative end, this identity follows from the definition, while away from the negative end, this identity follows from the definition and the exact triangle of functors in Theorem \ref{thm:double-noncpt}.
\end{proof}

    Moreover, we are able to prove that the restriction to the positive end is also fully faithful (once we restrict to the subcategory of sheaves with acyclic stalks at $-\infty$).

\begin{proposition}\label{prop:restrict-ff}
    Let $\widetilde{L} \subset J^1(M \times \mathbb{R}_{>0})$ be a conical Legendrian cobordism with no Reeb chords from $\Lambda_-$ to $\Lambda_+ \subset J^1(M)$. Then the restriction functor to positive end is fully faithful
    $$i_+^{-1}: Sh_{\widetilde{L}}(M \times \bR \times \bR_{>0})_0 \hookrightarrow Sh_{\Lambda_+}(M \times \bR)_0.$$
\end{proposition}
\begin{proof}
    Let $\SF, \SG \in Sh_{\widetilde{L}}(M \times \bR \times \bR_{>0})_0$. For any $u > 0$, by Lemma \ref{lem:small-pushoff} and Proposition \ref{prop:sep-double-cob}, we know that
    $$Hom(\SF, \SG) = Hom(\SF, T_u(\SG)).$$
    Since $\widetilde{L}$ is proper with respect to the projection $J^1(M \times \bR_{>0}) \rightarrow \bR_{>0}$, choose $u > 0$ be sufficiently large so that $\pi(\widetilde{L}) \cap \pi(T_u(\widetilde{L})) \subset J^1(M \times (s_+, +\infty))$. Then following Lemma \ref{lem:negative-0}, for $s_+ > 0$ sufficiently large, $\SF|_{(0, s_+)}$ is a local system on $\mathrm{supp}(T_u(\SG)|_{(0, s_+)})$, so
    $$SS^\infty(\mathscr{H} om(\SF|_{(0, s_+)}, T_u(\SG)|_{(0, s_+)})) \subset -SS^\infty(T_u(\SG)|_{(0, s_+)})$$
    consisting of points $(x, t, s; y, \tau, \sigma)$ such that $\tau < 0$. Therefore, by microlcoal Morse lemma \cite{KS}*{Corollary 5.4.19}
    $$Hom(\SF|_{(0, s_+)}, T_u(\SG)|_{(0, s_+)}) \simeq 0.$$
    On the other hand, by Lemma \ref{lem:negative-hom}, we know that for $0 < s_+' < s_+$ sufficiently large,
    $$Hom(\SF|_{(s_+', +\infty)}, T_u(\SG)|_{(s_+', +\infty)}) \simeq Hom(\SF, T_{u/s_+'}(\SG)) = Hom(\SF, \SG).$$
    Therefore, we can finish the proof.
\end{proof}

\subsubsection{Remark on general noncompact Legendrians}\label{sec:remark}
    We explain how our construction may be used to obtain the sheaf quantization for general noncompact Legendrians.

    When constructing and pushing off sheaves with singular support on the doubling $\Lambda \cup T_{u}(\Lambda)$, we have seen that the obstruction to obtain the doubling at $u \geq 0$ is either when $u$ is the length of the Reeb chords or when $u$ is the time that the distance $d(\Lambda, T_u(\Lambda)) = 0$. For (conical) Lagrangian cobordisms, the first property we use is that, with respect to the adapted metric that we choose, the only obstruction appears at $u = 0$. Therefore, what we need is simply existence of sheaves on the doubling for any $u > 0$.

    When understanding sheaves with singular support on $\Lambda \cup T_{u}(\Lambda)$ (for $u$ greater than all possible obstructions), we usually need some extra condition since sheaves singularly supported on the doubling are in general very complicated. For (conical) Lagrangian cobordisms, the second property we use is that, the obstruction is also local in space $(x, s) \in M \times \bR_{>0}$, where the obstruction is located at the negative end $M \times (0, s_-)$. Since there are no obstructions of doubling away from the negative end, they are exactly parametrized by local systems on $\Lambda$. Therefore, we can decompose sheaves on the doubling $\Lambda \cup T_{u}(\Lambda)$ as a fiber product.

    Finally, when separating the double copies, for (conical) Legendrian cobordisms, we use the condition that we are working in the 1-jet bundle. In general, it may be possible to consider the contact boundary of subcritical Weinstein manifolds or open contact manifolds with no Reeb orbits by using the colimits of positive isotopies following \cite{Kuo}, but we do not yet know a way to do that.

\section{Action Filtration and Length of Cobordisms}

    In this section, we study the the interaction between Lagrangian cobordisms between Legendrians and action filtrations for Homs between sheaves with singular support on the Legendrians.

    Recall that for $\Lambda \subset J^1(M)$ and $\SF, \SG \in Sh_\Lambda(M \times \bR)$, we have constructed a persistence module $\mathscr{H}om_{(0,+\infty)}(\SF, \SG)$ (using the notation in \cite{LiEstimate}) following Asano-Ike \cite{AsanoIke}. Now we try to understand the map between persistence modules for sheaves microsupported on $\Lambda_-$ and $\Lambda_+$ when they are connected by a Lagrangian cobordism $L$.

    For $\Lambda \subset J^1(M)$, we will use the convention
    $$\Lambda^s = \{(x, s\xi, st) | (x, \xi, t) \in \Lambda\} \subset J^1(M),$$
    which is the image of $\Lambda$ under the contactomorphism defined by scaling.

\begin{theorem}\label{thm:persist}
    Let $\widetilde{L} \subset J^1(M \times \mathbb{R}_{>0})$ be a conical Legendrian cobordism from $\Lambda_- \subset J^1(M)$ to $\Lambda_+ \subset J^1(M)$ with no Reeb chords that is conical outside $J^1(M \times (s_-, s_+))$. Let $\SF_-, \SG_- \in Sh_{\Lambda_-}(M \times \bR)_0$ and $\SF_+, \SG_+ \in Sh_{\Lambda_+}(M \times \bR)_0$, such that under the Lagrangian cobordism functor
    $$\Psi_L(\SF_-, \SL) = \SF, \; \Psi_L(\SG_-, \SL') = \SG, \;\, i_+^{-1}\SF = \SF_+, \; i_+^{-1}\SG = \SG_+.$$
    Let $\SF_+^{s_-}, \SG_+^{s_-} \in Sh_{\Lambda_-^{s_-}}(M \times \bR)_0$ and $\SF_+^{s_+}, \SG_+^{s_+} \in Sh_{\Lambda_+^{s_+}}(M \times \bR)_0$ be the images of $\SF_-, \SG_-$ and $\SF_+, \SG_+$ under the scaling contactomorphism. Then there is an (action decreasing) map between presistence modules
    $$\mathscr{H}om_{(0,+\infty)}(\SF_+^{s_+}, \SG_+^{s_+}) \longrightarrow \mathscr{H}om_{(0,+\infty)}(\SF_-^{s_-}, \SG_-^{s_-}),$$
which enhances the natural map induced by the Lagrangian cobordism functor.
\end{theorem}

    We briefly explain the natural map $Hom(\SF_+, \SG_+) \rightarrow Hom(\SF_-, \SG_-)$ induced by the Lagrangian cobordism functor. Theorem \ref{thm:cond-quan-cob} and Proposition \ref{prop:restrict-ff} implies that for $\SF_-, \SG_- \in Sh_{\Lambda_-}(M \times \bR)_0$ and $\SF_+, \SG_+ \in Sh_{\Lambda_+}(M \times \bR)_0$ related by the Lagrangian cobordism functor, we have natural maps given by restrictions to both ends
    $$Hom(\SF_+, \SG_+) \xleftarrow{\,\sim\,} Hom(\SF, \SG) \longrightarrow Hom(\SF_-, \SG_-).$$
The goal of the theorem is to enhance the natural map of modules to a map between persistence modules using the action filtration.

\subsection{Action decreasing maps of persistence modules}
    We briefly recall the notions and translate the language of persistence modules and maps between persistence modules into sheaves on $\bR$ with singular support condition and morphisms between the sheaves.

\begin{definition}
    A persistence module $M_\bR$ is an family of graded $\Bbbk$-modules $\{M_a\}_{a \in \bR}$ together with a family of maps $\{f_{a_0a_1}: M_{a_0} \rightarrow M_{a_1}\}_{a_0 \leq a_1}$ such that $f_{aa} = \mathrm{id}$ and $f_{a_1a_2} \circ f_{a_0a_1} = f_{a_0a_2}$. $M_\bR$ is called tame if $\dim_\Bbbk M_a < \infty$ for any $a \in \bR$.

    An (action decreasing) map between persistence modules $M_\bR$ and $N_\bR$ is a family of maps $\{\varphi_a: M_a \rightarrow N_a\}_{a\in \bR}$ such that $\varphi_{a_1} \circ f^M_{a_0a_1} = f^N_{a_0a_1} \circ \varphi_{a_0}$.
\end{definition}

    Persistence modules can be translated to sheaves on $\bR$ with singular support conditions. Let $\SF \in Sh_{\nu < 0}(\bR)$ be a sheaf. One can define a persistence module by
    $$M_a = H^*\Gamma((a, +\infty), \SF), \; f_{a_0a_1}: H^*\Gamma((a_0, +\infty), \SF) \rightarrow H^*\Gamma((a_1, +\infty), \SF).$$
    Conversely, given a persistence module $M_\bR$, one can define a sheaf $\SF \in Sh_{\nu < 0}(\bR)$ that way. Thus we will not distinguish sheaves in $Sh_{\nu < 0}(\bR)$ and persistence modules.

    The following theorems show rigorously that tame persistence modules correspond to constructible sheaves.

\begin{theorem}[Guillermou \cite{Guithree}*{Corollary 7.3}, Kashiwara-Schapira \cite{KSpersist}*{Theorem 1.7}]
    Let $\SF \in Sh^b_{\nu < 0}(\bR)$ be a constructible sheaf. Then there exists a locally finite family of intervals $\{ (a_\alpha, b_\alpha] | \alpha \in A\}$ such that
    $$\SF \simeq \bigoplus_{\alpha \in A} \Bbbk_{(a_\alpha, b_\alpha]}[-j_\alpha].$$
    Here, each interval $(a_\alpha, b_\alpha]$ is called a bar.
\end{theorem}
\begin{theorem}[Kashiwara-Schapira \cite{KSpersist}*{Theorem 1.25}]
    There is an equivalence between the category of 1-dimensional constructible sheaves $Sh^b_{\nu < 0,ctr}(\bR)$ and the category of tame graded persistence modules $\mathrm{Bar}_{<0,tame}^\bZ$.
\end{theorem}
\begin{remark}
    Note that there exists a nonzero map $\Bbbk_{(0,a_1]} \rightarrow \Bbbk_{(0, a_2]}$ if and only if $a_1 \geq a_2$. This justifies the name of action decreasing maps between persistence modules.
\end{remark}

    An effective way to obtain maps between persistence modules is to construct sheaves on $\bR^2$ with singular support conditions. Let $\SF \in Sh_{\sigma > 0,\nu < 0}(\bR^2)$ be a sheaf, $\SF_- = i_-^{-1}\SF$ and $\SF_+ = i_+^{-1}\SF$ where $i_\pm: \{\pm s\} \times \bR \hookrightarrow \bR^2$ for $s \gg 0$. The singular support condition shows that $\SF$ is $i_\pm$-noncharacteristic, and thus \cite[Proposition 5.4.10]{KS}
    $$\SF_-, \SF_+ \in Sh_{\nu < 0}(\bR^2)$$
    can be viewed as persistence modules. Consider the restriction maps
    $$\Gamma(\{-s\} \times (a, +\infty), \SF) \leftarrow \Gamma([-s, s] \times (a, +\infty), \SF) \rightarrow \Gamma(\{s\} \times (a, +\infty), \SF).$$
    Compactifying the manifold and taking the push forward of the sheaves, we could assume that the supports of the sheaves are compact.
    Then using microlocal Morse lemma \cite[Corollary 5.4.19]{KS}, the second restriction map is a quasi-isomorphism. This thus induces a map between persistence modules
    $$H^*\Gamma((a, +\infty), \SF_-) \longleftarrow H^*\Gamma((a, +\infty), \SF_+).$$
    In the next section, to prove Theorem \ref{thm:persist}, we will therefore try to construct such a sheaf with the singular support assumptions.

\subsection{Action decreasing maps from Lagrangian cobordisms}
    First, we recall the definition of the persistence module $\mathscr{H}om_{(0,+\infty)}(\SF, \SG)$ for $\SF, \SG \in Sh_\Lambda(M \times \bR)$ in \cite{LiEstimate} which is a consequence of the discussion in Asano-Ike \cite{AsanoIke}.

    Recall that for a contact Hamiltonian $H$, the Legendrian movie of $\Lambda$ is
    $$\Lambda_H = \{(x, \xi, u, \nu) | (x, \xi) = \varphi_H^u(x_0, \xi_0), \nu = -H_u \circ \varphi_H^u(x_0, \xi_0), \, (x_0, \xi_0) \in \Lambda\}.$$
    In this section, we let $\Lambda_q$ be the Legendrian movie of the identity contact flow and $\Lambda_r$ be the Legendrian movie of the Reeb flow on $J^1(M) \cong T^{*,\infty}_{\tau>0}(M \times \bR)$, namely,
    \begin{align*}
    \Lambda_q &= \{(x, t, u; \xi, 1, 0) | (x, t; \xi, 1) \in \Lambda, u \in \bR\}, \\
    \Lambda_r &= \{(x, t+u, u; \xi, 1, -1) | (x, t, \xi, 1) \in \Lambda, u \in \bR\}.
    \end{align*}

\begin{definition}
    Let $\Lambda \subset J^1(M)$ be a Legendrian. For $\SF, \SG \in Sh_\Lambda(M \times \bR)$, let $\SF_q \in Sh_{\Lambda_q}(M \times \bR)$ be the image of $\SF$ under the trivial contact flow, and $\SG_r \in Sh_{\Lambda_r}(M \times \bR)$ be the image of $\SG$ under the Reeb flow. Define
    $$\mathscr{H}om_{(a, b)}(\SF, \SG) = u_*\mathscr{H}om(\SF_q, \SG_r)|_{(a, b)}.$$
    Let $\widetilde{L} \subset J^1(M \times \bR_{>0})$ be a Legendrian cobordism. For ${\SF}, {\SG} \in Sh_{\widetilde{L}}(M \times \bR \times \bR_{>0})$, define
    $$\mathscr{H}om_{(a, b)}(\SF, \SG) = (u \times s)_*\mathscr{H}om(\SF_q, \SG_r)|_{(a, b) \times \bR_{>0}}.$$
\end{definition}

    The singular support along the $u$-direction correspond to Reeb chords on $\Lambda_\pm$ and respectively $\widetilde{L}$. Thus the following observations are immediate:

\begin{lemma}\label{lem:ss-persist}
    Let $\widetilde{L} \subset J^1(M \times \bR_{>0})$ be a Legendrian cobordism conical outside $J^1(M \times (s_-, s_+))$ with no Reeb chords. Then for ${\SF}, {\SG} \in Sh_{\widetilde{L}}(M \times \bR \times \bR_{>0})$,
    $$SS^\infty(\mathscr{H}om_{(0,+\infty)}(\SF, \SG)) \cap \{(s, u; 0, \nu) | \nu \neq 0 \} = \varnothing.$$
    For $i_\pm: M \times \bR \times \{s_\pm\} \times \bR \hookrightarrow M \times \bR \times \bR_{>0} \times \bR$,
    \begin{align*}
    SS^\infty(\mathscr{H}om_{(0,+\infty)}(\SF, \SG)) & \cap T^{*,\infty}((0, s_-) \times (0, +\infty)) = \{(s, sc_{i-}; c_{i-}, -1) | s < s_-\},\\
    SS^\infty(\mathscr{H}om_{(0,+\infty)}(\SF, \SG)) & \cap T^{*,\infty}((s_+, \infty) \times (0, +\infty)) = \{(s, sc_{j+}; c_{j+}, -1) | s > s_+\},
    \end{align*}
    where $c_{i-}$ are the lengths of the Reeb chords on $\Lambda_-$ and $c_{j+}$ are the lengths of the Reeb chords on $\Lambda_+$.
\end{lemma}
\begin{proof}
    Consider the first assertion. Since $SS^\infty(\SF_q) \cap SS^\infty(\SG_r) = \varnothing$, by \cite{KS}*{Proposition 5.4.14}, we know that
    $$SS^\infty(\mathscr{H}om(\SF_q, \SG_r)) \subset -SS^\infty(\SF_q) + SS^\infty(\SG_r) = -\widetilde{L}_q + \widetilde{L}_r.$$
    We know that $(x, s, z, u; 0, 0, 0, \nu) \in -\widetilde{L}_q + \widetilde{L}_r$ for $\nu \neq 0$ if and only if there are points $(x, s, z; y, t, 1)$ and $(x, s, z + u; y, t, 1) \in \widetilde{L}$. Since $\widetilde{L}$ has no Reeb chords,
    $$SS^\infty(\mathscr{H}om(\SF_q, \SG_r)) \cap \{(x, s, z, u; 0, 0, 0, \nu) | \nu \neq 0\} = \varnothing.$$
    Therefore, the first assertion on $SS^\infty(\mathscr{H}om_{(0,+\infty)}(\SF, \SG))$ follows from the singular support estimate for push forward functors \cite{KS}*{Proposition 5.4.4}.

    Then consider the second assertion. We know that $(x, s, z, u; 0, t, 0, \nu) \in -\widetilde{L}_q + \widetilde{L}_r$ for $(\sigma, \nu) \neq (0, 0)$ if and only if there are points $(x, s, z; y, t_0, 1)$ and $(x, s, z + u; y, t_1, 1) \in \widetilde{L}$ such that
    $$t = t_1 - t_0, \;\; \nu = -1.$$
    On $T^{*,\infty}(M \times \bR \times (0, s_-) \times (0, +\infty))$, this condition means that there are pairs of points $(x, s, st_0; s\xi, t_0, 1)$ and $(x, s, st_1 + u; s\xi, t_1, 1)$ such that $(x, \xi, t) \in \Lambda_-$, where now
    $$t = t_1 - t_0, \;\; u = -s(t_1 - t_0), \;\; \nu = -1.$$
    However, the pair of points $(x, \xi, t_0)$ and $(x, \xi, t_1) \in \Lambda_-$ corresponds to Reeb chords of length $t_1 - t_0$ on $\Lambda_-$. Therefore, the second assertion also follows from \cite{KS}*{Proposition 5.4.4}.
\end{proof}

\begin{corollary}
    Let $\widetilde{L} \subset J^1(M \times \bR_{>0})$ be a Legendrian cobordism with no Reeb chords. Then for ${\SF}, {\SG} \in Sh_{\widetilde{L}}(M \times \bR \times \bR_{>0})$ and any $c > 0$,
    $$\mathscr{H}om_{c}(\SF, \SG) = s_*\mathscr{H}om(\SF, \SG).$$
\end{corollary}
\begin{proof}
    This follows from microlocal Morse lemma \cite{KS}*{Corollary 5.4.19} and the first assertion of Lemma \ref{lem:ss-persist}.
\end{proof}

    Suppose one can prove the estimation that
    $$SS^\infty(\mathscr{H}om_{(0,+\infty)}(\SF, \SG)) \subset \{(s, u; \sigma, \nu) | \sigma > 0, \nu < 0\}.$$
    Then Theorem \ref{thm:persist} immediately follows. Unfortunately, we are not able to prove that, and in fact, it seems likely that the geometric statement is incorrect for arbitrary Lagrangian cobordisms (for example, Lagrangian caps \cite{LagCap}, especially the low dimensional ones \cite{LinCap}, might be potential non-examples, though algebraically we know that they cannot support any nontrivial sheaves). However, the property that
    $$SS^\infty(\mathscr{H}om_{(0,+\infty)}(\SF, \SG)) \cap \{(s, u; 0, \nu) | \nu \neq 0 \} = \varnothing$$
    will enable us to cut-off the singular support in $\{(s, u; \sigma, \nu) | \sigma < 0\}$ without changing the behaviour of the sheaf in $\{(s, u; \sigma, \nu) | \sigma > 0\}$.

    Let $\iota_{\sigma \geq 0}: Sh_{\sigma \geq 0}(\bR^2) \hookrightarrow Sh(\bR^2)$ be the inclusion. Tamarkin \cite{Tamarkin1}, following \cite{KS}, shows that the right adjoint of the inclusion is defined by convolution. The following theorem is a variation of Tamarkin \cite{Tamarkin1}*{Proposition 2.2} or \cite{GuiSchapira}*{Proposition 3.19}.

\begin{proposition}
    Let $\iota_{\sigma \leq 0}^*: Sh(\bR^2) \rightarrow Sh_{\sigma \leq 0}(\bR^2)$ be the left adjoint to the tautological inclusion $\iota_{\sigma \geq 0}: Sh_{\sigma \leq 0}(\bR^2) \hookrightarrow Sh(\bR^2)$. Then
    \begin{align*}
    \iota_{\sigma \leq 0}^* \SF &= \Bbbk_{[0,+\infty)} \star' \SF = a_*(\pi_1^{-1}\SF \otimes \pi_2^{-1}\Bbbk_{[0,+\infty)}), 
    \end{align*}
    where $a: \bR^2 \times \bR \rightarrow \bR^2, \, (s_1, u, s_2) \mapsto (s_1 + s_2, u)$.
\end{proposition}
\begin{proof}
    We need to prove that for any $\SF \in Sh(\bR^2)$ and $\SG \in Sh_{\sigma \leq 0}(\bR^2)$, there is a quasi-isomorphism
    $$Hom(\Bbbk_{[0,+\infty)} \star' \SF, \SG) \simeq Hom(\SF, \SG).$$
    Since $\SF = \Bbbk_{0} \star' \SF = a_*(\pi_1^{-1}\SF \otimes \pi_2^{-1}\Bbbk_0)$, we know that it suffices to show that
    $$Hom(\Bbbk_{(0,+\infty)} \star' \SF, \SG) \simeq 0.$$
    Let $j: \bR^2 \times (0, +\infty) \hookrightarrow \bR^2 \times \bR$ be the open embedding. Then we know that
    \begin{align*}
    \Bbbk_{(0,+\infty)} \star' \SF &= a_*(\pi_1^{-1}\SF \otimes \pi_2^{-1}\Bbbk_{(0,+\infty)}) = a_*(j_!j^{-1}(\pi_1^{-1}\SF \otimes \pi_2^{-1}\Bbbk_{(-\infty,+\infty)})) \\
    &= a_!j_!j^{-1}(\pi_1^{-1}\SF \otimes \pi_2^{-1}\Bbbk_{(-\infty,+\infty)}) = a_!(\pi_1^{-1}\SF \otimes \pi_2^{-1}\Bbbk_{(0,+\infty)}).
    \end{align*}
    Therefore, denoting $\Bbbk_{(0,+\infty)} \star \SF = a_!(\pi_1^{-1}\SF \otimes \pi_2^{-1}\Bbbk_{(0,+\infty)})$, it suffices show is that for any $\SF \in Sh(\bR^2)$ and $\SG \in Sh_{\sigma \leq 0}(\bR^2)$,
    $$Hom(\Bbbk_{(0,+\infty)} \star \SF , \SG) \simeq 0.$$
    This is equivalent to the result of Tamarkin \cite{Tamarkin1}*{Proposition 2.2} or \cite{GuiSchapira}*{Proposition 3.19}. We know by microlocal cut-off lemma \cite{KS}*{Proposition 5.2.3} that $\mathscr{H} \in Sh_{\sigma \leq 0}(\bR^2)^\perp$ is in the left orthogonal complement if and only if $\Bbbk_{(-\infty,0]} \star \mathscr{H} \simeq \mathscr{H}$ or equivalently $\Bbbk_{(-\infty,0)} \star \mathscr{H} \simeq 0$. One can explicitly show that
    $$\Bbbk_{(-\infty,0)} \star (\Bbbk_{(0,+\infty)} \star \SF) \simeq 0.$$
    Hence $\Bbbk_{(0,+\infty)} \star \SF \in Sh_{\sigma \leq 0}(\bR^2)^\perp$. This completes the proof.
\end{proof}
\begin{remark}
    We remark the difference with \cite{Tamarkin1}*{Proposition 2.2} or \cite{GuiSchapira}*{Proposition 3.19}. They proved that the standard convolution $\Bbbk_{[0,+\infty)} \star \SF = a_!(\pi_1^{-1}\SF \otimes \pi_2^{-1}\Bbbk_{[0,+\infty)})$ is the left adjoint to the localization functor $Sh(M \times \bR) \rightarrow Sh(M \times \bR)/Sh_{\sigma \leq 0}(M \times \bR)$. In our terminology, the essential image of this functor should be $Sh_{\sigma \geq 0}(M \times \bR)_0$.
\end{remark}

    The following lemma explains that the microlocal cutoff does not introduce extra singular support in $\{(s, u; \sigma, \nu) | \sigma \geq 0\}$. This is the place where we require the Lagrangian cobordism to be embedded.

\begin{lemma}\label{lem:ss-cutoff}
    Let $\SF \in Sh_{\sigma \neq 0}(\bR^2)$. Then $SS^\infty(\iota_{\sigma \geq 0}^* \SF) = SS^\infty(\SF) \cap \{(s, u; \sigma, \nu) | \sigma > 0\}$.
\end{lemma}
\begin{proof}
    First, since $\iota_{\sigma \geq 0}^* \SF = \Bbbk_{[0,+\infty)} \star' \SF$, by microlocal cut-off lemma of Kashiwara-Schapira \cite{KS}*{Proposition 5.2.3} we have
    $$SS^\infty(\iota_{\sigma \geq 0}^* \SF) \cap \{(s, u; \sigma, \nu) | \sigma > 0\} = SS^\infty(\SF) \cap \{(s, u; \sigma, \nu) | \sigma > 0\}.$$
    Next, setting $\Lambda = SS^\infty(\SF)$, we know that $\Lambda \cap \{(s, u; \sigma, \nu) | \sigma = 0\} = \varnothing$. Then $\iota_{\sigma \geq 0}^* \SF = \iota_{\sigma \geq 0}^* \circ \iota_\Lambda^* \SF$. By the definition of the left adjoint functor,
    $$SS^\infty(\iota_{\sigma \geq 0}^* \SF) \cap \{(s, u; \sigma, \nu) | \sigma \leq 0\} \subseteq SS^\infty(\iota_{\sigma \geq 0}^* \SF) \cap T^{*,\infty}\bR^2 \setminus \Lambda = \varnothing.$$
    This therefore finishes the proof of the lemma.
\end{proof}
\begin{remark}
    We remark that the existence of the left adjoint $\iota_\Lambda^*: Sh(M) \rightarrow Sh_\Lambda(M)$ only requires $\Lambda \subset T^{*,\infty}M$ to be closed. When $\Lambda \subset T^{*,\infty}M$ that is noncompact, the argument still works by choosing exhausting domains, applying wrapping in the compact domains and then taking the colimit; see for example the work of Kuo \cite[Proposition 1.2]{Kuo} or Zhang \cite[Theorem A.1]{ZhangS1}.
\end{remark}

    Using that characterization, we now cut off the singular support of $\mathscr{H}om_{(0,+\infty)}(\SF, \SG)$ without changing the behaviour of the sheaf in $\{(s, u; \sigma, \nu) | \sigma > 0\}$.

\begin{proposition}\label{prop:persist}
    Let $\widetilde{L} \subset J^1(M \times \bR_{>0})$ be a Legendrian cobordism conical outside $J^1(M \times (s_-, s_+))$ with no Reeb chords. Then for ${\SF}, {\SG} \in Sh_{\widetilde{L}}(M \times \bR \times \bR_{>0})$ and $i_\pm: \{s_\pm\} \times \bR_{>0} \hookrightarrow \bR_{>0} \times \bR_{>0}$,
    $$i_\pm^{-1}\big(\iota_{\sigma\geq 0}^* \mathscr{H}om_{(0,+\infty)}(\SF, \SG)\big) = i_\pm^{-1}\mathscr{H}om_{(0,+\infty)}(\SF, \SG).$$
\end{proposition}
\begin{proof}
    Since there is a natural morphism $\iota_{\sigma\geq 0}^* \mathscr{H}om_{(0,+\infty)}(\SF, \SG) \rightarrow \mathscr{H}om_{(0,+\infty)}(\SF, \SG)$, it suffices to show isomorphisms on stalks. Consider $i_-: \{s_-\} \times \bR_{>0} \hookrightarrow \bR_{>0} \times \bR_{>0}$. Then it follows from Lemma \ref{lem:ss-persist} and microlocal Morse lemma that for any $u_0 > 0$
    \begin{align*}
    \iota_{\sigma\geq 0}^* \mathscr{H}om_{(0,+\infty)}(\SF, \SG)_{(s_-, u_0)} &= \Gamma((0, s_-] \times \{ u_0 \}, \mathscr{H}om_{(0,+\infty)}(\SF, \SG)) \\
    &= \mathscr{H}om_{(0,+\infty)}(\SF, \SG)_{(s_-, u_0)}.
    \end{align*}
    Consider $i_+: \{s_+\} \times \bR_{>0} \hookrightarrow \bR_{>0} \times \bR_{>0}$. Then by Lemma \ref{lem:ss-persist} and microlocal Morse lemma, setting $h(\widetilde{L}) = \max_{\widetilde{L} \cap J^1(M \times [s_-, s_+])}z - \min_{\widetilde{L} \cap J^1(M \times [s_-, s_+])}z$, we know for $u_0 > h(\widetilde{L})$,
    \begin{align*}
    \iota_{\sigma\geq 0}^* \mathscr{H}om_{(0,+\infty)}(\SF, \SG)_{(s_+, u_0)} &= \Gamma((0, s_+] \times \{ u_0 \}, \mathscr{H}om_{(0,+\infty)}(\SF, \SG)) \\
    &= \mathscr{H}om_{(0,+\infty)}(\SF, \SG)_{(s_+, u_0)}.
    \end{align*}
    Then, using Lemma \ref{lem:ss-persist}, we know $SS^\infty(\mathscr{H}om_{(0,+\infty)}(\SF, \SG)) \cap T^{*,\infty}((s_+, \infty) \times (0, +\infty)) = \{(s, sc_{i+}; c_{i+}, -1) | s > s_+\}$. For any $u_0 > 0$, consider $\lambda > 0$ sufficiently large such that $\lambda u_0 > h(\widetilde{L})$. Then by Lemma \ref{lem:ss-cutoff}
    \begin{align*}
    \iota_{\sigma\geq 0}^* \mathscr{H}om_{(0,+\infty)}(\SF, \SG)_{(s_+, u_0)} &= \iota_{\sigma\geq 0}^* \mathscr{H}om_{(0,+\infty)}(\SF, \SG)_{(\lambda s_+, \lambda u_0)} \\
    &= \mathscr{H}om_{(0,+\infty)}(\SF, \SG)_{(\lambda s_+, \lambda u_0)} \\
    &= \mathscr{H}om_{(0,+\infty)}(\SF, \SG)_{(s_+, u_0)}.
    \end{align*}
    This therefore completes the proof.
\end{proof}

\begin{proof}[Proof of Theorem \ref{thm:persist}]
    Consider the sheaf $\iota_{\sigma\geq 0}^* \mathscr{H}om_{(0,+\infty)}(\SF, \SG)$. We have
    $$SS^\infty(\iota_{\sigma\geq 0}^* \mathscr{H}om_{(0,+\infty)}(\SF, \SG)) \subset \{(s, u; \sigma, \nu) | \sigma > 0, \nu < 0\}.$$
    Then the theorem follows from Proposition \ref{prop:persist}. Indeed, consider the restrictions
    $$\mathscr{H}om_{(0,+\infty)}(\SF_-, \SG_-) \leftarrow \iota_{\sigma\geq 0}^* \mathscr{H}om_{(0,+\infty)}(\SF, \SG) \rightarrow \mathscr{H}om_{(0,+\infty)}(\SF_+, \SG_+).$$
    Compactifying the manifold and taking push forward of the sheaves, we may assume that the supports of the sheaves are compact. By microlocal Morse lemma \cite{KS}*{Corollary 5.4.19}, the right arrow is an isomorphism and this gives the map of persistence modules. 

     Since the natural map given by Lagrangian cobordism functor comes exactly from restrictions to both ends and microlocal Morse lemma, this map between persistence modules enhances the natural map induced by the Lagrangian cobordism functor.
\end{proof}

\subsection{Legendrian capacity and length of cobordisms}
    Let $\SF, \SG \in Sh_{\widetilde{L}}(M \times \bR \times \bR_{>0})$, $\SF_- = i_-^{-1}\SF, \SG_- = i_-^{-1}\SG$ and $\SF_+ = i_+^{-1}\SF, \SG_+ = i_+^{-1}\SG$. Theorem \ref{thm:persist} implies that there is a commutative diagram
    \[\xymatrix{
    {H}om(\SF_-^{s_-}, T_c(\SG_-^{s_-})) & {H}om(\SF_+^{s_+}, T_c(\SG_+^{s_+})) \ar[l]_{r_{L,c}} \\
    {H}om(\SF_-^{s_-}, \SG_-^{s_-}) \ar[u]^{P_{c-}} & {H}om(\SF_+^{s_+}, \SG_+^{s_+}) \ar[l]_{r_L} \ar[u]_{P_{c+}}.
    }\]
    Following \cite{SabTraynorLength}, we can define Legendrian capacities.

\begin{definition}
    Let $\SF, \SG \in Sh_\Lambda(M \times \bR)$. Then for any $\theta \in Hom(\SF, \SG)$, the Legendrian capacity is
    $$c(\Lambda, \SF, \SG, \theta) = \sup \{c > 0 \,| P_c(\theta) \neq 0\}.$$
    In particular, for $\mathrm{id} \in Hom(\SF, \SF)$, the fundamental Legendrian capacity is denoted by $c(\Lambda, \SF) = c(\Lambda, \SF, \SF, \mathrm{id})$.
\end{definition}
\begin{remark}\label{rem:cap-distance}
    Consider the persistence distance on $Sh^b_{\tau > 0}(M \times \bR)$ defined by Asano-Ike \cite{AsanoIke}. In the special case when one of the sheaves is zero,
    $$d(0, \SF) = \inf\left\{c\, \big|\, 0 \simeq T_c : \SF \rightarrow T_c\SF\right\}.$$
    Therefore, the fundamental capacity $c(\Lambda_+, \SF) = d(0, \SF)$ as they are both the length of the barcode that contains $\mathrm{id} \in Hom(\SF, \SF)$.
\end{remark}

    Suppose $L \subset T^*(M \times \bR_{>0})$ is a Lagrangian cobordism. Let $s_-$ be the supremum and $s_+$ be the infimum such that $L$ is conical outside $T^*(M \times [s_-, s_+])$. Define the length of the Lagrangian cobordism as
    $$l(L) = \ln s_+ - \ln s_-.$$
    Note that this is indeed the length after the reparametrization $\bR_{>0} \rightarrow \bR, s \mapsto \ln s$.

\begin{corollary}
    Let ${L} \subset T^*(M \times \bR_{>0})$ be an exact Lagrangian cobordism from $\Lambda_-$ to $\Lambda_+$. Let $\SF, \SG \in Sh_{\widetilde{L}}(M \times \bR \times \bR_{>0})$, $\SF_- = i_-^{-1}\SF, \SG_- = i_-^{-1}\SG$ and $\SF_+ = i_+^{-1}\SF, \SG_+ = i_+^{-1}\SG$. Then
    $$s_-c(\Lambda_-, \SF_-, \SG_-, r_L(\theta)) \leq s_+c(\Lambda_+, \SF_+, \SG_+, \theta).$$
    In particular, the length of the Lagrangian cobordism
    $$l(L) \geq \ln c(\Lambda_-, \SF_-, \SG_-, r_L(\theta)) - \ln c(\Lambda_+, \SF_+, \SG_+, \theta).$$
\end{corollary}

    This for example recovers the results of Sabloff-Traynor \cite{SabTraynorLength}. We recall their results and give sketches of proofs.

\begin{theorem}[Sabloff-Traynor \cite{SabTraynorLength}*{Theorem 1.1}]
    Let $\Lambda_- = \Lambda_\text{Unknot}(u)$ and $\Lambda_+ = \Lambda_\text{Unknot}(v) \subset J^1(\bR^n)$ be Legendrian unknots with single Reeb chords of length $u$ and $v$. Then the minimal length of Lagrangian cobordisms $L \subset T^*(\bR^n \times \bR_{>0})$ satisfies
    \begin{enumerate}
      \item $l(L) \geq \ln u - \ln v$ if $u \geq v$;
      \item $l(L)$ can be arbitrarily small if $u \leq v$.
    \end{enumerate}
\end{theorem}
\begin{proof}
    Consider the microlocal rank~1 sheaf $\SF$ microsupported on $\widetilde{L}$ which restrict to the unique compactly supported microlocal rank~1 sheaves $\SF_\pm$ with microsupport on the Legendrian unknots $\Lambda_\pm$. Then the persistence modules are
    $$\mathscr{H}om_{(0,+\infty)}(\SF_-, \SF_-) = \Bbbk_{(0, u]}, \; \mathscr{H}om_{(0,+\infty)}(\SF_+, \SF_+) = \Bbbk_{(0, v]}.$$
    Lagrangian cobordisms between Legendrian unknots have to be concordances, and then nontriviality of the cobordism functor implies that the pair of barcodes in the persistence modules is matched up respectively. Hence the result follows from Theorem \ref{thm:persist}.
\end{proof}

\begin{theorem}[Sabloff-Traynor \cite{SabTraynorLength}*{Theorem 1.2}]
    Let $\Lambda_- = \Lambda_\text{Hopf}(u)$, $\Lambda_+ = \Lambda_\text{Hopf}(v) \subset J^1(\bR^n)$ be Legendrian Hopf links of two unknots $\Lambda_\text{Unknot}(1)$ related by a Reeb translation such that Reeb chords between the two components have length $u$ and $v$. Then the minimal length of Lagrangian cobordism $L \subset T^*(\bR^n \times \bR_{>0})$ that connects the lower (resp.~upper) component of $\Lambda_-$ to the lower (resp.~upper) component of $\Lambda$ satisfies
    \begin{enumerate}
      \item $l(L) \geq \ln u - \ln v$ if $u \geq v$;
      \item $l(L) \geq \ln(1 - u) - \ln(1 - v)$ if $u \leq v$.
    \end{enumerate}
\end{theorem}
\begin{proof}
    Consider the microlocal rank~1 sheaf $\SF$ microsupported on $\widetilde{L}$ which is the direct sum of two microlocal rank~1 sheaves on each connected component of $\widetilde{L}$, that restricts to $\SF_\pm$ with microsupport on the Legendrian Hopf link $\Lambda_\pm$ which is the direct sum of two microlocal rank~1 sheaves on each component of the Legendrian unknot. Then one can check that the persistence modules are
    \begin{align*}
    \mathscr{H}om_{(0,+\infty)}(\SF_-, \SF_-) = \Bbbk_{(0, 1]}^2 \oplus \Bbbk_{(0,1-u]} \oplus \Bbbk_{(0,u]}[-2] \oplus \Bbbk_{(u, 1+u]}, \\
    \mathscr{H}om_{(0,+\infty)}(\SF_+, \SF_+) = \Bbbk_{(0, 1]}^2 \oplus \Bbbk_{(0,1-v]} \oplus \Bbbk_{(0,v]}[-2] \oplus \Bbbk_{(v, 1+v]}.
    \end{align*}
    Moreover, since the Lagrangian cobordism matches up the upper and lower components of the Hopf links, we know that each component has to be a Lagrangian concordance between Legendrian unknots. Nontriviality of the cobordism functor implies that each pair of barcodes in the persistence modules is matched up respectively. Hence the result follows from Theorem \ref{thm:persist}.
\end{proof}

\section{Lagrangian Cobordism Functor as Correspondence}

    In this section, we prove that the Lagrangian cobordism functor constructed using the sheaf quantization functor in Theorem \ref{thm:cond-quan-cob-intro} is compatible with the Lagrangian cobordism functor constructed using the gapped specialization of Nadler-Shende in Theorem \ref{thm:cob-functor}.

\begin{theorem}\label{thm:compatible-cob}
    Let $\Lambda_\pm \subset J^1(M)$ be closed Legendrian submanifolds, and $L \subset J^1(M) \times \mathbb{R}$ an embedded exact Lagrangian cobordism from $\Lambda_-$ to $\Lambda_+$, which lifts to a Legendrian cobordism $\widetilde{L} \subset (J^1(M) \times \mathbb{R}) \times \mathbb{R} \cong J^1(M \times \mathbb{R}_{>0})$.
    Write $i_\pm: M \times \mathbb{R} \times \{s_\pm\} \hookrightarrow M \times \mathbb{R} \times \mathbb{R}_{>0}$ for $T \gg 0$. Then there is a commutative diagram between sheaf categories
    \[\xymatrix@R=6mm{
    & Sh_{\widetilde{L}}(M \times \mathbb{R} \times \mathbb{R}_{>0}) \ar[dl]_{(i_-^{-1}, m_{\widetilde{L}})} \ar[dr]^{i_+^{-1}} & \\
    Sh_{\Lambda_-}(M \times \mathbb{R}) \times_{Loc(\Lambda_-)} Loc(\widetilde{L}) \ar[rr]_{\hspace{40pt}\Phi_L} & & Sh_{\Lambda_+}(M \times \mathbb{R})
    }\]
    where $m_{\widetilde{L}}$ is the microlocalization functor. In particular, $(i_-^{-1}, m_{\widetilde{L}})$ is essentially surjective while $i_+^{-1}$ is full faithful.
\end{theorem}

    Since by Theorem \ref{thm:cond-quan-cob}, $\Psi_L$ is the right inverse of $(i_-^{-1}, m_{\widetilde{L}})$, the fiber product of restriction to the negative end and microlocalization along the cobordism, once we prove that $\Phi_L \circ (i_-^{-1}, m_{\widetilde{L}}) = i_+^{-1}$, this will indicate that
    $$i_+^{-1} \circ \Psi_L = \Phi_L \circ (i_-^{-1}, m_{\widetilde{L}}) \circ \Psi_L = \Phi_L.$$
    This implies that the Lagrangian cobordism functor obtained by the gapped specialization functor over Lagrangian skeleta and the functor obtained by the conditional sheaf quantization functor on the Legendrian cobordism are compatible with each other.

\subsection{Lagrangian cobordism functor following Nadler-Shende}\label{sec:NadShen}
    In order to prove Theorem \ref{thm:compatible-cob}, we will first need to recall the definition of our Lagrangian cobordism functor \cite{LiCobordism} following Nadler-Shende \cite{NadShen}.

    We know that $(T^*(M \times \bR), d\lambda)$ is a Weinstein sector, and the Liouville flow $Z_\lambda = s\partial/\partial s$. Suppose $L \cap \partial_\infty X \times (0, s_-) = \Lambda_- \times (0, s_-)$. Glue ${L} \cap \partial_\infty X \times (s_-, +\infty)$ with $\Lambda_- \times (0, s_-)$ along $\Lambda_- \times \{s_-\}$, and denote by $\Lambda_- \times \mathbb{R} \cup {L}$ their concatenation (note that this is the same as $L$). Since $\mu Sh_-$ is a sheaf and cosheaf of dg categories, we only need to construct a fully faithful functor
    $$\Phi_L: \, \mu Sh_{\mathfrak{c}_X \cup \Lambda_- \times \mathbb{R} \cup \widetilde{L}}\big(\mathfrak{c}_X \cup \Lambda_- \times \mathbb{R} \cup \widetilde{L}\big) \longrightarrow \mu Sh_{\mathfrak{c}_X \cup \Lambda_+ \times \mathbb{R}}(\mathfrak{c}_X \cup \Lambda_+ \times \mathbb{R}).$$
    This relies on the gapped specialization theorem of Nadler-Shende.

    We recall the construction of functor by Nadler-Shende. Consider a subanalytic Legendrian $\Lambda_1 \subset J^1(M \times \bR)$, which is either compact or locally closed, relatively compact with cylindrical ends. Let $\varphi_H^\zeta: J^1(M \times \bR) \rightarrow J^1(M \times \bR)$ be a contact isotopy for $\zeta \in (0, 1]$ conical near the cylindrical ends. Let $\Lambda_0 = \overline{\Lambda}_H \cap (X \times \mathbb{R} \times \{0\})$ be the set of limit points of $\varphi_H^\zeta(\Lambda_1)$ as $\zeta \rightarrow 0$. Then the gapped specialization is a functor
    $$\mu Sh_{\Lambda_1}(\Lambda_1) \hookrightarrow \mu Sh_{\Lambda_0}(\Lambda_0).$$

    The construction consists of two steps. First, we need to define a fully faithful embedding from $\mu Sh_\Lambda(\Lambda)$ back to $Sh(M \times \bR^N)$ under some embedding $\Lambda \hookrightarrow J^1(M \times \bR) \hookrightarrow T^{*,\infty}(M \times \bR^N)$. Using the relative doubling construction. For $\partial\Lambda_{\pm s} \times [0, 1) \subset T^{*}(U \times (-1, 1)) \times \mathbb{R}$, connect the ends $\partial \Lambda_{\pm s}$ by a family of standard cusps $\partial\Lambda\, \times \prec$. Let
    $$(\Lambda, \partial\Lambda)^\prec_s = \Lambda_{-s} \cup \Lambda_s \cup (\partial\Lambda \,\times \prec).$$
    Nadler-Shende showed that the microlocalization functor (relative to conical ends)
    $$Sh_{(\Lambda,\partial\Lambda)^\prec_\epsilon}(M \times \bR^N)_0 \rightarrow \mu Sh_{\Lambda_{-\epsilon}}(\Lambda_{-\epsilon})$$
    admits a right inverse, which is called the antimicrolocalization functor \cite[Section 7]{NadShen}.

    Second, we need to construct a fully faithful functor between subcategories of $Sh(M \times \bR \times \bR^N)$. Namely, for a doubled sheaf $\SF_\text{dbl} \in Sh_{(\Lambda,\partial\Lambda)^\prec_\epsilon}(M \times \bR^N)_0$, we consider the nearby cycle for the movie of the sheaf under the contact Hamiltonian flow (which is a contact lift of the Liouville flow in a neighbourhood of $\Lambda$)
    $$\Phi(\SF_\text{dbl}) = i^{-1}j_*\Psi_H(\SF_\text{dbl})$$
    and show that it is fully faithful in our setting. This full faithfulness criterion is proposed by Nadler \cite{NadNonchar} and proved for families of Legendrians by Zhou \cite{Zhou}*{Proposition 3.2} and in general by Nadler-Shende \cite[Section 4 \& 5]{NadShen}.

    Given the Lagrangian cobordism functor defined by nearby cycles, a commutativity criterion of nearby cycle functors will be important when checking different commutative diagrams; see for example \cite{NadNearby} or \cite{KochNearby,MaiNearby}. We extract the main technical lemma as follows.

    Write the projection maps
    $$\pi_i: N \times [-1,1] \times [-1,1] \rightarrow [-1,1], \,\, (x, t_1, t_2) \mapsto t_i, \,\,(i = 1, 2)$$
    and $\pi = \pi_1 \times \pi_2: N \times [0, 1] \times [0, 1] \rightarrow [0, 1] \times [0, 1]$. Write the inclusions
    \[\xymatrix{
    N \times \{0\} \times [-1, 0) \cup (0, 1] \ar[r]^{i} \ar[d]_{j} & N \times [-1, 1] \times [-1, 0)\cup (0, 1] \ar[d]^{\overline{j}} \\
    N \times \{0\} \times [-1, 1] \ar[r]^{\overline{j}} & N \times [-1, 1] \times [-1, 1].
    }\]
    Recall from \cite[Section 6.2]{KS} that for a closed embedding $i: N \hookrightarrow M$ and a subset $A \subseteq T^*M$, we define $i^\#(A) \subseteq T^*N$ as the set of points $(x, \xi) \in T^*N$ such that there exists $(y_n, \eta_n, x_n, 0) \in T^*M \times T^*N$ and
    $$x_n, y_n \hookrightarrow x, \; i^*\eta_n \rightarrow \xi, \; |x_n - y_n| |\eta_n| \rightarrow 0.$$
    For $\mathscr{F} \in Sh(M)$, it is known by \cite[Theorem 6.3.1]{KS} that
    $$SS^\infty(i^{-1}\mathscr{F}) \subseteq i^\#SS^\infty(\mathscr{F}).$$
    Recall also the notion of relative singular supports in for example \cite[Section 2.2]{NadShen}. Let $\pi: N \times B \rightarrow B$ and $\Pi: T^*(N \times B) \rightarrow T^*(N \times B)/\pi^*T^*B$ be the projection maps. For $\mathscr{F} \in Sh(N \times B)$, define
    $$SS_\pi^\infty(\mathscr{F}) = \Pi(SS^\infty(\mathscr{F})).$$

\begin{proposition}[\cite{LiCobordism}*{Proposition 3.3}]\label{basechange}
    Let $\mathscr{F} \in Sh(N \times [-1,1] \times [-1, 0)\cup (0, 1])$ be a sheaf such that
    \begin{enumerate}
      \item $i^\#SS^\infty(\mathscr{F}) \cap \pi_2^*T^{*,\infty}([-1, 0)\cup (0,1]) = \varnothing$,
      \item $SS^\infty(\mathscr{F}) \cap \pi^*T^{*,\infty}([-1, 0)\cup (0,1] \times [-1, 0)\cup (0,1]) = \varnothing$,
      \item $\overline{SS_{\pi}^\infty(\mathscr{F})} \cap T^{*,\infty}N \times \{(0, 0)\}$ is a subanalytic Legendrian.
    \end{enumerate}
    Then there is a natural isomorphism of sheaves
    $$\overline{i}^{-1}\overline{j}_*\mathscr{F} \simeq j_*i^{-1}\mathscr{F}.$$
\end{proposition}

\subsection{Construction of a suspension Lagrangian cobordism}\label{sec:suspend}
    We note that in the statement of Theorem \ref{thm:compatible-cob}, the Lagrangian cobordism functor of Nadler-Shende is defined by attaching to the negative end the Lagrangian $L$ along the radius (vertical) direction, while the conical Lagrangian cobordism is defined by connecting by the Legendrian $\widetilde{L}$ along the base (horizontal) direction. In order to investigate the relation between the two pictures we have to consider both directions (and connect them in a geometric way).

    Our goal in this section is thus to define a suspension exact Lagrangian cobordism $\Sigma{L} \subset T^{*}(M \times (1, +\infty) \times \mathbb{R}_{>0})$ that is diffeomorphic to ${L} \times \mathbb{R}$, such that the following conditions hold. Here we view the Lagrangian cobordisms $L$ and $\Lambda_+ \times (1, +\infty)$ as in a subdomain in the the symplectization $J^1(M) \times (1, +\infty) \cong T^*(M \times (1, +\infty))$. This will simplify some of the formulas in the discussion.

\begin{proposition}\label{prop:suspend}
    Let $L \subset T^*(M \times (1, +\infty))$ be an exact Lagrangian cobordism from $\Lambda_-$ to $\Lambda_+$ whose Legendrian lift is $\widetilde{L}$. Then there exists an exact Lagrangian $\Sigma L \subset T^*(M \times (1, +\infty) \times \bR_{>0})$ diffeomorphic to $L \times \bR$ such that
    \begin{enumerate}
      \item the symplectic reduction of the suspension to $\Sigma L_{s_-} \subset T^*(M \times \mathbb{R}_{>0}) \times \{s_-\}$ is the exact Lagrangian cobordism $L$;
      \item the symplectic reduction of the suspension to $\Sigma L_{s_+} \subset T^*(M \times \mathbb{R}_{>0}) \times \{s_+\}$ is the trivial Lagrangian cobordism $\Lambda_+ \times \mathbb{R}_{>0}$;
      \item the suspension $\Sigma L \subset T^{*}(M \times (1, +\infty) \times \mathbb{R}_{>0})$ is an exact Lagrangian cobordism between the Legendrian lifts from $\widetilde{L}$ to $\Lambda_+ \times (1, +\infty)$.
    \end{enumerate}
\end{proposition}
\begin{remark}\label{rem:suspend-reverse}
    Similarly, we can also construct a reverse suspension exact Lagrangian cobordism $\overline{\Sigma{L}} \subset T^{*}(M \times (1, +\infty) \times \mathbb{R}_{>0})$ diffeomorphic to ${L} \times \mathbb{R}$ such that
    \begin{enumerate}
      \item the symplectic reduction of the suspension to $\Sigma L_{s_-} \subset T^*(M \times \mathbb{R}_{>0}) \times \{s_-\}$ is the trivial Lagrangian cobordism $\Lambda_- \times \bR_{>0}$;
      \item the symplectic reduction of the suspension to $\Sigma L_{s_+} \subset T^*(M \times \mathbb{R}_{>0}) \times \{s_+\}$ is the exact Lagrangian cobordism $L$;
      \item the suspension $\Sigma L \subset T^{*}(M \times (1, +\infty) \times \mathbb{R}_{>0})$ is an exact Lagrangian cobordism between the Legendrian lifts from $\Lambda_- \times (1, +\infty)$ to $\widetilde{L}$.
    \end{enumerate}
\end{remark}

    First, consider the trivial exact Lagrangian cobordism $\widetilde{L} \times \bR_{>0} \subset T^*(M \times (1, +\infty) \times \bR_{>0})$ of the Legendrian lift $\widetilde{L} \subset J^1(M \times (1, +\infty))$, i.e.
    $$\widetilde{L} \times \bR_{>0} = \big\{(x, s, \theta; s\theta\xi, \theta t, st + w) | (x, s; s\xi, t; st + w) \in \widetilde{L} \big\}.$$
    Then, a natural approach is to define a diffeomorphism
    $$\phi: (1, +\infty) \times \bR_{>0} \xrightarrow{\sim} (1, +\infty) \times \bR_{>0}$$
    that sends the negative end $(1, +\infty) \times \{0\}$ to $(1, +\infty) \times \{0\}$, sends the positive end $(1, +\infty) \times \{+\infty\}$ to $\{1\} \times \mathbb{R}_{>0}$, sends $\{1\} \times \mathbb{R}_{>0}$ to $\{(1, 0)\}$ and sends $\{+\infty\} \times \bR_{>0}$ to $\{+\infty\} \times \bR_{>0} \cup \bR_{>0} \times \{+\infty\}$,
    which will extend to an exact symplectomorphism
    $$\varphi: T^{*}(M \times (1, +\infty) \times \bR_{>0}) \xrightarrow{\sim} T^{*}(M \times (1, +\infty) \times \mathbb{R}_{>0}).$$
    However, there are some technical difficulties to define the suspension cobordism using a diffeomorphism defined globally on $(1, +\infty) \times \bR_{>0}$, and instead we will consider only a certain subdomain.

\begin{figure}
  \centering
  \includegraphics[width=0.8\textwidth]{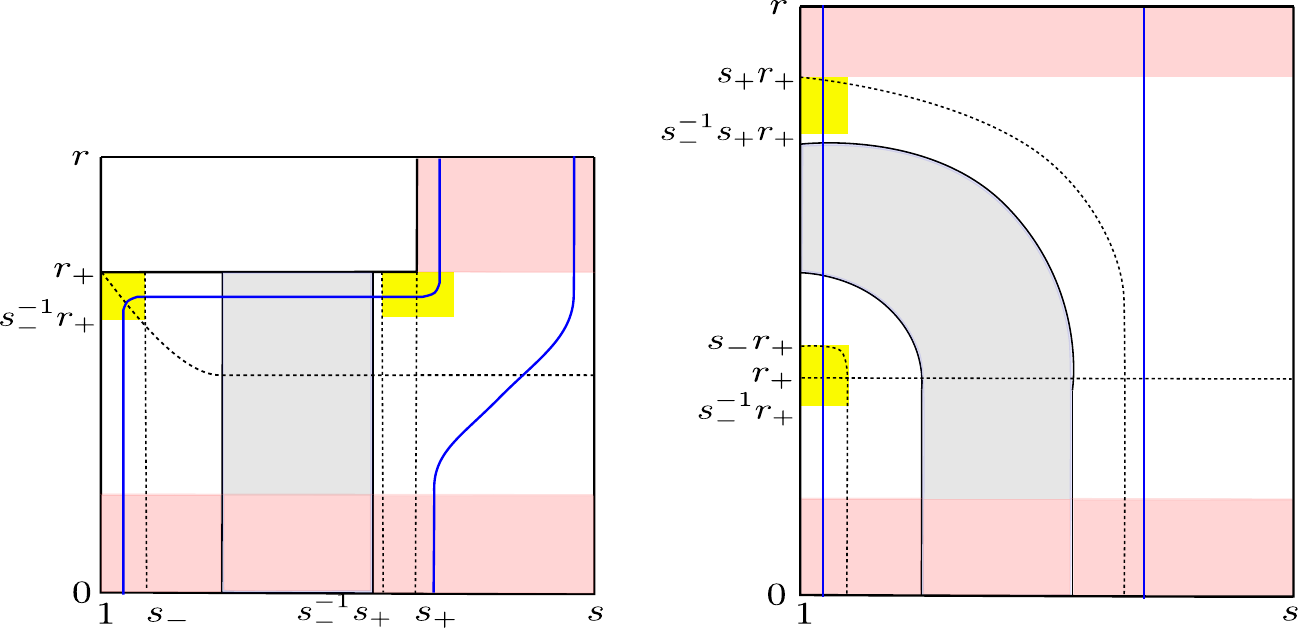}\\
  \caption{The diffeomorphism $\phi$ on a subdomain of $(1, +\infty) \times (0, +\infty)$. The grey region represents where the Lagrangian cobordism ${L} \subset T^*(M \times (1, +\infty))$ is not cylindrical. The pink regions are where Condition~(1) \& (2) are satisfied. The two blue lines are the preimage of $\phi^{-1}(\{s\} \times (0, +\infty))$ for $s < s_-$ and $s > s_+$. The two yellow regions are the regions in $\phi^{-1}((1, s_-) \times (0, +\infty))$ that are not controlled by Condition~(3).}\label{fig:suspend-cob}
\end{figure}

    Fix $s_- > 1$ and $r_- > 0$ sufficiently small, and respectively $s_+ > 1$ and $r_+ > 0$ sufficiently large. Suppose the Lagrangian cobordism $L$ is conical outside $T^*(M \times (s_-, s_+))$. Consider the diffeomorphism from a subdomain on the plane to the plane $\phi: (1, +\infty) \times \mathbb{R}_{>0} \setminus (1, s_+] \times [r_+, +\infty) \rightarrow (1, +\infty) \times \mathbb{R}_{>0}$ that satisfies the following conditions:
    \begin{enumerate}
      \item $\phi(s, r) = (s, r)$ for $s > 1$ and $r < r_-$ sufficiently small;
      \item $\phi(s, r) = (s_+^{-1} s, s_+ r)$ for $s > s_+$ and $r > r_+$ sufficiently large;
      \item $\phi(s, r) = (s, r)$ for $1 < s < s_-$ sufficiently small and $r < s_-^{-1}r_+ $, $\phi(s, r) = (r_+ r^{-1}, s_-^{-1} r_+ s)$ for $s_- < s < s_+$ and $s_-^{-1} r_+ < r < r_+$.
    \end{enumerate}
    See Figure \ref{fig:suspend-cob}. This induces a symplectomorphism (partially defined on the subdomain)
    $$\varphi: T^{*}\big(M \times ((1, +\infty) \times \bR_{>0} \backslash (1, s_+] \times [r_+, +\infty))\big) \xrightarrow{\sim} T^{*}\big(M \times (1, +\infty) \times \mathbb{R}_{>0}\big).$$
    Using the (partially defined) symplectomorphism, we define the suspension Lagrangian cobordism as follows.

\begin{definition}
    Let ${L} \subset T^*(M \times (1, +\infty))$ be an exact Lagrangian cobordism between Legendrians. Then the suspension Lagrangian $\Sigma {L}$ is the exact Lagrangian submanifold $\varphi(\widetilde{L} \times \mathbb{R}_{>0} \backslash T^*(M \times (1, s_+] \times [r_+, +\infty))) \subset T^*(M \times (1, +\infty) \times \mathbb{R}_{>0})$.
\end{definition}
\begin{remark}
    We explain the reason why the construction is more complicated than one may imagine. In fact, the Lagrangian $\widetilde{L} \times \bR_{>0}$ is conical with respect to the ${r}$-direction. In particular, the Liouville flow along the ${r}$-direction determines the positive/convex end $(1, +\infty) \times (r_+, +\infty)$ and negative/concave end $(1, +\infty) \times (0, r_-)$.

    Then we need to deform the Lagrangian on $(1, +\infty) \times \bR_{>0}$ in a way so that the Liouville flow changes and the new positive/convex end becomes the trivial Lagrangian cone on $(s_+, +\infty) \times (r_+, +\infty)$ while the negative/concave end stays the same as the original Lagrangian cobordism. Moreover, after the deformation the Lagrangian $\widetilde{L} \times \bR_{>0}$ should stay conical with respect to the ${r}$-direction, in order for it to be a Lagrangian cobordism.

    Therefore, it is in fact natural to come up with this complicated diffeomorphism, where Condition~(1) \& (2) ensure that the Lagrangian stay conical and the positive and negative end are exactly what we need. The reason to cut off the top left corner $(1, s_+] \times [r_+, +\infty)$ is to make sure that the conical condition is not violated on the top left corner after deformation (so that we only need to move $(s_+, +\infty) \times (r_+, +\infty)$ horizontally along the $s$-direction without extra uncontrolled deformation along the $r$-direction of the Liouville flow).

    This discussion also explains why we do not try to define $\Sigma L$ as both a cobordism from bottom to top and a cobordism from left to right, but rather only require that the behaviour of its symplectic reduction on the left and right slices. It seems that requiring the Lagrangian $\Sigma L$ to be tangent to both the horizontal and the vertical Liouville flow would make it too difficult to construct.
\end{remark}

    First, we verify Condition~(3) in Proposition \ref{prop:suspend}. Consider the suspension Lagrangian $\Sigma L$ and its positive/negative ends with respect to the $\bR_{>0}$ direction. Suppose $\widetilde{L} \times \bR_{>0} \cap T^*(M \times (1, \infty) \times (0, r_-)) = \{(x, s, r; sr\xi, rt, st + w) | (x, s; s\xi, t; st + w) \in \widetilde{L} \}$. With the assumption that $\phi(s, r) = (s, r)$ for $r < r_-$, we know that the induced symplectomorphism on this region is the identity,
    $$\Sigma {L} \cap T^*(M \times (1, \infty) \times (0, r_-)) = \big\{ (x, s, r; sr\xi, rt, st + w) | (x, s; s\xi, t; st + w) \in \widetilde{L} \big\}.$$
    With the assumption that $\phi(s, r) = (s_+^{-1} s, s_+ r)$ for $s > s_+$ and $r > r_+$, we know that the induced symplectomorphism sends $(x, s, r; \xi, \sigma, \rho)$ to $(x, s_+^{-1} s, s_+ r; \xi, s_+ \sigma, s_+^{-1} \rho)$, which, after coordinate changes, implies that
    $$\Sigma {L} \cap T^*(M \times (1, \infty) \times (s_+ r_+, \infty)) = \big\{ (x, s, r; sr\xi, rt, st) \big| (x, \xi, t) \in \Lambda_+ \big\}.$$
    Hence we can conclude the following lemma:

\begin{lemma}\label{lem:suspend-cob}
    Let ${L} \subset T^*(M \times \mathbb{R}_{>0})$ be a conical Legendrian cobordism from $\Lambda_-$ to $\Lambda_+$. Then $\Sigma {L} \subset T^*(M \times (1, \infty) \times \mathbb{R}_{>0})$ is an exact Lagrangian cobordism from the Legendrian lift $\widetilde{L}$ to the trivial Legendrian cone $\Lambda_+ \times (1, \infty)$.
\end{lemma}

    Next, we verify Condition~(1) \& (2) in Proposition \ref{prop:suspend} by considering the symplectic reduction of $\Sigma L$ along the hypersurface $T^*M \times \{s\} \times (1, +\infty) \times T^*\bR_{>0}$, denoted by
    $$\Sigma L_s = \pi(\Sigma L \cap T^*M \times \{s\} \times (1, +\infty) \times T^*\bR_{>0}) \subset T^*M \times T^*\bR_{>0}.$$

    For Condition~(1), consider the slice $T^*M \times \{s\} \times T^*\bR_{>0}$ for $1 < s < s_-$ sufficiently small. Firstly, on $T^*M \times \{s\} \times T^*(0, s_-^{-1} r_+)$, we know that the symplectomorphism is induced by $\phi(s, r) = (s, r)$, so
    $$\Sigma L_s \cap T^*M \times T^*(0, s_-^{-1} r_+) = \Lambda_- \times (0, s_-^{-1} r_+).$$
    Secondly, on $T^*M \times \{s\} \times T^*(s_+r_+ , +\infty)$, we know that the symplectomorphism is induced by $\phi(s, r) = (s_+^{-1}s, s_+ r)$, so since the first coordinate $s_+^{-1}s$ is fixed
    $$\Sigma L_s \cap T^*M \times T^*(s_+ r_+, +\infty) = \Lambda_+ \times (s_+r_+, +\infty);$$
    on $T^*M \times \{s\} \times T^*(s_-r_+, s_-^{-1} s_+ r_+)$, we know that the symplectomorphism is induced by $\phi(s, r) = (r_+ r^{-1}, r_+ s)$, so since the first coordinate $r_+ r^{-1}$ is fixed
    $$\Sigma L_s \cap T^*M \times T^*(s_-r_+, s_-^{-1} s_+ r_+) = L \cap T^*(M \times (s_-r_+, s_-^{-1} s_+ r_+)).$$

    Finally, consider $\Sigma L_s \cap T^*M \times T^*(s_-^{-1} r_+, s_- r_+)$ and $\Sigma L_s \cap T^*M \times T^*(s_-^{-1} s_+ r_+, s_+ r_+)$. For $\Sigma L_s \cap T^*M \times T^*(s_-^{-1} r_+, s_- r_+)$, suppose that $\gamma_1(\theta) = \phi^{-1}(s, \theta)$. Consider the isotopy from the curve $\gamma_1$ to the line segment $\gamma_0$ connecting $(s, s_-^{-1} r_+)$ and $(s_-, s^{-1} r_+^{-1})$. We use the following lemma (see for example \cite[Section 3.3]{McDuffSalamon}):

\begin{lemma}\label{lem:symp-reduce}
    Let $X$ be a symplectic manifold. Let $C_u, \, 0 \leq u \leq 1$ be a smooth family of coisotropic submanifolds in $X$ and $\mathscr{F}_u = \ker(\omega_X|_{TC_u})$ be the characteristic foliation on $TC_u$. Then there is a family of symplectomorphisms
    $$\varphi_u: C_0/\mathscr{F}_0 \xrightarrow{\sim} C_u/\mathscr{F}_u.$$
    Moreover, let $L \subset X$ be a Lagrangian submanifold. Suppose that $L \pitchfork C_u$ for any $0 \leq u \leq 1$. Let $L_u = \pi_u(L \cap C_u) \subset C_u/\mathscr{F}_u$ be the symplectic reduction. Then
    $$\varphi_u(L_0) = L_u.$$
\end{lemma}
\begin{proof}
    Using the isotopy extension theorem, there is a family of symplectomorphisms $\varphi_u: X \rightarrow X$ such that $\varphi_u(C_0) = C_u$. Properties of symplectomorphisms then ensure that it preserves characteristic foliations $\varphi_u(\SF_0) = \SF_u$. Then we obtain the first part of the claim.

    For the second part of the claim, note that $L \pitchfork \varphi_u(C_0)$ if and only if $\varphi_u^{-1}(L) \pitchfork C_0$. Since
    $$L \cap \varphi_u(C_0) = \varphi_u(\varphi_u^{-1}(L) \cap C_0),$$
    we can conclude that $\pi(\varphi_u^{-1}(L) \cap C_0)$ are Lagrangian isotopic. In particular, when they are exact Lagrangian isotopic, we can conclude that they are Hamiltonian isotopic.
\end{proof}

    Since $\widetilde{L} \times \bR_{>0} \cap T^*(M \times (1, s_-) \times (s_-^{-1} r_+, r_+))$ is the trivial suspension of the Legendrian cone $\Lambda_- \times (1, s_-) \times (s_-^{-1} r_+, r_+)$, where
    $$\Lambda_- \times (1, +\infty) \times \bR_{>0} = \{(x, s, r; sr\xi, rt, st) | (x, \xi, t) \in \Lambda_-\}.$$
    Thus $\widetilde{L} \times \bR_{>0} \pitchfork T^*M \times T^*_{\gamma_u}((1, s_-) \times (s_-^{-1} r_+, r_+))$ for any $0 \leq u \leq 1$. Using Lemma \ref{lem:symp-reduce}, we know that the symplectic reduction of $\widetilde{L} \times \bR_{>0}$ along $T^*M \times T^*_{\gamma_1}((1, s_-) \times (s_-^{-1} r_+, s_- r_+))$, is Hamiltonian isotopic to the reduction along $T^*M \times T^*_{\gamma_0}((1, s_-) \times (s_-^{-1} r_+, s_- r_+))$. We can easily tell that the symplectic reduction along $T^*M \times T^*_{\gamma_0}((1, s_-) \times (s_-^{-1} r_+, s_- r_+))$ is the trivial Legendrian cone, i.e.
    $$\Sigma L_s \cap T^*(M \times (s_-^{-1} r_+, s_- r_+)) \cong \Lambda_- \times (s_-^{-1} r_+, s_- r_+).$$

    Similarly, for $\Sigma L_s \cap T^*M \times T^*(s_-^{-1} s_+ r_+, s_+ r_+)$, let $\gamma_1(\theta) = \phi^{-1}(s, \theta)$. We can connect $\gamma_1$ to the line segment $\gamma_0$. Since $\widetilde{L} \times \bR_{>0} \cap T^*(M \times (s_-^{-1}s_+, s_-s_+) \times (s_-^{-1} r_+, r_+))$ is the trivial suspension of the Legendrian cone $\Lambda_+ \times (s_-^{-1}s_+, s_-s_+) \times (s_-^{-1} r_+, r_+)$, we can apply Lemma \ref{lem:symp-reduce} again. Therefore, we can conclude the following lemma:

\begin{lemma}\label{lem:suspend-reduce1}
    Let ${L} \subset T^*(M \times (1, +\infty))$ be an exact Lagrangian cobordism from $\Lambda_-$ to $\Lambda_+$. Then for $1 < s < s_-$ sufficiently small, the symplectic reduction $\Sigma {L}_{s} \subset T^*(M \times \{s\} \times \bR_{>0})$ is Hamiltonian isotopic to the Lagrangian cobordism ${L} \subset T^*(M \times \bR_{>0})$ under a compactly supported Hamiltonian isotopy.
\end{lemma}

    For Condition~(2), consider the slice $T^*M \times \{s\} \times T^*\bR_{>0}$ for $s > s_+$ sufficiently large. Let $\gamma_1(\theta) = \phi^{-1}(s, \theta)$. Since $\widetilde{L} \times \bR_{>0} \cap T^*(M \times (s_+, +\infty) \times \bR_{>0})$ is the trivial suspension of the Legendrian cone $\Lambda_+ \times (s_+, +\infty) \times \bR_{>0}$, we can apply Lemma \ref{lem:symp-reduce} for the deformation induced by the isotopy from $\gamma_1$ to the line segment $\gamma_0$. We can easily tell that the symplectic reduction along $T^*M \times T^*_{\gamma_0}((s_+, +\infty) \times (r_-, r_+))$ is the trivial Legendrian cone, i.e.
    $$\Sigma L_s \cap T^*(M \times (r_-, r_+)) \cong \Lambda_+ \times (r_-, r_+).$$
    Therefore, we can conclude the following lemma:

\begin{lemma}\label{lem:suspend-reduce2}
    Let ${L} \subset T^*(M \times (1, +\infty))$ be an exact Lagrangian cobordism from $\Lambda_-$ to $\Lambda_+$. Then for $s > s_+$ sufficiently large, the symplectic reduction $\Sigma {L}_{s} \subset T^*(M \times \{s\} \times \bR_{>0})$ is Hamiltonian isotopic to the trivial Lagrangian cobordism $\Lambda_+ \times \bR_{>0} \subset T^*(M \times \bR_{>0})$ under a compactly supported Hamiltonian isotopy.
\end{lemma}

\subsection{Comparison between the two constructions}
    Recall that we have used the notation $\mu Sh_\Lambda$ for the sheaf of categories. However, for simplicity, in this subsection we will use them as the global sections $\mu Sh_\Lambda(\Lambda)$ of the corresponding sheaves of categories.

    Given the suspension Lagrangian cobordism $\Sigma{L} \subset T^*(M \times (1, +\infty) \times \mathbb{R}_{>0})$ from $\widetilde{L}$ to $\Lambda_+ \times \mathbb{R}_{>0}$, we will prove the following commutative diagram of global sections of categories
    \[\xymatrix@C=2mm{
    \mu Sh_{M \cup (\Lambda_- \times \mathbb{R}_{>0})} \times_{\mu Sh_{\Lambda_-}} \mu Sh_L \ar[d]_{\Phi_L} & \mu Sh_{(M \times \mathbb{R}_{>0}) \cup (\widetilde{L} \times \mathbb{R}_{>0})} \times_{\mu Sh_L} \mu Sh_{\Sigma L} \ar[d]_{\Phi_{\Sigma L}} \ar[l]_{\tilde{j}_-^{-1} \hspace{2pt}} \ar[r]^{\hspace{42pt} \tilde{j}_+^{-1}} & \mu Sh_{M \cup (\Lambda_+ \times \mathbb{R}_{>0})} \ar[d]_{\Phi_{\Lambda_+ \times \mathbb{R}_{>0}}}  \\
    \mu Sh_{M \cup (\Lambda_+ \times \mathbb{R}_{>0})} & \mu Sh_{(M \times \mathbb{R}_{>0}) \cup (\Lambda_+ \times \mathbb{R}_{>0}^2)} \ar[l]_{j_-^{-1} \hspace{5pt} } \ar[r]^{\hspace{16pt} j_+^{-1}} & \mu Sh_{M \cup (\Lambda_+ \times \mathbb{R}_{>0})},
    }\]
    where $\Phi_{\Lambda_+ \times \mathbb{R}_{>0}} \simeq \mathrm{id}$, and $j_-^{-1}$ and $j_+^{-1}$ are (obvious) equivalences with composition being the identity. Hence the diagram simplifies to
    \[\xymatrix{
    \mu Sh_{(M \times \mathbb{R}_{>0}) \cup (\widetilde{L} \times \mathbb{R}_{>0})} \times_{\mu Sh_L} \mu Sh_{\Sigma L} \ar[r]^{\hspace{12pt} \tilde{j}_-^{-1} } \ar[dr]_{\hspace{22pt} \tilde{j}_+^{-1}} & \mu Sh_{M \cup (\Lambda_- \times \mathbb{R}_{>0})} \times_{\mu Sh_{\Lambda_-}} \mu Sh_L \ar[d]^{\Phi_L} \\
    & \mu Sh_{M \cup (\Lambda_+ \times \mathbb{R}_{>0})}.
    }\]
    Using the fact that the cobordism $\widetilde{L}$ to $\Lambda_+ \times \mathbb{R}_{>0}$ diffeomorphic to ${L} \times \mathbb{R}$ and the negative end is equal to $\widetilde{L}$, we have a canonical identification of categories
    $$\mu Sh_{(M \times \mathbb{R}_{>0}) \cup (\widetilde{L} \times \mathbb{R}_{>0})} \times_{\mu Sh_L} \mu Sh_{\Sigma L} \simeq \mu Sh_{(M \times \mathbb{R}_{>0}) \cup (\widetilde{L} \times \mathbb{R}_{>0})}.$$
    Therefore, we will show that $\tilde{j}_-^{-1} \simeq (i_-^{-1}, m_{\widetilde{L}})$ and $\tilde{j}_+^{-1} \simeq i_+^{-1}$, and hence complete the proof of Theorem \ref{thm:compatible-cob}.

\begin{lemma}\label{lem:commute}
    Let $L$ be an exact Lagrangian cobordism from $\Lambda_-$ to $\Lambda_+$ and $\Sigma L$ the suspension exact Lagrangian cobordism. There is a commutative diagram of global sections of categories
    \[\xymatrix@C=2mm{
    \mu Sh_{M \cup (\Lambda_- \times \mathbb{R}_{>0})} \times_{\mu Sh_{\Lambda_-}} \mu Sh_L \ar[d]_{\Phi_L} & \mu Sh_{(M \times \mathbb{R}_{>0}) \cup (\widetilde{L} \times \mathbb{R}_{>0})} \times_{\mu Sh_L} \mu Sh_{\Sigma L} \ar[d]_{\Phi_{\Sigma L}} \ar[l]_{\tilde{j}_-^{-1} \hspace{2pt}} \ar[r]^{\hspace{42pt} \tilde{j}_+^{-1}} & \mu Sh_{M \cup (\Lambda_+ \times \mathbb{R}_{>0})} \ar[d]_{\Phi_{\Lambda_+ \times \mathbb{R}_{>0}}}  \\
    \mu Sh_{M \cup (\Lambda_+ \times \mathbb{R}_{>0})} & \mu Sh_{(M \times \mathbb{R}_{>0}) \cup (\Lambda_+ \times \mathbb{R}_{>0}^2)} \ar[l]_{j_-^{-1} \hspace{5pt} } \ar[r]^{\hspace{16pt} j_+^{-1}} & \mu Sh_{M \cup (\Lambda_+ \times \mathbb{R}_{>0})}.
    }\]
\end{lemma}
\begin{proof}
    Recall from Section \ref{sec:NadShen} the construction of the Lagrangian cobordism functor. Fix a contact embedding $T^*(M \times \bR_{>0}) \hookrightarrow T^{*,\infty}N$. Write
    $$N \times \{0\} \xrightarrow{{i}} N \times [0, 1] \xleftarrow{{j}} N (0, 1]$$
    and let $\phi_Z^\zeta,\,0 < \zeta \leq 1$ is the contact Hamiltonian flow that lifts the Liouville flow on $T^*(M \times \bR)$. For $\SF \in \mu Sh_{M \cup (\Lambda \times \bR_{>0})} \times_{\mu Sh_\Lambda} \mu Sh_L$, the image under the Lagrangian cobordism functor $\Phi_L$ is defined by
    $$\Phi_L(\SF)_\text{dbl} = i^{-1}j_*(\Psi_Z(\SF)_\text{dbl}),$$
    Then consider $\SF \in \mu Sh_{(M \times \bR_{>0}) \cup (\widetilde{L} \times \bR_{>0})} \times_{\mu Sh_L} \mu Sh_{\Sigma L}$. Write
    $$N \times \bR_{>0} \times \{0\} \xrightarrow{\overline{i}} N \times \bR_{>0} \times [0, 1] \xleftarrow{\overline{j}} N \times \bR_{>0} \times (0, 1]$$
    and let $\phi_{\bar{Z}}^\zeta,\,0 < \zeta \leq 1$ is the contact Hamiltonian flow that lifts the pull back Liouville flow on $T^*(M \times \bR \times \bR_{>0})$. We apply Proposition \ref{basechange} and Remark \ref{GKSviaCont} and get
    \[\begin{split}
    \Phi_L(\tilde{j}_-^{-1}\SF)_\text{dbl} &= i^{-1}j_*\big(\Psi_Z(\tilde{j}_-^{-1}\SF)_\text{dbl}\big) = i^{-1}j_*\big(\tilde{j}_-^{-1}\Psi_{\overline{Z}}(\SF)_\text{dbl}\big) \\
    &= i^{-1}\tilde{j}_-^{-1}\overline{j}_*\big(\Psi_{\overline{Z}}(\SF)_\text{dbl}\big) = j_-^{-1}\overline{i}^{-1}\overline{j}_*\big(\Psi_{\overline{Z}}(\SF)_\text{dbl}\big) = j_-^{-1}\Phi_{\Sigma L}(\SF)_\text{dbl}.
    \end{split}\]
    This prove the commutativity for the square on the left. Using the same argument, we have commutativity for the square on the right.
\end{proof}

    In the next lemma, we explain why there is an identification $\tilde{j}_+^{-1} \simeq i_+^{-1}$.

\begin{lemma}\label{lem:positive-end}
    Let $L$ be an exact Lagrangian cobordism from $\Lambda_-$ to $\Lambda_+$ and $\Sigma L$ the suspension exact Lagrangian cobordism. There is a commutative diagram of global sections of categories
    \[\xymatrix@R=3mm@C=6mm{
    \mu Sh_{(M \times \mathbb{R}_{>0}) \cup (\widetilde{L} \times \mathbb{R}_{>0})} \times_{\mu Sh_L} \mu Sh_{\Sigma L} \ar[dr]_-{ \tilde{j}_+^{-1} \hspace{12pt} } \ar[dd]_{\rotatebox{90}{$\sim$}} & \\
    & \mu Sh_{M \cup (\Lambda_+ \times \mathbb{R}_{>0})}.\\
    \mu Sh_{(M \times \mathbb{R}_{>0}) \cup (\widetilde{L} \times \mathbb{R}_{>0})} \ar[ur]_-{\hspace{0pt} i_+^{-1}} &
    }\]
\end{lemma}
\begin{proof}
    Using the property that the symplectic reduction $\Sigma L_s \cong \Lambda_+ \times \bR_{>0}$, there is a commutative diagram
    \[\xymatrix{
    \mu Sh_{(M \times \mathbb{R}_{>0}) \cup (\widetilde{L} \times \mathbb{R}_{>0})} \ar[d]^{i_+^{-1}} \ar[r] & \mu Sh_L \ar[d]^{i_+^{-1}} & \mu Sh_{\Sigma L} \ar[l] \ar[d]^{i_+^{-1}} \\
    \mu Sh_{M \cup (\Lambda_+ \times \mathbb{R}_{>0})} \ar[r] & \mu Sh_{\Lambda_+} & \mu Sh_{\Lambda_+ \times \bR_{>0}} \ar[l].
    }\]
    The restriction $\tilde{j}_+^{-1}$ is the restriction functor on the homotopy pull back which commutes with the above diagram
    $$\tilde{j}_+^{-1}: \mu Sh_{(M \times \mathbb{R}_{>0}) \cup (\widetilde{L} \times \mathbb{R}_{>0})} \times_{\mu Sh_L} \mu Sh_{\Sigma L} \rightarrow \mu Sh_{M \cup (\Lambda_+ \times \mathbb{R}_{>0})} \times_{\mu Sh_{\Lambda_+}} \mu Sh_{\Lambda_+ \times \bR_{>0}}.$$
    However, since $\Sigma L$ is diffeomorphic to $L \times \bR$, the restriction is an equivalence $\mu Sh_{\Sigma L} \xrightarrow{\sim} \mu Sh_L$. Thus there are equivalences of the homotopy pull back given by natural restriction functors
    \[\xymatrix{
    \mu Sh_{(M \times \mathbb{R}_{>0}) \cup (\widetilde{L} \times \mathbb{R}_{>0})} \times_{\mu Sh_L} \mu Sh_{\Sigma L} \ar[r]^-{\sim} \ar[d]_{\tilde{j}_+^{-1}} & \mu Sh_{(M \times \mathbb{R}_{>0}) \cup (\widetilde{L} \times \mathbb{R}_{>0})} \ar[d]^{i_+^{-1}} \\
    \mu Sh_{M \cup (\Lambda_+ \times \mathbb{R}_{>0})} \times_{\mu Sh_{\Lambda_+}} \mu Sh_{\Lambda_+ \times \bR_{>0}} \ar[r]^-{\sim} & \mu Sh_{M \cup (\Lambda_+ \times \mathbb{R}_{>0})}.
    }\]
    Therefore, the diagram in the statement commutes by the commutativity of the restriction functors.
\end{proof}

    In the final lemma, we explain why there is an identification between $\tilde{j}_-^{-1}$ and $(i_-^{-1}, m_{\widetilde{L}})$.

\begin{lemma}\label{lem:negative-end}
    Let $L$ be an exact Lagrangian cobordism from $\Lambda_-$ to $\Lambda_+$ and $\Sigma L$ the suspension exact Lagrangian cobordism. There is a commutative diagram of global sections of categories
    \[\xymatrix@R=3mm@C=6mm{
    \mu Sh_{(M \times \mathbb{R}_{>0}) \cup (\widetilde{L} \times \mathbb{R}_{>0})} \times_{\mu Sh_L} \mu Sh_{\Sigma L}  \ar[dr(0.6)]_-{ \tilde{j}_-^{-1} \hspace{12pt}} \ar[dd]_{\rotatebox{90}{$\sim$}} & \\
    &  \mu Sh_{M \cup (\Lambda_- \times \mathbb{R}_{>0})} \times_{\mu Sh_{\Lambda_-}} \mu Sh_L.\\
    \mu Sh_{(M \times \mathbb{R}_{>0}) \cup (\widetilde{L} \times \mathbb{R}_{>0})} \ar[ur]_-{ (i_-^{-1}, m_{\widetilde{L}})} &
    }\]
\end{lemma}
\begin{proof}
    Using the property that the symplectic reduction $\Sigma L_s \cong L$, there is a commutative diagram
    \[\xymatrix{
    \mu Sh_{(M \times \mathbb{R}_{>0}) \cup (\widetilde{L} \times \mathbb{R}_{>0})} \ar[d]^{i_-^{-1}} \ar[r]^-{m_{\widetilde{L}}} & \mu Sh_L \ar[d]^{i_-^{-1}} & \mu Sh_{\Sigma L} \ar[l] \ar[d]^{i_-^{-1}} \\
    \mu Sh_{M \cup (\Lambda_- \times \mathbb{R}_{>0})} \ar[r] & \mu Sh_{\Lambda_-} & \mu Sh_{L}. \ar[l]
    }\]
    The restriction $\tilde{j}_-^{-1}$ is the restriction functor on the homotopy pull back which commutes with the above diagram
    $$\tilde{j}_-^{-1}: \mu Sh_{(M \times \mathbb{R}_{>0}) \cup (\widetilde{L} \times \mathbb{R}_{>0})} \times_{\mu Sh_L} \mu Sh_{\Sigma L} \rightarrow \mu Sh_{M \cup (\Lambda_- \times \mathbb{R}_{>0})} \times_{\mu Sh_{\Lambda_-}} \mu Sh_{L}.$$
    However, since $\Sigma L$ is diffeomorphic to $L \times \bR$, the restriction is an equivalence $\mu Sh_{\Sigma L} \xrightarrow{\sim} \mu Sh_L$, and the composition determined by the trivial cobordism is the identity $\mu Sh_L \xleftarrow{\sim} \mu Sh_{\Sigma L} \xrightarrow{\sim} \mu Sh_L$. Therefore, the composition is the microlocalization $m_{\widetilde{L}}$
    $$\mu Sh_{(M \times \mathbb{R}_{>0}) \cup (\widetilde{L} \times \mathbb{R}_{>0})} \xrightarrow{m_{\widetilde{L}}} \mu Sh_L \xleftarrow{\sim} \mu Sh_{\Sigma L} \xrightarrow{\sim} \mu Sh_L.$$
    Therefore, the diagram in the statement commutes by the equivalence of homotopy pull back
    $$\mu Sh_{(M \times \mathbb{R}_{>0}) \cup (\widetilde{L} \times \mathbb{R}_{>0})}  \xrightarrow{\sim} \mu Sh_{(M \times \mathbb{R}_{>0}) \cup (\widetilde{L} \times \mathbb{R}_{>0})} \times_{\mu Sh_L} \mu Sh_{\Sigma L}$$
    given by natural restriction functor and commutativity of all the restriction functors.
\end{proof}

    Combining Lemmas \ref{lem:commute}, \ref{lem:positive-end}, \ref{lem:negative-end}, and the discussion at the beginning of this section, we can conclude Theorem \ref{thm:compatible-cob}. In particular, the two approaches of Lagrangian cobordism functor agree with each other.

\section{Examples and Applications}

\subsection{Computation of Lagrangian handle attachments}
    We provide some computations of examples of Lagrangian cobordisms. In the literature, the only known examples of embedded Lagrangian cobordisms are Legendrian isotopies \cite{GroEliashGF,Chantraine}, Lagrangian handle attachments \cite{EHK,Rizconnectsum} and Lagrangian caps \cite{LagCap}. 

    Lagrangian caps are exact Lagrangian cobordisms from loose or stabilized Legendrians to $\varnothing$, and thus we will not get non-trivial sheaves. Lagrangian cobordisms induced by Legendrian isotopies have been studied by Guillermou-Kashiwara-Schapira \cite{GKS,LiCobordism}. Therefore, we will only give computations of Lagrangian handle attachments (the reader may compare with \cite[Section 4.1]{LiCobordism}).

    Consider Lagrangian $k$-handle attachments on Legendrian $n$-submanifolds ($0 \leq k \leq n$). The local model of the front projection of $\widetilde{L}$ in $\bR^{n} \times \bR_t \times \bR_{>0,s}$ is shown in Figure \ref{fig:handle-attach}. The front projection of $\widetilde{L}$ determines a stratification on $\bR^{n} \times \bR_t \times \bR_{>0,s}$, which restricts to stratifications on $\bR^{n} \times \bR_t$, determined by the front projections of $\Lambda_\pm$.

\begin{figure}
  \centering
  \includegraphics[width=0.9\textwidth]{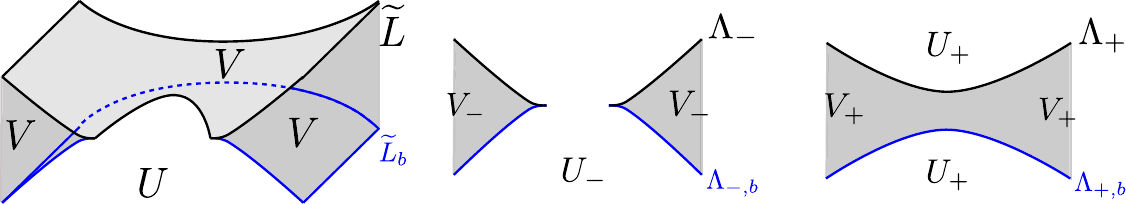}\\
  \caption{The front projection of the standard exact Lagrangian $1$-handle attachment on $1$-dimensional Legendrian links where the positive $s$-direction is pointing inward. The grey region is $V$ and the white region is $U$, and $V_\pm$ and $U_\pm$ are the intersections of $U$ and $V$ with a single horizontal slice $\bR^n \times \bR_t \times \{s_\pm\}$. Later, we will consider the the stratum in the front projection below the cusps in blue and denote it by $\widetilde{L}_b$ and $\Lambda_{\pm,b}$.}\label{fig:handle-attach}
\end{figure}

    For the stratification of $\bR^{n} \times \bR_t \times \bR_{>0,s}$, denote by $V$ the domain whose $t$-coordinate is bounded by the front projection the cobordism and $U$ the domain whose $t$-coordinate is unbounded on each vertical slice $\{(x_1, . . . , x_n)\} \times \bR_t \times \{s\}$. Restricting to $\bR^{n} \times \bR_t \times \{s_\pm\}$, we can get stratifications on the slices, and denote the restrictions of the regions by $V_\pm$ and $U_\pm$ of $V$ and $U$; see Figure \ref{fig:handle-attach}. For a sheaf in $\mathscr{F}_- \in Sh_{\Lambda_-}(\bR^{n} \times \bR_t)$, suppose the stalk in the region $V_-$ is $B$ and the stalk in $U_-$ is $A$. The microstalk of $\mathscr{F}$ is
    $$F = \mathrm{Cone}(A \to B).$$

    Consider the Lagrangian cobordism functor induced by $L$. First, we need to extend the microlocalization of $m_{\Lambda_-}(\mathscr{F}_-)$ from $\Lambda_-$ to $L$. Since $\Lambda_- = S^k \times D^{n-k}$ and ${L} = D^{k+1} \times D^{n-k}$, this is simply a extension of the local system on $S^k$ to $D^{k+1}$. Denote the local system on ${L}$ with stalk $F$ by $\mathscr{L}_L$, with $\mathscr{L}|_{\Lambda_-} = m_{\Lambda_-}(\mathscr{F}_-)$. This determines an object
    $$(\mathscr{F}_-, \mathscr{L}_L) \in Sh_{\Lambda_-}^b(\bR^{n+1}) \times_{Loc^b(\Lambda_-)} Loc^b(L).$$
    Here is how we determine the sheaf $\mathscr{F} \in Sh^b_{\widetilde{L}}(\bR^{n+1} \times \bR_{>0})$ that is the image of $(\mathscr{F}_-, \mathscr{L}_L)$. 

    (1) We determine the stalk of the sheaf $\mathscr{F}$. Since the region $U$ and $V$ are connected, we know that the stalk of $\mathscr{F}$ in $V$ has to be $B$, and the stalk of $\mathscr{F}$ in $U$ has to be $A$. 

    (2) We determine the local system of the sheaf $\mathscr{F}$ in the region $U$. Since $U$ admits a deformation retract to $U_-$, we can conclude that $\mathscr{F}|_U = A_U$. 

    (3) We determine the local system of the sheaf $\mathscr{F}$ in the region $V$. Since $V \cong D^k \times D^{n-k+2}$ and $V_- \cong S^k \times D^{n-k+1}$, we know that $V$ is not contractible relative to the negative boundary $V_-$, globally there could be nontrivial monodromy that comes from the local monodromy, parametrized by the fiber of restriction functor $Loc^b(V) \to Loc^b(V_-)$. Because there are transition maps
    $$A \to B \to A$$
    whose compositions are quasi-isomorphisms, without loss of generality, we assume that it is the identity \cite[Corollary 3.18]{STZ}. Then there is a splitting of chain complexes
    $$B \simeq A \oplus \mathrm{Cone}(A \to B) \simeq A \oplus F,$$
    and we can identify $\mathscr{F}|_V \simeq A_V \oplus \mathscr{L}_V$, where $\mathscr{L}_V$ is a local system on $V$ with stalk $F$.

    Consider the stratum of the front projection of $\widetilde{L}$ and $\Lambda_\pm$ below the family of cusps and denote it by $\widetilde{L}_b$ and $\Lambda_{\pm,b}$ as in Figure \ref{fig:handle-attach}. The inclusion $(L_b, \Lambda_{-,b}) \xrightarrow{\sim} (L, \Lambda_-)$ is a homotopy equivalence. We also know that the inclusion $(L_b, \Lambda_{-,b}) \xrightarrow{\sim} (V, V_-)$ induced by the front projection is a homotopy equivalence. By calculating the microlocalization, we can see that
    $$m_{\widetilde{L}_b}(\mathscr{F}) = \mathscr{L}_V|_{L_b} = \mathscr{L}_{L_b}$$
    relative to the boundary $\Lambda_{-,b} \xrightarrow{\sim} V_-$, meaning that they live in the same point in the fiber of the restriction functor. Thus, $\mathscr{L}_V$ is completely determined by $\mathscr{L}_L$.

    By Theorem \ref{thm:cond-quan-cob}, we know that the sheaf $\mathscr{F}$ is unquely determined by $(\mathscr{F}_-, \mathscr{L}_L)$. Therefore, the $\mathscr{F}$ we obtained has to be the image of $(\mathscr{F}_-, \mathscr{L}_L)$.

\subsection{Construction of Legendrians with simple sheaf categories}
    We provide one application of our framework. Using the sheaf quantization result in Section \ref{sec:quan-cob}, we show that the construction of Ekholm and Courte-Ekholm \cite{CourteEkholm} gives Legendrian submanifolds whose sheaf categories are equivalent to local systems on the Lagrangian fillings.

    Consider a Legendrian $\Lambda \subset J^1(M)$ with a Lagrangian filling $L$. We define a new Legendrian submanifold $\Lambda(L, L) \subset J^1(M \times \bR)$ diffeomorphic to $L \cup_\Lambda L$. Let $\widetilde{L} \subset J^1(M \times \bR_{>0})$ be the Legendrian lift conical outside $J^1(M \times (0, s_+))$. Let $\widetilde{L}^- \subset J^1(M \times \bR_{<0})$ the conical Legendrian under the map $\bR_{>0} \rightarrow \bR_{<0}, s \mapsto -s$. Consider (up to translations along $\bR_s$)
    \begin{align*}
    \Lambda(L, L) & \cap J^1(M \times (-\infty, -\epsilon)) = \widetilde{L} \cap J^1(M \times (0, s_+)), \\
    \Lambda(L, L) & \cap J^1(M \times (\epsilon, +\infty)) = \widetilde{L}^- \cap J^1(M \times (-s_+, 0)).
    \end{align*}
    Finally, consider an even function $\rho: (-2\epsilon, 2\epsilon) \rightarrow \bR$ such that $\rho|_{(-2\epsilon, -\epsilon)} = s - \epsilon + s_+$, $\rho|_{(\epsilon, 2\epsilon)} = -s + \epsilon + s_+$ and $\rho''(s) \leq 0$. Then we define
    $$\Lambda(L, L) \cap J^1(M \times (-2\epsilon, 2\epsilon)) = \{(x, s; \rho(s)\xi, \rho'(s)t, \rho(s)t) | (x, \xi, t) \in \Lambda\}.$$
    The union of the above three pieces then defines the closed Legendrian $\Lambda(L, L)$ (see \cite[Remark 4]{CourteEkholm} for the relation between their formalism and ours).

    Using the construction of the reverse suspension Lagrangian cobordism in Section \ref{sec:suspend} Remark \ref{rem:suspend-reverse}, we can prove that $\Lambda(L, L) \subset J^1(M \times \bR)$ has a exact Lagrangian filling $K(L, L) \subset J^1(M \times \bR) \times \bR_{>0} \cong T^*(M \times \bR \times \bR_{>0})$ as follows. We define (up to translations along $\bR_r$)
    \begin{align*}
    K(L, L) & \cap T^*(M \times (-\infty, -\epsilon) \times \bR_{>0}) = \overline{\Sigma L} \cap T^*(M \times (0, s_+) \times \bR_{>0}), \\
    K(L, L) & \cap T^*(M \times (\epsilon, +\infty) \times \bR_{>0}) = \overline{\Sigma L}^- \cap T^*(M \times (-s_+, 0) \times \bR_{>0}).
    \end{align*}
    By Lemma \ref{lem:suspend-cob}, this defines a Lagrangian filling of $\Lambda(L, L) \cap J^1(M \times (-\infty, -\epsilon))$ and $\Lambda(L, L) \cap J^1(M \times (\epsilon, +\infty))$. By Lemma \ref{lem:suspend-reduce1}, we know that the symplectic reduction along the hyperplanes $s = \pm \epsilon$ are Hamiltonian isotopic to the Lagrangian filling $L$ of $\Lambda$. Then define
    $$K(L, L) \cap T^*(M \times (-2\epsilon, 2\epsilon) \times \bR_{>0}) = \{(x, s, r; \rho(s)r\xi, \rho'(s)rt, \rho(s)t) | (x, s; \xi, t) \in L\}.$$
    The union of the above three pieces then defines the exact Lagrangian filling $K(L, L)$

    Given the Legendrian submanifold $\Lambda(L, L) \subset J^1(M \times \bR)$ with the exact Lagrangian filling $K(L, L) \subset J^1(M \times \bR) \times \bR_{>0}$, we prove the following result.

\begin{theorem}\label{thm:courteekholm}
    Let $L$ be an exact Lagrangian filling of $\Lambda \subset J^1(M)$. Then there exists an equivalence
    $$Sh_{\Lambda(L, L)}(M \times \bR)_0 \simeq Loc(L).$$
\end{theorem}

    Theorem \ref{thm:courteekholm-intro} will follow from Theorem \ref{thm:courteekholm} once we are able to realize any manifold with boundary $L$ such that $TL \otimes_\bR \bC$ is trivial as a Lagrangian with Legendrian boundary.

    For $n = 1$ this is clear. Let $n = 2$. We also know that any surface with boundary $(L, \partial L)$ has a trivial complexified tangent bundle and can arise as Lagrangian fillings of some Legendrian link in $(J^1(\bR) \times \bR_{>0}, J^1(\bR))$ \cite{EHK}. Let $n \geq 3$. Then by the $h$-principle of flexible Lagrangians of Eliashberg-Ganatra-Lazarev \cite[Theorem 4.5]{FlexLag}, $(L, \partial L)$ admits a (flexible) Lagrangian embedding with Legendrian boundary into $(D^{2n}, S^{2n-1})$ and hence by general position into $(J^1(\bR^{n-1}) \times \bR_{>0}, J^1(\bR^{n-1}))$. Therefore Theorem \ref{thm:courteekholm-intro} follows.

\begin{proof}[Proof of Theorem \ref{thm:courteekholm}]
    In the construction of Courte-Ekholm \cite{CourteEkholm}, we know that
    \begin{align*}
    \Lambda(L, L) & \cap J^1(M \times (-\infty, -\epsilon)) = \widetilde{L} \cap J^1(M \times (0, s_+)), \\
    \Lambda(L, L) & \cap J^1(M \times (\epsilon, +\infty)) = -\widetilde{L} \cap J^1(M \times (-s_+, 0)).
    \end{align*}
    Therefore, by Theorem \ref{thm:cond-quan-cob}, we know that $Sh_{\Lambda(L, L)}(M \times \bR \times (-\infty, -\epsilon))_0 \simeq Sh_{\Lambda(L, L)}(M \times \bR \times (\epsilon, +\infty))_0 \simeq Loc(L)$. We also know that
    $$\Lambda(L, L) \cap J^1(M \times (-2\epsilon, 2\epsilon)) \cong \Lambda \times (-2\epsilon, 2\epsilon)$$
    is a deformation of the trivial Legendrian cylinder. Therefore, we have $Sh_{\Lambda(L, L)}(M \times (-2\epsilon, 2\epsilon))_0 \simeq Sh_\Lambda(M \times \bR)$. By Proposition \ref{prop:restrict-ff}, we know that the restriction functors are fully faithful
    $$Sh_{\Lambda(L, L)}(M \times \bR \times (-\infty, -\epsilon))_0 \hookrightarrow Sh_\Lambda(M \times \bR)_0 \hookleftarrow Sh_{\Lambda(L, L)}(M \times \bR \times (\epsilon, +\infty))_0$$
    and the essential images coincide. Therefore, their (homotopy) pullback, i.e.~$Sh_{\Lambda(L, L)}(M \times \bR \times \bR)_0$, is the common essential image of the restriction functors, which is isomorphic to $Loc(L)$. This completes the proof.
\end{proof}

    More generally, for two different exact Lagrangian fillings $L_0$ and $L_1$ of $\Lambda$, we can consider the Legendrian $\Lambda(L_0, L_1)$. We still have the homotopy pullback formula for the category of sheaves
    \[\xymatrix{
    Sh_{\Lambda(L_0, L_1)}(M \times \bR \times \bR)_0 \ar[r] \ar[d] & Loc(L_1) \ar[d] \\
    Loc(L_0) \ar[r] & Sh_\Lambda(M \times \bR)_0.
    }\]
    If the essential images are disjoint, then the category of sheaves with singular support on $\Lambda(L_0, L_1)$ is trivial. For example, when $\Lambda$ is the Legendrian $m(9_{46})$-knot with maximal Thurston-Bennequin number, and $L_0, L_1$ are two different Lagrangian fillings \cite{ChantraineNonsym,ObstructConcord}, it is proved by Murphy \cite[Appendix]{CourteEkholm} that $\Lambda(L_0, L_1)$ is a loose Legendrian \cite{Loose}. Hence the category of sheaves is trivial.

    On the contrary, it is not always the case that the sheaf category is trivial. For instance, we now know a number of Legendrian links with different Lagrangian fillings such that the images of local systems on the different fillings have overlaps (for example, when the moduli space of rank 1 local systems on the fillings induced different cluster charts on the moduli space of microlocal rank 1 sheaves \cite{STWZ,CasalsGao,GaoShenWeng}). Then it is clear that the homotopy pull back category $Sh_{\Lambda(L_0, L_1)}(M \times \bR)_0$ is not zero.

\begin{remark}
    We remark the difference between the computation of sheaves, or equivalently computations of Legendrian contact homologies with loop space coefficients, and the computations of Legendrian contact homologies without loop space coefficients.

    When one works with Legendrian contact homology over $\bZ/2\bZ$ without loop space coefficients $\mathcal{A}(\Lambda)$, one can prove that $\mathcal{A}(\Lambda(L_0, L_1)) \simeq 0$ whenever the augmentations induced by $L_0, L_1$ are different if the Reeb chords of $\Lambda$ are supported in non-negative degrees. Let $c$ be a degree 0 Reeb chord on $\Lambda$ such that $\epsilon_{L_0}(c) \neq \epsilon_{L_1}(c)$. By assumption $\partial_\Lambda c = 0$. Note that a degree $d$ Reeb chord $c$ on $\Lambda$ lifts uniquely to a degree $(d+1)$ Reeb chord $\widetilde{c}$ on $\Lambda(L_0, L_1)$. Thus
    $$\partial_{\Lambda(L_0, L_1)}\widetilde{c} = \epsilon_{L_0}(c) + \epsilon_{L_1}(c) + \widetilde{\partial_\Lambda c} = \epsilon_{L_0}(c) + \epsilon_{L_1}(c) = 1$$
    using Morse flow trees. Hence $\mathcal{A}(\Lambda(L_0, L_1)) \simeq 0$.

    On the contrary, when working with Legendrian contact homology with loop space coefficients $\mathcal{A}_{C_{-*}(\Omega_*\Lambda)}(\Lambda)$, even over the field $\bZ/2\bZ$, there will be more restrictions for the homology to vanish. This is because to compute the differential now we will need to control the values of enhanced augmentations
    $$\Phi_{L_i}: \mathcal{A}_{C_{-*}(\Omega_*\Lambda)}(\Lambda) \rightarrow C_{-*}(\Omega_*L_i)$$
    and gluing the loop space coefficients by $C_{-*}(\Omega_*L_i) \rightarrow C_{-*}(\Omega_*\Lambda(L_0, L_1))$. Even though for Legendrian knots, the augmentation $\epsilon_{L_i}: \mathcal{A}_{C_{-*}(\Omega_*\Lambda)}(\Lambda) \rightarrow \Bbbk$ always sends the generator of $\bZ[\pi_1(\Lambda)]$ to $-1$ \cite{Leverson}, the enhanced augmentations $\Phi_{L_i}$ may send the generator to different values in $\bZ[\pi_1(L_i)]$ (for example, compare \cite[Section 8.1]{EHK} with \cite{PanTorus}).
\end{remark}

\bibliographystyle{amsplain}
\bibliography{legendriansheaf}
\end{document}